\documentclass[1 [leqno,11pt]{amsart}
\usepackage{amssymb, amsmath}
\allowdisplaybreaks[4]
\usepackage{color}
\usepackage{hyperref}
\usepackage{mathrsfs}
\usepackage{tikz}

\let\pa=\partial
\let\f=\frac

\let\p=\partial
\let\om=\omega
\def\lam{\lambda}

\def\s{\sigma}

\def\oo{\infty}

\def\cF{{\mathcal F}}

\def\cL{{\mathcal L}}
\def\cM{{\mathcal M}}
\def\cN{{\mathcal N}}

\def\eqdef{\buildrel\hbox{\footnotesize def}\over =}
\def\Z{\mathop{\mathbb Z\kern 0pt}\nolimits}
\def\N{\mathop{\mathbb N\kern 0pt}\nolimits}
\def\Q{\mathop{\mathbb Q\kern 0pt}\nolimits}
\def\R{{\mathop{\mathbb R\kern 0pt}\nolimits}}
\def\C{{\mathop{\mathbb C\kern 0pt}\nolimits}}
\def\T{{\mathop{\mathbb T\kern 0pt}\nolimits}}

\def\dive{{\mathop{\rm div}\nolimits}\,}

\def\ccL{{\mathscr{L}}}

\def\sgn{{\rm sgn}}

\def\tt{(1+t)}

\def\e{\varepsilon}
\def\eqdefa{\buildrel\hbox{\footnotesize def}\over =}

\newcommand{\Rmnum}[1]{\uppercase\expandafter{\romannumeral #1} }

\newcommand{\beq}{\begin{equation}}
\newcommand{\eeq}{\end{equation}}
\newcommand{\ben}{\begin{eqnarray}}
\newcommand{\een}{\end{eqnarray}}
\newcommand{\beno}{\begin{eqnarray*}}
\newcommand{\eeno}{\end{eqnarray*}}

 \numberwithin{equation}{section}
\newcommand{\andf}{\quad\hbox{and}\quad}
\newcommand{\with}{\quad\hbox{with}\quad}

\newtheorem{thm}{Theorem}[section]
\newtheorem{lem}{Lemma}[section]
\newtheorem{rmk}{Remark}[section]
\newtheorem{col}{Corollary}[section]
\newtheorem{prop}{Proposition}[section]

\begin{document}

\title[Long-time behavior in presence of Couette]{Long-time behaviour of two-dimensional Navier-Stokes equations in the presence of Couette flow on the half plane}

\author{Ning Liu}
\address[N. Liu]{Department of Mathematics, New York University Abu Dhabi, Saadiyat Island, P.O. Box 129188, Abu Dhabi, United Arab Emirates.}
\email{nl2983@nyu.edu}

\author{Nader Masmoudi}
\address[N. Masmoudi]{NYUAD Research Institute, New York University Abu Dhabi, Saadiyat Island, Abu Dhabi, P.O. Box 129188, United Arab Emirates.\ 
Courant Institute of Mathematical Sciences, New York University, 251 Mercer Street New York, NY 10012 USA}
\email{masmoudi@cims.nyu.edu}

\author{Weiren Zhao}
\address[W. Zhao]{Department of Mathematics, New York University Abu Dhabi, Saadiyat Island, P.O. Box 129188, Abu Dhabi, United Arab Emirates.}
\email{zjzjzwr@126.com, wz19@nyu.edu}

\date{\today}
\begin{abstract}
    In this paper, we study the long-time behavior of solutions to the two-dimensional Navier-Stokes equations in the presence of Couette flow on the half plane with Navier-slip boundary conditions. We prove that the total vorticity will approach 
    \begin{align*}
        -1+\frac{M_2(\omega_{0})}{\nu^{3/2}(1+t)^{5/2}}
\bar{\Omega}\left(
\frac{x}{\sqrt{\nu(1+t)^3}},
\frac{y}{\sqrt{\nu(1+t)}}
\right),
    \end{align*}
    where $-1$ is the vorticity of the Couette flow and $\bar{\Omega}$ is the kernel of a Fokker-Planck type operator $\cL=\p_Y^2+\frac32 X\p_X+\frac12 Y\p_Y+\frac52-Y\p_X$. In the proof, we introduce a new idea of studying the spectrum of such type operators with boundary. 
\end{abstract}

\maketitle
\section{Introduction}

In this paper, we study the two-dimensional incompressible Navier--Stokes equations on the half-plane
$
\R^2_+ \eqdefa \{(x,y)\in\R^2 \mid y\ge 0\},
$
namely
\begin{equation*}
(NS)\quad
\left\{
\begin{aligned}
&\pa_t v-\nu\Delta v + v\cdot\nabla v+\nabla p =0,
&& (t,x,y)\in \R^+\times \R^2_+,\\
&\dive v =0,
\end{aligned}
\right.
\end{equation*}
where $v(t,x,y)=(v^1,v^2)$ denotes the velocity field, $p$ is the pressure, and $\nu>0$ is the viscosity coefficient. We impose the Navier-slip boundary condition
\[
\p_y v^1|_{y=0}=1,
\qquad
v^2|_{y=0}=0,
\]
and study the long-time behavior of the solution with initial data in the presence of the Couette flow, namely, $v_{\mathrm{in}}=(y,0)+u_{\mathrm{in}}$. Since the Couette flow $(y,0)$ is a solution to $(NS)$, it is natural to introduce the perturbation $u=v-(y,0)$, where $u$ solves
\begin{equation}\label{eq:intro-u}
\left\{
\begin{aligned}
&\pa_t u-\nu\Delta u+y\p_x u+(u^2,0)+u\cdot\nabla u+\nabla p=0,
&& (t,x,y)\in\R^+\times\R^2_+,\\
&\dive u=0,\\
&\p_y u_1|_{y=0}=0,\qquad u_2|_{y=0}=0,\\
&u|_{t=0}=u_{0}(x,y).
\end{aligned}
\right.
\end{equation}
Let
$
\omega \eqdefa \p_x u^2-\p_y u^1
$
be the associated vorticity. Then $\omega$ satisfies
\begin{equation}\label{eqs:w}
\left\{
\begin{aligned}
&\pa_t \omega-\nu\Delta\omega+y\p_x\omega+u\cdot\nabla\omega=0,
&& (t,x,y)\in\R^+\times\R^2_+,\\
&u=(-\p_y,\p_x)\phi,\qquad \Delta\phi=\omega,\\
&\omega|_{y=0}=0,\qquad \phi|_{y=0}=0,\\
&\omega|_{t=0}=\omega_{0}(x,y).
\end{aligned}
\right.
\end{equation}

The global well-posedness of $(NS)$ for initial velocity in $L^2(\R^2)$ has been known since Leray's classical work \cite{leray1933etude}. In \cite{kato1984strong}, Kato proved that the $d$-dimensional Navier--Stokes equations are locally well-posed for arbitrary initial data in $L^d(\R^d)$, and globally well-posed for sufficiently small data. On the vorticity side, the Cauchy problem on $\R^2$ with initial vorticity in $L^1$ was studied in \cite{ben1994global,brezis1994remarks,kato1994navier}, where analogues of Leray's and Kato's results were established. Cottet \cite{cottet1986equations} and, independently, Giga--Miyakawa--Osada \cite{giga1988two} extended the existence theory to finite measure initial data in $\cM(\R^2)$. Uniqueness in $\cM(\R^2)$ was later proved by Gallagher and Gallay \cite{gallagher2005uniqueness}, building on the partial results in \cite{giga1988two,kato1994navier,gallay2005global,gallagher2005dirac}.

For the long-time behavior of $(NS)$ on $\R^2$, many works have shown that vorticity solutions with initial data in $L^1$ are asymptotically governed by the linear heat equation with the same total mass. This was first established by Giga and Kambe \cite{giga1988large}; a dynamical-systems approach was later developed in \cite{gallay2002invariant}. In \cite{carpio1994asymptotic}, Carpio observed a deep connection between this asymptotic behavior and the uniqueness of the fundamental solution for the vorticity equation. Finally, Gallay and Wayne \cite{gallay2005global} proved the asymptotic approximation for arbitrary initial data in $L^1$. More precisely, if $\omega_0\in L^1(\R^2)$, then the corresponding vorticity converges to the Oseen vortex with mass
$M=\int_{\R^2}\omega_0\,dx\,dy .$

For the Navier--Stokes system $(NS)$ on $\R^2_+$ with Navier boundary conditions, one may extend the vorticity oddly across the boundary by setting $\omega(t,x,-y)=-\omega(t,x,y)$. The extended function then solves the vorticity equation on the whole plane and satisfies $\omega|_{y=0}=0$. Consequently, the half-plane problem with Navier boundary condition remains globally well posed for initial data in $\cM$. Moreover, the refined asymptotic expansions on $\R^2$ obtained in \cite{gallay2005global} imply the corresponding long-time behavior on $\R^2_+$; see Remark~\ref{rmk1.3} below. We also refer to \cite{abe2021vorticity, dalibard2026viscousevolutionpointvortex, wang2026navier} and the references therein for related studies of $(NS)$ on the half plane with no-slip boundary conditions.

The Navier--Stokes equations possess the scaling symmetry
\begin{equation}\label{scaling NS}
v_\lambda(t,x,y)=\lambda v(\lambda^2 t,\lambda x,\lambda y),
\qquad
\omega_\lambda(t,x,y)=\lambda^2\omega(\lambda^2 t,\lambda x,\lambda y),
\end{equation}
which plays a fundamental role in both well-posedness and long-time asymptotics. The global spaces appearing above, such as $L^2$ for the velocity and $L^1$ or $\cM$ for the vorticity, are invariant under this scaling. Moreover, the proof of convergence toward the Oseen vortex relies essentially on self-similar variables induced by \eqref{scaling NS}, which reduces the asymptotic problem to the convergence toward the only steady state.

For the perturbative system near Couette flow \eqref{eqs:w}, there is no exact scaling symmetry compatible with all terms in \eqref{eqs:w}. Nevertheless, scaling considerations remain highly informative for both the local theory and the long-time behavior. To understand the short-time regime on an interval $[0,t_0]$, one applies the scaling \eqref{scaling NS} with $\lambda=\sqrt{t_0}\ll1$ and obtains that $\omega_\lambda$ solves the following equation on $[0,1]$: 
\begin{equation}\label{eq0.3}
\p_t\omega_\lambda-\nu\Delta \omega_\lam + \nabla^\perp\Delta^{-1}\omega_\lam\cdot \nabla \omega_\lam +\lam^2 y\p_x\omega_\lam=0.
\end{equation}
This suggests that the transport term $y\p_x\omega$ is perturbative at small time, so that \eqref{eqs:w} should remain locally well-posed in $L^1$.

For the long-time regime, the hypo-elliptic structure of the linearized equation around Couette flow leads to a different anisotropic scaling,
\begin{equation}\label{scaling Couette}
\omega_\lambda(t,x,y)=\lambda^a\omega(\lambda^2 t,\lambda^3 x,\lambda y).
\end{equation}
Dimensional analysis then gives, for $\lambda=\sqrt{t_0}\gg1$,
\[
\p_t\omega_\lam -\nu \p_y^2\omega_\lam+y\p_x\omega_\lam -\lam^{-4}\p_x^2\omega_\lam+\lam^{-a} \nabla^\perp (\p_y^2+\lam^{-2}\p_x^2)^{-1}\omega_\lam\cdot\nabla \omega_\lam=0.
\]
The exponent $a$ should be chosen so as to preserve the relevant conserved quantity. On $\R^2$, for instance, one takes $a=4$ to preserve the total mass $M[\omega]=M[\omega_\lambda]$. In general, the nonlinear term is also asymptotically small when $a>0$. This suggests that solutions of \eqref{eqs:w} should behave, for large time, like solutions of the linear equation
\[
\p_t\omega-\nu\p_y^2\omega+y\p_x\omega=0.
\]

In the recent work \cite{liu2026nonlinear}, the authors studied the long-time behavior of \eqref{eqs:w} on the whole plane $\R^2$ for general initial data. More precisely, using the conservation of the total mass
$
M(t)=\int_{\R^2}\omega(t,x,y)\,dx\,dy,$
they proved that \eqref{eqs:w} is globally well-posed for initial data in $L^1$. Moreover, the solution converges in $L^1$ to $M(\omega_0)G_L(t)$ at the optimal rate $t^{-1/2}$ as $t\to\infty$, where $G_L$ is the convolution kernel of the linearized equation
\begin{equation}\label{linear w}
\pa_t\omega_L-\nu\Delta\omega_L+y\p_x\omega_L=0.
\end{equation}

Near monotone shear flows such as the Couette flow, the dynamics are known to enjoy important stabilizing effects, including enhanced dissipation and inviscid damping, see \cite{albritton2022enhanced, BM2015, chen2025nonlinear,IJ2018, IJ2020, MasmoudiZhao2020, zhao2025inviscid} and references therein. In this sense, the flow is expected to be more stable, and this feature is also reflected in the long-time behavior obtained in \cite{liu2026nonlinear}. We refer the reader to that work for further discussion on the relation between such asymptotic behavior and other stability phenomena \cite{BVW2018, MasmoudiZhao2019, MasmoudiZhao2020cpde, arbon2024, li2025stability} near Couette flow.

On the other hand, in many physical experiments, Couette-type flows are produced by boundary forcing rather than posed in the whole space. This naturally leads us to investigate the problem in domains with boundaries. As a first step, the half-plane provides the simplest setting in which one can isolate and understand the effect of the boundary on the large-time dynamics.

The purpose of the present paper is to extend these results to \eqref{eqs:w} on the half-plane $\R^2_+$. Due to the Couette flow, the odd extension no longer reduces the half-plane problem to the whole-space one, because the term $y\p_x\omega$ changes parity in $y$. Another essential difference is that the conserved quantity on the half plane is no longer the total mass. Indeed, for the half-plane problem, the mass
$$
M(t)=\int_{\R^2_+}\omega(t,x,y)\,dx\,dy$$
is not conserved, whereas the first moment in the $y$-direction,
$$
M_2(t)=\int_{\R^2_+} y\,\omega(t,x,y)\,dx\,dy,$$
is conserved.

As in Proposition~2.1 of \cite{liu2026nonlinear}, one can adapt Kato's classical argument to obtain local well-posedness in $L^1$ for \eqref{eqs:w}. In this paper, we further prove that the $L^1$ norm of the solution is non-increasing in time, which yields global well-posedness. The main issue is then to determine the asymptotic profile and the convergence rate as $t\to\infty$. Our result shows that the total mass $M(t)$ on $\R^2_+$ always decays to zero at rate $t^{-1/2}$, while the solution itself is asymptotically described by the conserved quantity $M_2$ multiplied by a distinguished profile generated by the corresponding linear evolution.

Motivated by the long-time scaling \eqref{scaling Couette}, we introduce the self-similar variables
\begin{equation}\label{change of variable w:Om}
\begin{aligned}
\omega(t,x,y)
&=
\frac{1}{\nu^{3/2}(1+t)^{5/2}}
\Omega\!\left(
t,\frac{x}{\sqrt{\nu(1+t)^3}},\frac{y}{\sqrt{\nu(1+t)}}
\right),\\
X&=\frac{x}{\sqrt{\nu(1+t)^3}},
\qquad
Y=\frac{y}{\sqrt{\nu(1+t)}}.
\end{aligned}
\end{equation}
The factor $\nu^{-3/2}(1+t)^{-5/2}$ is chosen so that the $y$-momentum is preserved, namely
\[
M_2[\omega]=M_2[\Omega].
\]
In these variables, \eqref{eqs:w} becomes
\begin{equation}\label{eqs:Om}
\left\{
\begin{aligned}
&(1+t)\pa_t\Omega=\cL_t\Omega+\cN_t\Omega,
&& (t,X,Y)\in\R_+\times\R^2_+,\\
&\Omega|_{Y=0}=0,\\
&\Omega|_{t=0}=\Omega_0(X,Y),
\end{aligned}
\right.
\end{equation}
where
\[
\Omega_0(X,Y)=\nu^{3/2}\omega_0(\nu^{1/2}X,\nu^{1/2}Y),
\]
and where the operators $\Delta_t$, $\cL_t$, and $\cN_t$ are defined by
\begin{equation}
\begin{aligned}
\Delta_t&\eqdefa (1+t)^{-2}\p_X^2+\p_Y^2,\\
\cL_t&\eqdefa \Delta_t+\frac32 X\p_X+\frac12 Y\p_Y+\frac52-Y\p_X,\\
\cN_t\Omega&\eqdefa
\nu^{-3/2}(1+t)^{-5/2}
\bigl(
\p_Y\Delta_t^{-1}\Omega\,\p_X\Omega
-\p_X\Delta_t^{-1}\Omega\,\p_Y\Omega
\bigr).
\end{aligned}
\end{equation}
For simplicity, we write $f=\Delta_t^{-1}g$ for the solution of $\Delta_t f=g$ on $\R^2_+$ with Dirichlet boundary condition $f|_{Y=0}=0$.

As indicated by the above scaling argument, for large time $t\gg \nu^{-1}$ the nonlinear term $\cN_t\Omega$ and the horizontal diffusion term $(1+t)^{-2}\p_X^2$ are negligible because of the small factors $\nu^{-3/2}(1+t)^{-5/2}$ and $(1+t)^{-2}$. The evolution should therefore be close to the linear equation
\begin{equation}\label{eqs:F}
(1+t)\p_tF=\cL F,
\qquad
\cL=\p_Y^2+\frac32 X\p_X+\frac12 Y\p_Y+\frac52-Y\p_X,
\end{equation}
which corresponds, in the original variables, to
\[
\p_t f-\nu\p_y^2f+y\p_x f=0.
\]

Our first main result concerns the operator $\cL$ on the weighted space $L^2(m)$ defined by
\[
\|f\|_{L^2(m)}^2
\eqdefa
\int_{\R^2_+}|f(X,Y)|^2(1+X^2+Y^2)^m\,dX\,dY .
\]

\begin{thm}[Linear estimates]\label{thm1}
For $m>5$, there exists a unique function $\bar{\Omega}\in L^2(m)$ such that
\[
\mathcal{L}\bar{\Omega}=0,
\qquad
M_2(\bar{\Omega})=1.
\]
Moreover, $\p_X^{\alpha_1}\p_Y^{\alpha_2}\bar{\Omega}\in L^2(m)$ for all $\alpha_1,\alpha_2,m\in\N$.

There exists $C_{m}>0$ such that, for all $\tau>0$ and $F_0\in L^2(m)$,
\begin{equation}\label{eq:thm1}
\|e^{\tau\cL}F_0-M_2(F_0)\bar{\Omega}\|_{L^2(m)}
\le
C_{m}e^{-\tau}
\|F_0-M_2(F_0)\bar{\Omega}\|_{L^2(m)}.
\end{equation}
\end{thm}

Theorem~\ref{thm1} shows that, for any initial data $F_0\in L^2(m)$, the solution $F(\tau)$ of \eqref{eqs:F} converges to $M_2(F_0)\bar{\Omega}$ at the optimal rate $(1+t)^{-1}$ in the original time variable.

Our second main result states that the same asymptotic profile governs the nonlinear problem \eqref{eqs:Om}.

\begin{thm}[Nonlinear estimates]\label{Thm2}
Let $m>5$ and $\Omega_0\in L^2(m)$. For any $\epsilon>0$, the solution $\Omega(t)$ of \eqref{eqs:Om} satisfies
\begin{equation}\label{eq:thm2}
\|\Omega(t)-M_2(\Omega_0)\bar{\Omega}\|_{L^2(m)}
\le
C_{m,\epsilon}(1+t)^{-1}
\bigl(1+\nu^{-3/2}\|\Omega_0\|_{L^2(m)}\bigr)^{15m+14+\epsilon}
\|\Omega_0\|_{L^2(m)},
\end{equation}
for all
\[
t\ge
T_1=
C_{m,\epsilon}
\bigl(1+\nu^{-3/2}\|\Omega_0\|_{L^2(m)}\bigr)^{7m+7+\epsilon}.
\]
As a corollary, in the original variables, the solution $\omega(t)$ to \eqref{eqs:w} satisfies
\[
\left\|
\omega(t)-\f{M_2(\omega_0)}{\nu^{3/2}(1+t)^{5/2}}
\bar{\Omega}\!\left(
\frac{x}{\sqrt{\nu(1+t)^3}},
\frac{y}{\sqrt{\nu(1+t)}}
\right)
\right\|_{L^1}
=
O\bigl((1+t)^{-3/2}\bigr).
\]
\end{thm}

\begin{rmk}
The quantity $\nu^{-3/2}\|\Omega_0\|_{L^2(m)}$ seems different from the relative Reynolds number $\nu^{-1}\|\Omega_0\|_{L^2(m)}$ used in \cite{liu2026nonlinear}, but the two are equivalent up to the change of variables. Indeed, the relevant dimensionless quantity is
\[
\nu^{-1}\|\omega_0\|_{L^1}
=
\nu^{-3/2}\|\Omega_0\|_{L^1}
\le
C_m\nu^{-3/2}\|\Omega_0\|_{L^2(m)}.
\]
The exponent $3/2$ arises precisely from the normalization preserving $M_2$.
\end{rmk}

\begin{rmk}\label{rmk1.3}
We now compare the long-time behaviors of Navier-Stokes systems with and without the Couette flow on $\R^2$ and $\R^2_+$. Throughout, we write
\[
(M,M_1,M_2)=\int (1,x,y)\omega_0\,dx\,dy,
\]
where the integration is taken over the corresponding domain.

\begin{itemize}
\item $(NS)$ without the Couette flow on $\R^2$:
\[
\omega=
\frac{M}{\nu t}G\!\left(\frac{x}{\sqrt{\nu t}},\frac{y}{\sqrt{\nu t}}\right)
+\f{M_1x+M_2y}{2\nu^{2}t^{2}}
G\!\left(\frac{x}{\sqrt{\nu t}},\frac{y}{\sqrt{\nu t}}\right)
+O_{L^1}(t^{-1}).
\]
This asymptotic expansion is proved in \cite{gallay2005global}, where $G$ denotes the Gaussian profile.

\item $(NS)$ without the Couette flow on $\R^2_+$:
\[
\omega=
\f{M_2y}{\nu^{2}t^{2}}
G\!\left(\frac{x}{\sqrt{\nu t}},\frac{y}{\sqrt{\nu t}}\right)
+O_{L^1}(t^{-1}).
\]
By odd extension, the half-plane problem reduces to the whole-plane problem with $M=M_1=0$. We note that the remainder $O(t^{-1})$ is optimal, since the linear equation has a special solution of this order, namely
$
\f{xy}{\nu^{2}t^{3}}\,G\left(\frac{x}{\sqrt{\nu t}},\frac{y}{\sqrt{\nu t}}\right).$

\item $(NS)$ with the Couette flow on $\R^2$:
\[
\omega=
\f{M}{\nu t^{2}}G_L(\f{x}{\sqrt{\nu t^3}},\f{y}{\sqrt{\nu t}})
-
\f{M_2}{\nu^{3/2} t^{5/2}}(\p_X+\p_Y)G_L(\f{x}{\sqrt{\nu t^3}},\f{y}{\sqrt{\nu t}})
+
O_{L^1}(t^{-1}).
\]
Here
$
G_L(X,Y)=\frac{\sqrt{3}}{2\pi}\exp(-3X^2+3XY-Y^2).
$
This expression differs from that in \cite{liu2026nonlinear} only because of the choice of coordinates. The second-order term involving $M_2$ is not stated explicitly in \cite{liu2026nonlinear}, but it follows from a minor extension of their argument.

\item Our result for $(NS)$ with the Couette flow on $\R^2_+$:
\[
\omega=
\f{M_2}{\nu^{3/2}t^{5/2}}\bar{\Omega}(\f{x}{\sqrt{\nu t^3}},\f{y}{\sqrt{\nu t}})
+
O_{L^1}(t^{-3/2}).
\]
\end{itemize}

These comparisons show that, for the cases on $\R^2_+$, the leading term is analogous to the second-order term in the whole-space asymptotics, because the principal whole-space profiles do not satisfy the Dirichlet boundary condition. However, there is a crucial difference: for $(NS)$ without the Couette flow on $\R^2_+$, the leading asymptotic profile is simply the restriction of a whole-space profile satisfying the boundary condition, whereas with the Couette flow, none of the whole-space asymptotic terms satisfy the boundary condition. As a consequence, our profile $\bar{\Omega}$ is genuinely new and substantially more difficult to identify.

Also, we remark that the error $O(t^{-3/2})$ in our result is smaller than all the errors in the other cases.
\end{rmk}

We next explain the strategy of the proof and highlight the main new ingredients.

As discussed above, the asymptotic behavior of \eqref{eqs:w} on $\R^2_+$ cannot be deduced from the whole-space analysis in \cite{liu2026nonlinear}. The first main difficulty is to identify the kernel $\bar{\Omega}$ of $\cL$ and to prove Theorem~\ref{thm1}. We also note that, in the study of $\cL$ on $\R^2$, the authors of \cite{liu2026nonlinear} introduced a special change of variables reducing the operator to the standard Fokker--Planck operator. Because of the boundary, such a reduction is no longer available in the half-plane. To overcome this difficulty, we combine two complementary points of view.

By studying the Airy operator, we obtain the solution to the linearized problem. A formal approximation strongly suggests the following profile, see more details in section \ref{section 2}: 
\[
\omega(t,x,y)\approx M_2\nu^{-3/2}t^{-5/2}\bar{\Omega}(X,Y),
\]
which is precisely why we adopt the variables \eqref{change of variable w:Om}. To obtain the semi-group estimate, we use the resolvent estimate and a new representation formula with suitable decompositions, see \eqref{eq2.8}-\eqref{eq2.10} and \eqref{eq2.15}. 

\begin{prop}\label{prop1.1}
For $m>5$, there exists $C_{m}>0$ such that, for every $F_0\in L^2(m)$,
\begin{equation}\label{eq2.4}
\|e^{\tau\cL}F_0\|_{L^2(m)}
\le
C_{m}\|F_0\|_{L^2(m)}.
\end{equation}
Moreover, if we assume in addition
$
\int_{\R^2_+}YF_0\,dX\,dY=0,
$
then
\begin{equation}\label{eq2.5a}
\|e^{\tau\cL}F_0\|_{L^2(m)}
\le
C_{m}e^{-\tau}\|F_0\|_{L^2(m)}.
\end{equation}
\end{prop}

At this point, once the semi-group bounds are available, it remains to prove the existence of a steady state $\bar{\Omega}\in L^2(m)$ satisfying $\cL\bar{\Omega}=0$. Our second point of view is to solve this steady equation directly in the self-similar variables. By extending the solution centrally symmetrically, we apply the Fourier transform in both $X$ and $Y$ directions, and the equation becomes
\begin{equation}\notag
\frac32 k\p_k f+\left(\frac12\eta-k\right)\p_\eta f+(\eta^2-\tfrac12)f=g(k),
\end{equation}
together with the normalization $\p_\eta  f(0,0)=-2i$ and the constraint 
\[
\int_\R f(k,\eta)\,d\eta=0,\quad \forall k\in \mathbb{R}
\]
Solving along the characteristic curves gives the representation
\begin{equation}\notag
\begin{aligned}
f(k,\eta)
&=
-2i(k+\eta)e^{-(k+\eta)^2+k(k+\eta)-\frac13k^2}+\frac23\int_0^k\left(\frac{k}{\ell}\right)^{1/3}
\exp\Bigl(
-(k+\eta)^2\bigl(1-(\tfrac{\ell}{k})^{2/3}\bigr)\\
&\quad\qquad\qquad\qquad
+k(k+\eta)\bigl(1-(\tfrac{\ell}{k})^{4/3}\bigr)
-\tfrac13(k^2-\ell^2)
\Bigr)
\frac{g(\ell)}{\ell}\,d\ell,
\end{aligned}
\end{equation}
where $g$ satisfies $g(0)=0$ and solves the convolution equation
\begin{equation}\notag
\int_0^k
(k^{2/3}-\ell^{2/3})^{-1/2}
e^{-\frac1{12}(k^{2/3}-\ell^{2/3})^3}
g(\ell)\ell^{-4/3}\,d\ell
=
\frac{3i}{2}k^{1/3}e^{-k^2/12}.
\end{equation}
Applying the Laplace transform once more allows us to solve for $g$, and hence to define $\bar{\Omega}$. A careful analysis then shows that the resulting profile indeed belongs to $L^2(m)$. Combined with Proposition~\ref{prop1.1}, this completes the proof of Theorem~\ref{thm1}.

\begin{rmk}

The explicit kernel $\bar{\Omega}$ obtained from the characteristic method does not look identical to the formal limit suggested by \eqref{eq1.4}. Since the two constructions should in fact coincide, we are led to the following conjecture.

\medskip
\noindent\textbf{Conjecture.}
We conjecture that the kernel $\bar{\Omega}$ of $\mathcal{L}$ has the following explicit representation via Airy functions
\[
\frac1{2\pi}\int_\R e^{i\ell X}|\ell|^{2/3}\sum_{n=1}^\infty
A_n^{-2}
e^{e^{i\frac{\pi}{3}\sgn(\ell)}\xi_n|\ell|^{2/3}}
Ai\!\left(
e^{i\frac{\pi}{6}\sgn(\ell)}|\ell|^{1/3}Y+\xi_n
\right)
\overline{Ai'(\xi_n)}\,d\ell
\]
To prove it, one may need to justify the approximation below \eqref{eq1.4} mathematically rigorously or to prove that the above function is in $L^2(m)$. 

We note that each individual term 
$$f_n(X,Y)=\frac1{2\pi}\int_\R e^{i\ell X}|\ell|^{2/3}
e^{e^{i\frac{\pi}{3}\sgn(\ell)}\xi_n|\ell|^{2/3}}
Ai\!\left(
e^{i\frac{\pi}{6}\sgn(\ell)}|\ell|^{1/3}Y+\xi_n
\right)\,d\ell$$ 
in the above series solves $\cL f_n=0$, but does not even belong to $L^1$. Therefore, proving the conjecture would require identifying very delicate cancellations inside the infinite sum.

\end{rmk}

Once Theorem~\ref{thm1} is established, the proof of Theorem~\ref{Thm2} follows the same general strategy as in \cite{liu2026nonlinear}. We first prove that the $L^1$ norm of $\omega(t)$ is non-increasing and then derive decay estimates for the $L^p$ norms by Moser's method. Next, we estimate the $L^2(m)$ norm of $\Omega(t)$ and exploit hypo-ellipticity to gain higher regularity. Finally, Duhamel's formula together with the semi-group estimate \eqref{eq:thm1} yields the nonlinear asymptotic stability stated in Theorem~\ref{Thm2}.

The remainder of the paper is organized as follows.

In Section~\ref{section 2}, we first discuss the computations by using the complex Airy operator, and then derive the formula for the semi-group $e^{\tau\cL}$ by a resolvent approach and prove Proposition~\ref{prop1.1}.

In Section~\ref{section 3}, we solve the eigenvalue problem $\cL\bar{\Omega}=0$ and prove the existence of $\bar{\Omega}$, thereby completing the proof of Theorem~\ref{thm1}.

In Section~\ref{section 4}, we establish several estimates in the weighted spaces $L^2(m)$.

In Section~\ref{section 5}, we prove Theorem~\ref{Thm2}.

We end this introduction with a few notational conventions used throughout the paper.

\noindent\textbf{Notations.}
For $a\lesssim b$, we mean that there exists a positive constant $C$, independent of the relevant parameters, such that $a\le Cb$. In estimates involving the weighted space $L^2(m)$, we use
\[
a(X,Y)=\sqrt{|X|^2+|Y|^2},
\qquad
b(X,Y)=\langle X,Y\rangle=\sqrt{1+|X|^2+|Y|^2}.
\]
Finally, we write $L_T^r(L^p)$ for the space $L^r([0,T];L^p(\R_{+}^2))$, and $L_{[T_1,T_2]}^r(L^p)$ for the space $L^r([T_1,T_2];L^p(\R_+^2))$.

\section{The semi-group associated to the linear operator}\label{section 2}
In this section, we study the semi-group $e^{t\mathcal{L}}$ generated by the linearized operator $\mathcal{L}$ defined in \eqref{eqs:F}. We remark that by a suitable change of coordinates, the operator $\mathcal{L}$ can be expressed as a Fokker-Planck type operator:
\begin{align*}
    \cL_{FK}\eqdefa 4\p_{\tilde Y}^2-\f{\tilde{Y}^2}{4}-\f{\sqrt{3}}2(\tilde{X}\p_{\tilde{Y}}-\tilde{Y}\p_{\tilde{X}})+C
\end{align*}
However, such a change of coordinates is not well-adapted to the boundary condition. 
\subsection{A formal computation by the eigen-system of the Airy operator}
In this subsection, we present the detailed derivation of the solution formula for the linearized equation and the formal limit.

Applying the tangential Fourier transform $\omega_k(t,y)\eqdefa \int_\R \omega(t,x,y) e^{-ikx}dx$, we deduce from \eqref{linear w} that
\begin{equation}\label{linear w k}
\pa_t \om_k-\nu(\p_y^2-k^2) \omega_k +ik y \omega_k=0,
\end{equation}
with Dirichlet boundary condition and initial data $\omega_{0,k}$.

For any fixed $0\neq k\in \mathbb{R}$, we denote the following complex Airy operator
$$
\cL_{\nu,k}\eqdefa \nu(\p_y^2-k^2)-iky
$$
and recall the following properties for $\nu$ and $k$ given:
\begin{itemize}
\item Let $e_n(y)=C_n Ai(e^{i\f{\pi}6 \sgn(k)}(\f{|k|}{\nu})^\f13 y+\xi_n)$ and $\lambda_n=-\nu|k|^2 +e^{i\f{\pi}3 \sgn(k)}\nu^\f13|k|^\f23 \xi_n$, where $C_n$ is a constant to be chosen, $Ai(z)$ is the complex Airy function solving $\p_z^2 Ai(z)=zAi(z)$, and $\xi_n\approx -n^\f23<0$ is the n-th zero point of $Ai(z)$. (All $\xi_n\in \R$ and $0>\xi_1>\xi_2>\cdots$)
By direct computations, we can check that 
$$
\cL_{\nu,k} e_n =\lambda_n e_n \andf e_n(0)=0.
$$

\item Similarly, considering the adjoint operator $\cL_{\nu,k}^*\eqdefa \nu(\p_y^2-k^2)+iky$, one can derive similar eigenvalues by changing $k$ into $-k$. Let $e_n^*(y)=C_n Ai(e^{-i\f{\pi}6 \sgn(k)}(\f{|k|}{\nu})^\f13 y+\xi_n)$ which will satisfy
$$
\cL_{\nu,k}^* e_n^* =\overline{\lambda_n} e_n^* \andf e_n^*(0)=0.
$$

\item The eigen-functions $e_n$ and $e_n^*$ are orthogonal in the following sense:
$$
( e_m, e_n^* )_{L^2_y} =0 \qquad \text{for} \qquad m\neq n
$$
We choose $C_n$ to be such that
$
1= ( e_n, e_n^* )_{L^2_y}.
$
Then, the orthogonality becomes
$$
( e_m, e_n^* )_{L^2_y} =\delta_{m,n}.
$$
From a change of variable, we know $C_n^2= A_n^{-2}(\f{|k|}{\nu})^\f13$ where $A_n$ denotes
$$
A_n^2\eqdefa \int Ai(e^{i\f{\pi}6 \sgn(k)}y+\xi_n)\overline{Ai(e^{-i\f{\pi}6 \sgn(k)}y+\xi_n)}dy.
$$

\item The eigen-functions $e^*_n$ form a complete basis of $L^2(\R_+)$, and for any function $f\in L^2(\R_+)$, we can decompose:
$$
f(y)=\sum_{n=1}^\oo ( f, e_n^* )_{L^2_y} e_n(y).
$$
The summation above holds in the sense of Abel summation of order $\beta$ for $1<\beta<\f43$, which means that for almost every $y>0$,
$$
f(y) = \lim_{x\rightarrow 1-}  \sum_{n=1}^\oo x^{n^\beta}( f, e_n^* )_{L^2_y} e_n(y).
$$
\end{itemize}

The above properties of the complex Airy operator on the half line can be found in some books, like section 14.5 of \cite{helffer2013spectral}. Also, we cite \cite{savchuk2017spectral} and the references within for the recent study of general complex Airy operators.

By applying the above points, we can solve \eqref{linear w k} as
\begin{equation}
\omega_k(t) = \sum_{n=1}^\oo e^{\lambda_n t} ( \omega_{0,k}, e_n^* )_{L^2_y} e_n(y).
\end{equation}
For the summation over $n$, it makes sense for any $t>0$, since $\Re\lambda_n \leq -c\nu^\f13|k|^\f23 n^{\f23}$. This expression converges to the initial data $\omega_{0,k}$ in a very weak sense as $t\rightarrow0+$,  since the Abel summation of order $\beta>1$ is not enough. 

Then, by taking the tangential Fourier inverse transform, we get the solution formula for \eqref{linear w}:
\begin{equation}\label{eq1.3}
\begin{aligned}
\omega (t,x,y) &= \f1{2\pi}\int_\R\sum_{n=1}^\oo e^{\lambda_n t} ( \omega_{0,k}, e_n^* )_{L^2_y} e_n(y) e^{ikx}dk\\
&= \f1{2\pi}\int_\R\sum_{n=1}^\oo A_n^{-2}(\f{|k|}{\nu})^\f13e^{(-\nu|k|^2 +e^{i\f{\pi}3 \sgn(k)}\nu^\f13|k|^\f23 \xi_n) t}  Ai(e^{i\f{\pi}6 \sgn(k)}(\f{|k|}{\nu})^\f13 y+\xi_n)\\
&\qquad\qquad\qquad\times \int_0^\oo \omega_{0,k}(z) \overline{ Ai(e^{-i\f{\pi}6 \sgn(k)}(\f{|k|}{\nu})^\f13 z+\xi_n)}dz\ e^{ikx}dk.
\end{aligned}
\end{equation}
To catch the main profile, we introduce a new variable $\ell=\nu^\f12 t^\f32 k$ which satisfies
$$
(\f{|k|}{\nu})^\f13 = \f{|\ell|^\f13}{\sqrt{\nu t}}, \qquad \nu|k|^2t=\f{|\ell|^2}{t^2} \andf \nu^\f13|k|^\f23t=|\ell|^\f23.
$$
After the transform, we derive from \eqref{eq1.3} that
\begin{equation}\label{eq1.4}
\begin{aligned}
\omega (t,x,y) &= \f1{2\pi\nu t^2}\int_\R |\ell|^\f13\sum_{n=1}^\oo A_n^{-2} e^{-\f{|\ell|^2}{t^2}+e^{i\f{\pi}3 \sgn(\ell)}\xi_n|\ell|^\f23}  Ai(e^{i\f{\pi}6 \sgn(\ell)}|\ell|^\f13 \f{y}{\sqrt{\nu t}}+\xi_n)\\
&\qquad\qquad\qquad\qquad\times\int_0^\oo \omega_{0,\f{\ell}{\sqrt{\nu t^3}}}(z) \overline{ Ai(e^{-i\f{\pi}6 \sgn(\ell)}|\ell|^\f13 \f{z}{\sqrt{\nu t}}+\xi_n)}dz\ e^{i\ell \f{x}{\sqrt{\nu t^3}}}d\ell,
\end{aligned}
\end{equation}
which gives the solution formula.

Next, we introduce a formal analysis to obtain the main limit in the above formula as $t$ goes to infinity. Given $\ell$ and $z$, there holds for large enough time $t$ 
$$
\overline{ Ai(e^{-i\f{\pi}6 \sgn(\ell)}|\ell|^\f13 \f{z}{\sqrt{\nu t}}+\xi_n)} \approx \overline{Ai'(\xi_n)} e^{i\f{\pi}6 \sgn(\ell)}|\ell|^\f13 \f{z}{\sqrt{\nu t}},
$$
and therefore,
\begin{align*}
\int_0^\oo \omega_{0,\f{\ell}{\sqrt{\nu t^3}}}(z) \overline{ Ai(e^{-i\f{\pi}6 \sgn(\ell)}|\ell|^\f13 \f{z}{\sqrt{\nu t}}+\xi_n)}dz
&\approx \overline{Ai'(\xi_n)} e^{i\f{\pi}6 \sgn(\ell)} \f{|\ell|^\f13}{\sqrt{\nu t}}\int_0^\oo z\,\omega_{0,0}(z)dz \\
&\approx M_2[\omega_0] \overline{Ai'(\xi_n)} e^{i\f{\pi}6 \sgn(\ell)} \f{|\ell|^\f13}{\sqrt{\nu t}}.
\end{align*}

Formally, for the very large time $t$, the solution formula \eqref{eq1.4} becomes
\begin{equation}\notag
\begin{aligned}
\omega (t,x,y) &\approx \f{M_2[\omega_0]}{2\pi\nu^\f32 t^\f52}\int_\R e^{i\ell \f{x}{\sqrt{\nu t^3}}}|\ell|^\f23\sum_{n=1}^\oo A_n^{-2} e^{e^{i\f{\pi}3 \sgn(\ell)}\xi_n|\ell|^\f23}  Ai(e^{i\f{\pi}6 \sgn(\ell)}|\ell|^\f13 \f{y}{\sqrt{\nu t}}+\xi_n)\overline{Ai'(\xi_n)} d\ell.
\end{aligned}
\end{equation} 

\subsection{The Derivation of the formula}

In this subsection, we study the evolution of the following linear equation on the half plane:
\begin{equation}\label{eq2.5}
    \left\{
\begin{aligned}
&\pa_t f-\p_y^2 f+y\p_x f=0,
&& (t,x,y)\in\R^+\times\R^2_+,\\
&f|_{y=0}=0,\\
&f|_{t=0}=f_{0}(x,y),
\end{aligned}
\right.
\end{equation}
which is equivalent to study $(1+t)\p_t F=\cL F$ with $F$ given by
\begin{equation}\label{eq2.6}
    f(t,x,y)=(1+t)^{-\f52} F(t,(1+t)^{-\f32}{x},(1+t)^{-\f12}y).
\end{equation}
    
Inspired by the scaling property, we first introduce the following Fourier transform in the $x$ direction as
$$
    \hat{f}(t,k,y)\eqdefa \int_\R f(t,x,y) e^{-ikx}dx,
$$
and then introduce the Fourier-Laplace transform in the $t$ direction as
$$
    f_{k,\lambda}(y) \eqdefa \int_0^\oo \hat{f}(t,k,y) e^{-ik^{\f23}\lambda t}dt.
$$
Then, $f_{k,\lambda}$ solves the following second order ODE in $y$ direction:
\begin{equation}\label{eq2.7}
    i(k^\f23\lambda+ky)f_{k,\lambda}(y)-\p_y^2 f_{k,\lambda}(y)=\hat{f_0}(k,y).
\end{equation}

To solve the equation \eqref{eq2.7}, we observe that the homogeneous equation has solutions expressed by the complex Airy function:
$$
    [i(k^\f23\lambda+ky)-\p_y^2] Ai\bigl(e^{i\f{\pi}6}(\lambda+k^\f13y)\bigr)=[i(k^\f23\lambda+ky)-\p_y^2] Ai\bigl(e^{i\f{5\pi}6}(\lambda+k^\f13y)\bigr)=0.
$$
Given any $\lambda\in \C$, we have $\lim\limits_{y\rightarrow\oo} Ai\bigl(e^{i\f{\pi}6}(\lambda+k^\f13y)\bigr)=0$ and $\lim\limits_{y\rightarrow\oo}Ai\bigl(e^{i\f{5\pi}6}(\lambda+k^\f13y)\bigr)=\oo$. When $\lambda$ is chosen such that $Ai(e^{i\f\pi6}\lambda)\neq 0$ (i.e. $\lambda\notin \{ e^{-i\f\pi6} \xi_n \}$), we define the following two solutions of the homogeneous equation:
$$
    A_{k,\lambda}\eqdefa Ai\bigl(e^{i\f{\pi}6}(\lambda+k^\f13y)\bigr),\qquad
    B_{k,\lambda}\eqdefa \f{Ai(e^{i\f{5\pi}6}\lambda)}{Ai(e^{i\f{\pi}6}\lambda)}Ai\bigl(e^{i\f{\pi}6}(\lambda+k^\f13y)\bigr)-Ai\bigl(e^{i\f{5\pi}6}(\lambda+k^\f13y)\bigr),
$$
which satisfy $\lim\limits_{y\rightarrow\oo} A_{k,\lambda}(y)=0$ and $B_{k,\lambda}(0)=0$. The Wronskian between $A_{k,\lambda}$ and $B_{k,\lambda}$ can be computed out as
\begin{align*}
    W_{k,\lambda}&\eqdefa A_{k,\lambda}'(0)B_{k,\lambda}(0)-A_{k,\lambda}(0)B_{k,\lambda}'(0)\\
    &=-e^{i\f\pi6}k^\f13Ai(e^{i\f{5\pi}6}\lambda)Ai'\bigl(e^{i\f{\pi}6}\lambda)+e^{i\f{5\pi}6}k^\f13 Ai\bigl(e^{i\f{\pi}6}\lambda)Ai'(e^{i\f{5\pi}6}\lambda)\\
    &=(e^{i\f{5\pi}6}-e^{i\f\pi6})k^\f13 Ai(0)Ai'(0)\\
    &=(2\pi)^{-1} k^\f13.
\end{align*}
Therefore, when $\lambda\notin \{ e^{-i\f\pi6} \xi_n \}$, the solution of \eqref{eq2.7} with boundary condition $f_{k,\lambda}(0)=\lim_{y\rightarrow\oo}f_{k,\lambda}(y)=0$ is given by
\begin{equation}\label{eq2.8}
    \begin{aligned}
        f_{k,\lambda}= &-\f1{W_{k,\lambda}}\Bigl(A_{k,\lambda}(y)\int_0^y B_{k,\lambda}(z)\hat{f_0}(k,z)dz+B_{k,\lambda}(y)\int_y^\oo A_{k,\lambda}(z)\hat{f_0}(k,z)dz\Bigr)\\
        =&- 2\pi k^{-\f13}\Bigl(A_{k,\lambda}(y)\int_0^y B_{k,\lambda}(z)\hat{f_0}(k,z)dz+B_{k,\lambda}(y)\int_y^\oo A_{k,\lambda}(z)\hat{f_0}(k,z)dz\Bigr).
    \end{aligned}
\end{equation}

From the inverse Laplace transform on $\lambda\in \R$, we already arrive at the following solution formula for \eqref{eq2.5} in the variable $(k,y)$:
\begin{equation}\label{eq2.9}
    \hat{f}(t,k,y)=\f{k^\f23}{2\pi} \int_\R f_{k,\lambda}(y)e^{ik^\f23\lambda t}d\lambda.
\end{equation}

Next, we transform the formula into a form that will be more convenient for subsequent use. By taking \eqref{eq2.8} into \eqref{eq2.9} and decompose the two parts in $B_{k,\lambda}$, we find 
\begin{equation}\label{eq2.10}
    \begin{aligned}
        \hat{f}(t,k,y)=&-k^\f13 \int_\R e^{ik^\f23\lambda t}\f{Ai(e^{i\f{5\pi}6}\lambda)}{Ai(e^{i\f{\pi}6}\lambda)} Ai\bigl(e^{i\f{\pi}6}(\lambda+k^\f13y)\bigr)\int_0^\oo Ai\bigl(e^{i\f{\pi}6}(\lambda+k^\f13z)\bigr) \hat{f_0}(k,z)dz d\lambda\\
        &+k^\f13 \int_\R e^{ik^\f23\lambda t}\Bigl(Ai\bigl(e^{i\f{\pi}6}(\lambda+k^\f13y)\bigr)\int_0^y Ai\bigl(e^{i\f{5\pi}6}(\lambda+k^\f13z)\bigr) \hat{f_0}(k,z)dz \\
        &\qquad\qquad\qquad + Ai\bigl(e^{i\f{5\pi}6}(\lambda+k^\f13y)\bigr)\int_y^\oo Ai\bigl(e^{i\f{\pi}6}(\lambda+k^\f13z)\bigr) \hat{f_0}(k,z)dz \Bigr) d\lambda.
    \end{aligned}
\end{equation}
For simplicity of notations, we only consider the case $k>0$ in the formula, and the negative cases can be handled in the same way.

For the last part, we shall use the following Lemma:
\begin{lem}
    Given $y>z\geq0$, there holds
    \begin{equation}\label{eq2.11}
        \begin{aligned}
            &\int_\R e^{ik^\f23\lambda t}Ai\bigl(e^{i\f{\pi}6}(\lambda+k^\f13y)\bigr)Ai\bigl(e^{i\f{5\pi}6}(\lambda+k^\f13z)\bigr) d\lambda\\
            &=-\f{k^{-\f13}}{\sqrt{4\pi t}}\exp \bigl(-\f{(y-z)^2}{4t}-i\f{kt}2(y+z)-\f{k^2}{12}t^3\bigr).
        \end{aligned}
    \end{equation}
\end{lem}
\begin{proof}
    One can directly prove this integration via the inverse  Fourier-Laplace transform representation of Airy functions. But here, we give an alternative proof by considering the equation \eqref{eq2.5} on the whole plane. 

    On one hand, this equation \eqref{eq2.5} on the whole plane has an explicit solution formula. In the variable $(k,y)$, it writes
    \begin{equation}\label{eq2.12}
        \hat{f}(t,k,y)=\int_\R \f1{\sqrt{4\pi t}}\exp \bigl(-\f{(y-z)^2}{4t}-i\f{kt}2(y+z)-\f{k^2}{12}t^3\bigr) \hat{f_0}(k,z)dz.
    \end{equation}

    On the other hand, after the Laplace transform, \eqref{eq2.7} on $\R$ can be solved by
\begin{align*}
    f_{k,\lambda}(y)=2\pi k^{-\f13}\Bigl( &Ai\bigl(e^{i\f\pi6}(\lambda+k^\f13y)\bigr)\int_{-\oo}^y Ai\bigl(e^{i\f{5\pi}6}(\lambda+k^\f13z)\bigr)\hat{f_0}(k,z)dz\\
    &+Ai\bigl(e^{i\f{5\pi}6}(\lambda+k^\f13y)\bigr)\int_y^\oo Ai\bigl(e^{i\f\pi6}(\lambda+k^\f13z)\bigr)\hat{f_0}(k,z)dz\Bigr),
\end{align*}
where we used the facts that $\lim\limits_{y\rightarrow +\infty}Ai\bigl(e^{i\f\pi6}(\lambda+k^\f13y)\bigr)=\lim\limits_{y\rightarrow-\oo}Ai\bigl(e^{i\f{5\pi}6}(\lambda+k^\f13y)\bigr)=0$ and the Wronskian between these two functions is
$$
    W[Ai\bigl(e^{i\f\pi6}(\lambda+k^\f13y),Ai\bigl(e^{i\f{5\pi}6}(\lambda+k^\f13y)]=-W_{k,\lambda}=-(2\pi)^{-1}k^\f13.
$$
Therefore, we apply the Laplace inverse transform and arrive at another formula:
\begin{equation}\label{eq2.13}
    \begin{aligned}
        \hat{f}(t,k,y)=-k^\f13\int_\R e^{ik^\f23\lambda t}&\Bigl( Ai\bigl(e^{i\f\pi6}(\lambda+k^\f13y)\bigr)\int_{-\oo}^y Ai\bigl(e^{i\f{5\pi}6}(\lambda+k^\f13z)\bigr)\hat{f_0}(k,z)dz\\
    &+Ai\bigl(e^{i\f{5\pi}6}(\lambda+k^\f13y)\bigr)\int_y^\oo Ai\bigl(e^{i\f\pi6}(\lambda+k^\f13z)\bigr)\hat{f_0}(k,z)dz\Bigr) d\lambda.
    \end{aligned}
\end{equation}

Comparing \eqref{eq2.12} and \eqref{eq2.13}, one can get the equality \eqref{eq2.11} and finish the proof.
\end{proof}

By applying \eqref{eq2.11} to the last two lines in \eqref{eq2.10}, we rewrite the formula as
\begin{equation}\label{eq2.14}
    \begin{aligned}
    \hat{f}(t,k,y)=&-k^\f13 \int_\R e^{ik^\f23\lambda t}\f{Ai(e^{i\f{5\pi}6}\lambda)}{Ai(e^{i\f{\pi}6}\lambda)} Ai\bigl(e^{i\f{\pi}6}(\lambda+k^\f13y)\bigr)\int_0^\oo Ai\bigl(e^{i\f{\pi}6}(\lambda+k^\f13z)\bigr) \hat{f_0}(k,z)dz d\lambda\\
        &- \f{1}{\sqrt{4\pi t}}\int_0^\oo \exp \bigl(-\f{(y-z)^2}{4t}-i\f{kt}2(y+z)-\f{k^2}{12}t^3\bigr) \hat{f_0}(k,z)dz.
\end{aligned}
\end{equation}

Inspired by the formal computations below \eqref{eq1.4}, the slowest decay comes from the parts with very small $k$. However, the formal limit in the first line of the above formula does not vanish because $Ai(e^{i\f\pi6}\lambda)\neq 0$. This suggests that we reorganize the formula before doing the semi-group estimates.

By the Taylor expansion, we have
\begin{align*}
    Ai\bigl(e^{i\f{\pi}6}(\lambda+k^\f13z)\bigr)
    =& Ai\bigl(e^{i\f{\pi}6}\lambda\bigr)+e^{i\f\pi6}k^\f13 z Ai'\bigl(e^{i\f{\pi}6}\lambda\bigr)+\f12 e^{i\f\pi3}k^\f23 z^2 Ai''\bigl(e^{i\f{\pi}6}\lambda\bigr)\\
    &+\f{i}2kz^3 \int_0^1 (1-s)^2  Ai'''\bigl(e^{i\f{\pi}6}(\lambda+sk^\f13z )\bigr)ds\\
    =& Ai\bigl(e^{i\f{\pi}6}\lambda\bigr)+e^{i\f\pi6}k^\f13 z Ai'\bigl(e^{i\f{\pi}6}\lambda\bigr)+ \f{i}2 \lambda k^\f23 z^2 Ai\bigl(e^{i\f{\pi}6}\lambda\bigr)\\
    &+\f12e^{i\f{2\pi}3}kz^3 \int_0^1 (1-s)^2 (\lambda+sk^\f13 z)  Ai'\bigl(e^{i\f{\pi}6}(\lambda+sk^\f13z )\bigr)ds.
\end{align*}
We observe that the $Ai(e^{i\f\pi6})$ in the first term and the third term above can be used to cancel the denominator. Therefore, one derives from \eqref{eq2.11} that 
\begin{align*}
    &-k^\f13 \int_\R e^{ik^\f23\lambda t}\f{Ai(e^{i\f{5\pi}6}\lambda)}{Ai(e^{i\f{\pi}6}\lambda)} Ai\bigl(e^{i\f{\pi}6}(\lambda+k^\f13y)\bigr)\int_0^\oo Ai\bigl(e^{i\f{\pi}6}\lambda\bigr)  \hat{f_0}(k,z)dz d\lambda\\
    &=\f{1}{\sqrt{4\pi t}}\int_0^\oo \exp \bigl(-\f{y^2}{4t}-i\f{kt}2y-\f{k^2}{12}t^3\bigr) \hat{f_0}(k,z)dz,
\end{align*}
and by taking $\p_t$ on \eqref{eq2.11} that
\begin{align*}
    &-k^\f13 \int_\R e^{ik^\f23\lambda t}\f{Ai(e^{i\f{5\pi}6}\lambda)}{Ai(e^{i\f{\pi}6}\lambda)} Ai\bigl(e^{i\f{\pi}6}(\lambda+k^\f13y)\bigr)\int_0^\oo \f{i}2\lambda k^\f23 z^2 Ai\bigl(e^{i\f{\pi}6}\lambda\bigr)  \hat{f_0}(k,z)dz d\lambda\\
    &=\f12\f{d}{dt}\Bigl(\f{1}{\sqrt{4\pi t}}\int_0^\oo \exp \bigl(-\f{y^2}{4t}-i\f{kt}2y-\f{k^2}{12}t^3\bigr) z^2 \hat{f_0}(k,z)dz\Bigr)\\
    &=\f{1}{\sqrt{4\pi t}}\int_0^\oo \bigl(-\f1{4t}+\f{y^2}{8t^2}-\f{iky}4 -\f{k^2t^2}8\bigr) \exp \bigl(-\f{y^2}{4t}-i\f{kt}2y-\f{k^2}{12}t^3\bigr)z^2 \hat{f_0}(k,z)dz.
\end{align*}

We denote $E(y,z)= \exp \bigl(-\f{(y-z)^2}{4t}-i\f{kt}2(y+z)\bigr)$ and use Taylor expansion to compute
\begin{align*}
    E(y,z)=&E(y,0)+z\p_z E(y,0)+\f12 z^2 \p_z^2 E(y,0)+\f12 z^3\int_0^1 (1-s)^2 E(y,sz)ds\\
    =& \exp \bigl(-\f{y^2}{4t}-i\f{kt}2y\bigr)+z(\f{y}{2t}-\f{i}2 kt)\exp \bigl(-\f{y^2}{4t}-i\f{kt}2y\bigr)\\
    &+ z^2 \bigl(-\f1{4t}+\f{y^2}{8t^2}-\f{iky}4 -\f{k^2t^2}8\bigr) \exp \bigl(-\f{y^2}{4t}-i\f{kt}2y-\f{k^2}{12}t^3\bigr)\\
    &+ \f12 z^3 \int_0^1 (1-s)^2 \Bigl(-\f{3(y-zs)}{4t^2}+\f{3ik}{4}+\bigl(\f{y-sz}{2t}-\f{ikt}2\bigr)^3\Bigr)\\
    &\qquad\times\exp \bigl(-\f{(y-sz)^2}{4t}-i\f{kt}2(y+sz)\bigr)ds.
\end{align*}

Combining \eqref{eq2.14} and all the above expansions, we arrive at the following formula:
\begin{equation}\label{eq2.15}
    \begin{aligned}
        \hat{f}(t,k,y)=&-T_1-T_2-T_3-T_4,
    \end{aligned}
\end{equation} 
where the four terms are given by
\begin{equation}\label{eq2.16}
    \begin{aligned}
        T_1&\eqdefa e^{i\f\pi6}k^\f23 \int_{0}^\oo z \hat{f_0}(k,z)dz \int_\R e^{ik^\f23\lambda t}\f{Ai(e^{i\f{5\pi}6}\lambda)}{Ai(e^{i\f{\pi}6}\lambda)} Ai\bigl(e^{i\f{\pi}6}(\lambda+k^\f13y)\bigr) Ai'\bigl(e^{i\f{\pi}6}\lambda\bigr)  d\lambda,\\
        T_2&\eqdefa \f1{\sqrt{4\pi t}}\bigl(\f{y}{2t}-\f{i}2 kt\bigr)\exp \bigl(-\f{y^2}{4t}-i\f{kt}2y-\f{k^2}{12}t^3\bigr)\int_{0}^\oo z \hat{f_0}(k,z)dz  ,\\
        T_3&\eqdefa \f12e^{i\f{2\pi}3} k^\f43 \int_\R e^{ik^\f23\lambda t}\f{Ai(e^{i\f{5\pi}6}\lambda)}{Ai(e^{i\f{\pi}6}\lambda)} Ai\bigl(e^{i\f{\pi}6}(\lambda+k^\f13y)\bigr)\\
        &\qquad\qquad\times\int_0^\oo z^3 \hat{f_0}(k,z)\int_0^1 (1-s)^2 (\lambda+sk^\f13 z)  Ai'\bigl(e^{i\f{\pi}6}(\lambda+sk^\f13z )\bigr)ds dz d\lambda,\\
        T_4 &\eqdefa \f{1}{4\sqrt{\pi t}}\int_0^\oo \int_0^1 (1-s)^2 \Bigl(-\f{3(y-zs)}{4t^2}+\f{3ik}{4}+\bigl(\f{y-sz}{2t}-\f{ikt}2\bigr)^3\Bigr)\\
        &\qquad\qquad\times\exp \bigl(-\f{(y-sz)^2}{4t}-i\f{kt}2(y+sz)-\f{k^2}{12}t^3\bigr) ds  z^3\hat{f_0}(k,z)dz.
    \end{aligned}
\end{equation}

We remark that $T_1$ and $T_2$ are the main terms with the coefficient $\int_0^\oo z \hat{f_0}(k,z)dz$, and we will show in the next subsection that they contain a main part depending only on the vertical momentum $M_2[f_0]=\int_0^\oo z \hat{f_0}(0,z)dz.$ The terms $T_3$ and $T_4$ will be proven to have better time decays.

\subsection{The time decay of the $L^2$ norms}

In this subsection, we study the $L^2$ norms of $\hat{f}(t,k,y)$ given by the solution formula \eqref{eq2.16} and \eqref{eq2.16}. With assumptions of initial data $f_0$ belonging to some weighted spaces, we shall prove the sharp time decay for this linear evolution. Recall that we only use the formula \eqref{eq2.15} to describe the solution for $k\geq0$, all the $L^2$ norms in this subsection refer to $L^2(\R_+\times\R_+)$ for $(k,y)$ variables.

First, we consider $T_2$ and $T_4$, which are expressed without Airy functions.

\begin{lem}
    Let $T_2$ and $T_4$ be given by \eqref{eq2.16}. Then, for $m>4$, one has
    \begin{equation}\label{eq2.17}
        t^\f12\left\|T_2-\f{M_2[f_0]}{\sqrt{4\pi t}}\bigl(\f{y}{2t}-\f{i}2 kt\bigr)\exp \bigl(-\f{y^2}{4t}-i\f{kt}2y-\f{k^2}{12}t^3\bigr)\right\|_{L^2}+\|T_4\|_{L^2}\leq C_m t^{-\f52}\|f_0\|_{L^2(m)}.
    \end{equation}
    Moreover, there holds
    \begin{equation}\label{eq2.18}
         \|T_2\|_{L^2}+\|\f{M_2[f_0]}{\sqrt{4\pi t}}\bigl(\f{y}{2t}-\f{i}2 kt\bigr)\exp \bigl(-\f{y^2}{4t}-i\f{kt}2y-\f{k^2}{12}t^3\bigr)\|_{L^2}\leq C_m t^{-\f32}\|f_0\|_{L^2(m)}.
    \end{equation}
\end{lem}
\begin{proof}
    First, one observe that for $a,b\in \{0,1,2,3\}$ 
    $$
        |y^a k^b \exp \bigl(-\f{y^2}{4t}-i\f{kt}2y-\f{k^2}{12}t^3\bigr)|\leq C t^{\f{a}2-\f32b} \exp \bigl(-\f{y^2}{8t}-\f{k^2t^3}{20}\bigr),
    $$
    which implies that 
    \begin{equation}\label{eq2.19}
        \begin{aligned}
        |\f1{\sqrt{4\pi t}}\bigl(\f{y}{2t}-\f{i}2 kt\bigr)\exp \bigl(-\f{y^2}{4t}-i\f{kt}2y-\f{k^2}{12}t^3\bigr)|
        \leq C t^{-1} \exp \bigl(-\f{y^2}{8t}-\f{k^2t^3}{20}\bigr),
    \end{aligned}
    \end{equation}
    and 
    \begin{equation}\label{eq2.20}
        \begin{aligned}
        |\f{1}{4\sqrt{\pi t}} &  \Bigl(-\f{3(y-zs)}{4t^2}+\f{3ik}{4}+\bigl(\f{y-sz}{2t}-\f{ikt}2\bigr)^3\Bigr)\exp \bigl(-\f{(y-sz)^2}{4t}-i\f{kt}2(y+sz)-\f{k^2}{12}t^3\bigr)|\\
        &\leq C t^{-2} \exp \bigl(-\f{(y-sz)^2}{8t}-\f{k^2t^3}{20}\bigr).
        \end{aligned}
    \end{equation}

    From \eqref{eq2.20}, we find that 
    \begin{equation}\label{eq2.21}
        \begin{aligned}
            \|T_4\|_{L^2}&\leq C t^{-2}\|e^{-\f{k^2t^3}{20}}\int_0^1 (1-s)^2 \int_0^\oo \|e^{-\f{(y-sz)^2}{8t}} \|_{L^2_y(\R_+)} z^3 |\hat{f_0}(k,z)|dz ds \|_{L^2_k(\R_+)}\\
            &\leq C t^{-\f74}\|e^{-\f{k^2t^3}{20}}\|_{L^2_k(\R_+)} \|z^3\hat{f_0}(k,z)\|_{L^\oo_k(L^1_z)}\\
            &\leq C_m t^{-\f52} \|f_0\|_{L^2(m)},
        \end{aligned}
    \end{equation}
where we used that for $m>4$,
$$
    \|z^3\hat{f_0}(k,z)\|_{L^\oo_k(L^1_z)} \leq \|y^3 f_0(x,y)\|_{L^1(\R^2_+)}\leq C_m \|f_0\|_{L^2(m)}.
$$
Similarly, we deduce from \eqref{eq2.19} that
\begin{align*}
    &\|T_2\|_{L^2}+\|\f{M_2[f_0]}{\sqrt{4\pi t}}\bigl(\f{y}{2t}-\f{i}2 kt\bigr)\exp \bigl(-\f{y^2}{4t}-i\f{kt}2y-\f{k^2}{12}t^3\bigr)\|_{L^2}\\
    \leq & C t^{-1} \|e^{-\f{k^2 t^3}{20}}\|_{L^2_k(\R_+)} \|e^{-\f{y^2}{8t}} \|_{L^2_y} \|z\hat{f_0}(k,z)\|_{L^\oo_k(L^1_z)}\\
    \leq & C_m t^{-\f32} \|f_0\|_{L^2(m)},
\end{align*}
which gives \eqref{eq2.18}.

To achieve better decay, we can use \eqref{eq2.19} and 
\begin{equation}\label{eq2.21a}
    |\int_0^\oo z \hat{f_0}(k,z)dz-M_2[f_0]|\leq C k \| z\p_k \hat{f_0}(k,z)\|_{L^\oo_k(L^1_z)} \leq C_m k\|f_0\|_{L^2(m)}
\end{equation}
to derive
\begin{align*}
    &\|T_2-\f{M_2[f_0]}{\sqrt{4\pi t}}\bigl(\f{y}{2t}-\f{i}2 kt\bigr)\exp \bigl(-\f{y^2}{4t}-i\f{kt}2y-\f{k^2}{12}t^3\bigr)\|_{L^2}\\
    \leq & C_m t^{-1} \|k e^{-\f{k^2 t^3}{20}}\|_{L^2_k(\R_+)} \|e^{-\f{y^2}{8t}} \|_{L^2_y} \|f_0\|_{L^2(m)}\\
    \leq & C_m t^{-3} \|f_0\|_{L^2(m)},
\end{align*}
which together with \eqref{eq2.21} proves \eqref{eq2.17}. This finishes the proof.
\end{proof}

Now, we turn to the $L^2$ estimate of $T_1$. In order to change the power of $k$ into some time decay, we want to change the integral domain for $\lambda$ a little bit, so that $e^{ik^\f23\lambda t}$ relates $k$ with time decays.

For any fixed $y>0$, the function $e^{ik^\f23\lambda t}\f{Ai(e^{i\f{5\pi}6}\lambda)}{Ai(e^{i\f{\pi}6}\lambda)} Ai\bigl(e^{i\f{\pi}6}(\lambda+k^\f13y)\bigr) Ai'\bigl(e^{i\f{\pi}6}\lambda\bigr)$ is analytic with respect to $\lambda$, except at the singular points $\lambda\in \{e^{-i\f\pi6}\xi_n \}$ which are the zeros of $Ai(e^{i\f\pi6}\lambda)$. Also, the function $e^{ik^\f23\lambda t}\f{Ai(e^{i\f{5\pi}6}\lambda)}{Ai(e^{i\f{\pi}6}\lambda)} Ai\bigl(e^{i\f{\pi}6}(\lambda+k^\f13y)\bigr) Ai'\bigl(e^{i\f{\pi}6}\lambda\bigr)$ is bounded in the sector $\arg \lambda \in (-\f76\pi,\f16\pi)$ (decay with an exponential rate depending on $y$), and $e^{ik^\f23\lambda t}$ decays exponentially in the sector $\arg\lambda \in (0,\pi)$. Therefore, we apply Cauchy's Theorem to replace the integral over $\R$ by an integral over the contour consisting of the two rays $e^{i\delta}[0,\infty)$ and $-e^{-i\delta}[0,\infty)$ with $0<\delta<\f\pi6$. 

\begin{figure}[htbp]
    \centering

\begin{tikzpicture}[scale=2]

\draw[->] (-2,0) -- (2,0) node[right] {Re$\lambda$};
\draw[->] (0,-0.2) -- (0,1) node[above] {Im$\lambda$};

\draw[line width=1pt] ( {cos(10)}, {sin(10)} ) -- ( {2*cos(10)}, {2*sin(10)} );
\draw[line width=1pt] ( {-cos(10)}, {sin(10)} ) -- ( {-2*cos(10)}, {2*sin(10)} );

\draw[line width=1pt] ( {cos(10)}, {sin(10)} ) -- ( {-cos(10)}, {sin(10)} );

\draw[dashed] (0,0) -- ( {2*cos(10)}, {2*sin(10)} );
\draw[dashed] (0,0) -- ( {-2*cos(10)}, {2*sin(10)} );

\foreach \r in {0.5,0.8,1.1,1.4,1.7} {
    \fill ( {-\r*cos(30)}, {\r*sin(30)} ) circle (0.02);
}

\node at (2.2,0.5) {$L_3=e^{i\delta}[1,\infty)$};
\node at (-2.2,0.5) {$L_1=-e^{-i\delta}[1,\infty)$};
\node at (0.3,0.28) {$L_2$};

\node at (-1.5,1) {$e^{-i\pi/6}\xi_n$};
\node at (-0.33,0.35) {$e^{-i\pi/6}\xi_1$};

\end{tikzpicture}
\caption{The contour $L_1+L_2+L_3$.}
\end{figure}
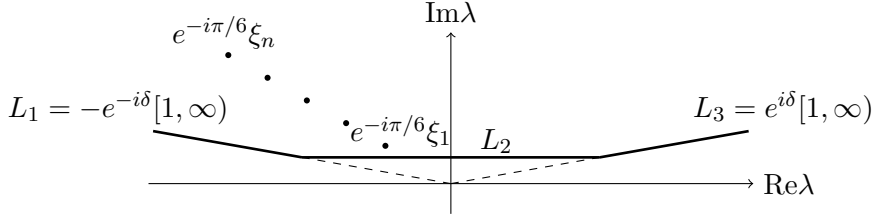

Now, we fix the choice of $0<\delta<\f\pi{100}$ small enough such that $\sin \delta<\f12 |\xi_1|$. The term $T_1$ is equivalent to the one by replacing the integration over $\lambda\in \R$ to the integration on the contour consisting $L_1=-e^{-i\delta}[1,\infty)$, $L_2=\cos\delta[-1,1]+i\sin\delta$ and $L_3=e^{-i\pi/6}\xi_n$, which means 
\begin{equation}\label{decompose T1}
    T_1=\bigl(T_{1,1}+T_{1,2}+T_{1,3}\bigr)\int_{0}^\oo z \hat{f_0}(k,z)dz
\end{equation}
 with 
\begin{align*}
    T_{1,1}&\eqdefa e^{i\f\pi6}k^\f23 \int_{L_1} e^{ik^\f23\lambda t}\f{Ai\bigl(e^{i\f{5\pi}6}\lambda\bigr)}{Ai\bigl(e^{i\f{\pi}6}\lambda\bigr)} Ai\bigl(e^{i\f{\pi}6}(\lambda+k^\f13y)\bigr) Ai'\bigl(e^{i\f{\pi}6}\lambda\bigr)  d\lambda \\
    &= e^{i(\f\pi6-\delta)}k^\f23 \int_1^\oo e^{-ie^{-i\delta}k^\f23\lambda t}\f{Ai'\bigl(e^{-i(\f{5\pi}6+\delta)}\lambda\bigr)}{Ai\bigl(e^{-i(\f{5\pi}6+\delta)}\lambda\bigr)} Ai\bigl(e^{-i(\f{\pi}6+\delta)}\lambda\bigr) Ai\bigl(e^{-i(\f{5\pi}6+\delta)}\lambda+e^{i\f{\pi}6}k^\f13y\bigr)   d\lambda,\\
    T_{1,2}&\eqdefa e^{i\f\pi6}k^\f23  \int_{L_2} e^{ik^\f23\lambda t}\f{Ai\bigl(e^{i\f{5\pi}6}\lambda\bigr)}{Ai\bigl(e^{i\f{\pi}6}\lambda\bigr)} Ai\bigl(e^{i\f{\pi}6}(\lambda+k^\f13y)\bigr) Ai'\bigl(e^{i\f{\pi}6}\lambda\bigr)  d\lambda\\
    &=\cos\delta\, e^{i\f\pi6}k^\f23  \int_{-1}^1 e^{(-\sin \delta +i\cos\delta\,\lambda)k^\f23 t}\f{Ai'\bigl(e^{i\f{\pi}6}(\cos \delta\, \lambda+i\sin\delta)\bigr)}{Ai\bigl(e^{i\f{\pi}6}(\cos \delta\, \lambda+i\sin\delta)\bigr)}\\
    &\qquad \qquad\qquad\qquad\times Ai\bigl(e^{i\f{5\pi}6}(\cos \delta\, \lambda+i\sin\delta)\bigr) Ai\bigl(e^{i\f{\pi}6}(\cos \delta\, \lambda+i\sin\delta+k^\f13y)\bigr)   d\lambda,\\
    T_{1,3}&\eqdefa e^{i\f\pi6}k^\f23 \int_{L_3} e^{ik^\f23\lambda t}\f{Ai\bigl(e^{i\f{5\pi}6}\lambda\bigr)}{Ai\bigl(e^{i\f{\pi}6}\lambda\bigr)} Ai\bigl(e^{i\f{\pi}6}(\lambda+k^\f13y)\bigr) Ai'\bigl(e^{i\f{\pi}6}\lambda\bigr)  d\lambda \\
    &= e^{i(\f\pi6+\delta)}k^\f23 \int_1^\oo e^{ie^{i\delta}k^\f23\lambda t}\f{Ai'\bigl(e^{i(\f{\pi}6+\delta)}\lambda\bigr)}{Ai\bigl(e^{i(\f{\pi}6+\delta)}\lambda\bigr)}  Ai\bigl(e^{i(\f{5\pi}6+\delta)}\lambda\bigr) Ai\bigl(e^{i(\f{\pi}6+\delta)}\lambda+e^{i\f{\pi}6}k^\f13y\bigr)   d\lambda.
\end{align*}

For the Airy function in the above formulas, we shall use the following Lemma:
\begin{lem}\label{lem2.3}
    There exists some large constant $C$ independent of $\lambda$, such that when $\lambda\in [1,\oo)$, there hold
    \begin{equation}\label{eq2.22}
        \| \f{Ai'\bigl(e^{-i(\f{5\pi}6+\delta)}\lambda+e^{i\f\pi6}y\bigr)}{Ai\bigl(e^{-i(\f{5\pi}6+\delta)}\lambda\bigr)} \|_{L^\oo_y(\R_+)}
        +\| \f{Ai'\bigl(e^{i(\f{\pi}6+\delta)}\lambda+e^{i\f\pi6}y\bigr)}{Ai\bigl(e^{i(\f{\pi}6+\delta)}\lambda\bigr)} \|_{L^\oo_y(\R_+)}\leq C \lambda^\f12,
    \end{equation}
    and
    \begin{equation}\label{eq2.23}
        \begin{aligned}
            &\|Ai\bigl(e^{-i(\f{\pi}6+\delta)}\lambda\bigr) Ai\bigl(e^{-i(\f{5\pi}6+\delta)}\lambda+e^{i\f{\pi}6}k^\f13y\bigr) \|_{L^2_y(\R_+)}\\
        &\qquad\qquad+\|Ai\bigl(e^{i(\f{5\pi}6+\delta)}\lambda\bigr) Ai\bigl(e^{i(\f{\pi}6+\delta)}\lambda+e^{i\f{\pi}6}k^\f13y\bigr) \|_{L^2_y(\R_+)} \leq C k^{-\f16} \lambda^{-\f34};
        \end{aligned}
    \end{equation}
    when $\lambda\in [-1,1]$, there hold 
    \begin{equation}\label{eq2.24}
        \|\f{Ai'\bigl(e^{i\f{\pi}6}(\cos \delta\, \lambda+i\sin\delta+y)\bigr)}{Ai\bigl(e^{i\f{\pi}6}(\cos \delta\, \lambda+i\sin\delta)\bigr)}\|_{L^\oo_y(\R_+)}\leq C,
    \end{equation}
    and 
    \begin{equation}\label{eq2.25}
        \|Ai\bigl(e^{i\f{5\pi}6}(\cos \delta\, \lambda+i\sin\delta)\bigr) Ai\bigl(e^{i\f{\pi}6}(\cos \delta\, \lambda+i\sin\delta+k^\f13y)\bigr)\|_{L^2_y(\R_+)}\leq C k^{-\f16}.
    \end{equation}
\end{lem}

The proof of this lemma requires a detailed analysis of the asymptotic behavior of the Airy function and its derivative. We prove it in Appendix \ref{appendix A}.

With Lemma \ref{lem2.3}, we can prove the following estimate of $T_1$:
\begin{lem}
    Let $T_1$ be given by \eqref{eq2.16}. Then, for $m>3$, one has
    \begin{equation}\label{eq2.28}
        \begin{aligned}
            \|T_1-M_2[f_0] e^{i\f\pi6}k^\f23 \int_\R e^{ik^\f23\lambda t}\f{Ai(e^{i\f{5\pi}6}\lambda)}{Ai(e^{i\f{\pi}6}\lambda)} Ai\bigl(e^{i\f{\pi}6}(\lambda+k^\f13y)\bigr) Ai'\bigl(e^{i\f{\pi}6}\lambda\bigr)  d\lambda\|_{L^2} \leq C_m t^{-3}\|f_0\|_{L^2(m)}.
        \end{aligned}
    \end{equation}
    Moreover, there holds
    \begin{equation}\label{eq2.29}
        \begin{aligned}
            &\|T_1\|_{L^2}+\|M_2[f_0] e^{i\f\pi6}k^\f23 \int_\R e^{ik^\f23\lambda t}\f{Ai(e^{i\f{5\pi}6}\lambda)}{Ai(e^{i\f{\pi}6}\lambda)} Ai\bigl(e^{i\f{\pi}6}(\lambda+k^\f13y)\bigr) Ai'\bigl(e^{i\f{\pi}6}\lambda\bigr)  d\lambda\|_{L^2} \\
        &\qquad\qquad\leq C_m t^{-\f32}\|f_0\|_{L^2(m)}.
        \end{aligned}
    \end{equation}
\end{lem}
\begin{proof}
    From \eqref{decompose T1}, we have that 
    \begin{equation}\notag
        \begin{aligned}
        &\|T_1\|_{L^2}+\|M_2[f_0] e^{i\f\pi6}k^\f23 \int_\R e^{ik^\f23\lambda t}\f{Ai(e^{i\f{5\pi}6}\lambda)}{Ai(e^{i\f{\pi}6}\lambda)} Ai\bigl(e^{i\f{\pi}6}(\lambda+k^\f13y)\bigr) Ai'\bigl(e^{i\f{\pi}6}\lambda\bigr)  d\lambda\|_{L^2} \\
        \leq & C\|z\hat{f_0}(k,z)\|_{L^\oo_k(L^1_z)} \bigl(\|T_{1,1}\|_{L^2 }+\|T_{1,2}\|_{L^2 }+\|T_{1,3}\|_{L^2 }\bigr)\\
        \leq & C_m\|f_0\|_{L^2(m)} \bigl(\|T_{1,1}\|_{L^2 }+\|T_{1,2}\|_{L^2 }+\|T_{1,3}\|_{L^2 }\bigr).
    \end{aligned}
    \end{equation}
    Also, from \eqref{eq2.21a}, we derive that 
    \begin{equation}\notag
        \begin{aligned}
            &\|T_1-M_2[f_0] e^{i\f\pi6}k^\f23 \int_\R e^{ik^\f23\lambda t}\f{Ai(e^{i\f{5\pi}6}\lambda)}{Ai(e^{i\f{\pi}6}\lambda)} Ai\bigl(e^{i\f{\pi}6}(\lambda+k^\f13y)\bigr) Ai'\bigl(e^{i\f{\pi}6}\lambda\bigr)  d\lambda\|_{L^2}\\
            \leq &C \|\f1k \bigl(\int_0^\oo z \hat{f_0}(k,z)dz-M_2[f_0]\bigr)\|_{L^\oo_k} \bigl(\|kT_{1,1}\|_{L^2 }+\|kT_{1,2}\|_{L^2 }+\|kT_{1,3}\|_{L^2 }\bigr)\\
            \leq & C_m\|f_0\|_{L^2(m)} \bigl(\|kT_{1,1}\|_{L^2 }+\|kT_{1,2}\|_{L^2 }+\|kT_{1,3}\|_{L^2 }\bigr).
        \end{aligned}
    \end{equation}
    Therefore, to prove \eqref{eq2.28} and \eqref{eq2.29}, it suffices to prove that for $a=0$ or $1$,
    \begin{equation}\label{eq2.30}
        \|k^a T_{1,1}\|_{L^2 }+\|k^aT_{1,2}\|_{L^2 }+\|k^aT_{1,3}\|_{L^2 }
        \leq C t^{-\f32(1+a)}.
    \end{equation}

    By applying \eqref{eq2.22} and \eqref{eq2.23}, we have that 
    \begin{align*}
        &\|k^a T_{1,1}\|_{L^2 }+\|k^aT_{1,3}\|_{L^2 } \\
        \leq &\| k^{\f23+a} \int_1^\oo e^{-\sin \delta\, k^\f23 \lambda t} \Bigl(| \f{Ai'\bigl(e^{-i(\f{5\pi}6+\delta)}\lambda\bigr)}{Ai\bigl(e^{-i(\f{5\pi}6+\delta)}\lambda\bigr)} |\|Ai\bigl(e^{-i(\f{\pi}6+\delta)}\lambda\bigr) Ai\bigl(e^{-i(\f{5\pi}6+\delta)}\lambda+e^{i\f{\pi}6}k^\f13y\bigr) \|_{L^2_y(\R_+)}\\
        &\qquad\qquad+| \f{Ai'\bigl(e^{i(\f{\pi}6+\delta)}\lambda\bigr)}{Ai\bigl(e^{i(\f{\pi}6+\delta)}\lambda\bigr)} | \|Ai\bigl(e^{i(\f{5\pi}6+\delta)}\lambda\bigr) Ai\bigl(e^{i(\f{\pi}6+\delta)}\lambda+e^{i\f{\pi}6}k^\f13y\bigr) \|_{L^2_y(\R_+)} \Bigr)d\lambda \|_{L^2_k(\R_+)}\\
        \leq &C\int_1^\oo\|k^{\f12+a}  e^{-\sin (\delta) k^\f23 \lambda t}  \|_{L^2_k(\R_+)}\lambda^{-\f14}d\lambda\\
        \leq & C \int_1^\oo (\lambda t)^{-\f32(1+a)}\lambda^{-\f14}d\lambda \leq C t^{-\f32(1+a)}.
    \end{align*}
    Similarly, we deduce from \eqref{eq2.24} and \eqref{eq2.25} that 
    \begin{align*}
        \|k^a T_{1,2}\|_{L^2} \leq C \|k^{\f12+a} e^{-\sin\delta\, k^\f23 t}\|_{L^2_k} \leq C t^{-\f32(1+a)}.
    \end{align*}
    
    Combining the above two estimates, we have proven \eqref{eq2.30}, which finishes the proof.
\end{proof}

Finally, to estimate $T_3$, we again need to change the integral domain as for $T_1$. Similar to \eqref{decompose T1}, we introduce the decomposition of $T_3$ as
\begin{equation}\label{decompose T2}
    T_3=T_{3,1}+T_{3,2}+T_{3,3}
\end{equation}
 with 
\begin{align*}
    T_{3,1}&\eqdefa \f12e^{i\f{2\pi}3} k^\f43 \int_{L_1} e^{ik^\f23\lambda t}\f{Ai(e^{i\f{5\pi}6}\lambda)}{Ai(e^{i\f{\pi}6}\lambda)} Ai\bigl(e^{i\f{\pi}6}(\lambda+k^\f13y)\bigr)\\
        &\qquad\qquad\times\int_0^\oo z^3 \hat{f_0}(k,z)\int_0^1 (1-s)^2 (\lambda+sk^\f13 z)  Ai'\bigl(e^{i\f{\pi}6}(\lambda+sk^\f13z )\bigr)ds dz d\lambda \\
    &= \f12 e^{i(\f{2\pi}3-\delta)}k^\f43 \int_1^\oo e^{-ie^{-i\delta}k^\f23\lambda t} Ai\bigl(e^{-i(\f{\pi}6+\delta)}\lambda\bigr) Ai\bigl(e^{-i(\f{5\pi}6+\delta)}\lambda+e^{i\f{\pi}6}k^\f13y\bigr)  \\
    &\qquad\qquad\times \int_0^\oo z^3 \hat{f_0}(k,z)\int_0^1 (1-s)^2 (\lambda+sk^\f13 z) \f{ Ai'\bigl(e^{-i(\f{5\pi}6+\delta)}\lambda+e^{i\f{\pi}6}sk^\f13z \bigr)}{Ai\bigl(e^{-i(\f{5\pi}6+\delta)}\lambda\bigr)}ds dz d\lambda,\\
    T_{3,2}&\eqdefa \f12e^{i\f{2\pi}3} k^\f43 \int_{L_2} e^{ik^\f23\lambda t}\f{Ai(e^{i\f{5\pi}6}\lambda)}{Ai(e^{i\f{\pi}6}\lambda)} Ai\bigl(e^{i\f{\pi}6}(\lambda+k^\f13y)\bigr)\\
        &\qquad\qquad\times\int_0^\oo z^3 \hat{f_0}(k,z)\int_0^1 (1-s)^2 (\lambda+sk^\f13 z)  Ai'\bigl(e^{i\f{\pi}6}(\lambda+sk^\f13z )\bigr)ds dz d\lambda\\
    &=\f12\cos\delta\, e^{i\f{2\pi}3}k^\f43  \int_{-1}^1 e^{(-\sin \delta +i\cos\delta\,\lambda)k^\f23 t}Ai\bigl(e^{i\f{5\pi}6}(\cos \delta\, \lambda+i\sin\delta)\bigr) Ai\bigl(e^{i\f{\pi}6}(\cos \delta\, \lambda+i\sin\delta+k^\f13y)\bigr)\\
    &\qquad\qquad\times   \int_0^\oo z^3 \hat{f_0}(k,z)\int_0^1 (1-s)^2 (\lambda+sk^\f13 z) \f{Ai'\bigl(e^{i\f{\pi}6}(\cos \delta\, \lambda+i\sin\delta+sk^\f13 z)\bigr)}{Ai\bigl(e^{i\f{\pi}6}(\cos \delta\, \lambda+i\sin\delta)\bigr)}ds dz d\lambda,\\
    T_{3,3}&\eqdefa \f12e^{i\f{2\pi}3} k^\f43 \int_{L_3} e^{ik^\f23\lambda t}\f{Ai(e^{i\f{5\pi}6}\lambda)}{Ai(e^{i\f{\pi}6}\lambda)} Ai\bigl(e^{i\f{\pi}6}(\lambda+k^\f13y)\bigr)\\
        &\qquad\qquad\times\int_0^\oo z^3 \hat{f_0}(k,z)\int_0^1 (1-s)^2 (\lambda+sk^\f13 z)  Ai'\bigl(e^{i\f{\pi}6}(\lambda+sk^\f13z )\bigr)ds dz d\lambda \\
    &= \f12 e^{i(\f{2\pi}3+\delta)}k^\f43 \int_1^\oo e^{ie^{i\delta}k^\f23\lambda t}Ai\bigl(e^{i(\f{5\pi}6+\delta)}\lambda\bigr) Ai\bigl(e^{i(\f{\pi}6+\delta)}\lambda+e^{i\f{\pi}6}k^\f13y\bigr)\\
    &\qquad\qquad\times    \int_0^\oo z^3 \hat{f_0}(k,z)\int_0^1 (1-s)^2 (\lambda+sk^\f13 z) \f{Ai'\bigl(e^{i(\f{\pi}6+\delta)}\lambda+e^{i\f\pi6}sk^\f13z\bigr)}{Ai\bigl(e^{i(\f{\pi}6+\delta)}\lambda\bigr)}  ds dz d\lambda.
\end{align*}

\begin{lem}
    Let $T_3$ be given by \eqref{eq2.16}. Then, for $m>5$, one has
    \begin{equation}\label{eq2.32}
        \|T_3\|_{L^2}\leq C_m \|f_0\|_{L^2(m)}(t^{-\f52}+t^{-3}).
    \end{equation}
\end{lem}
\begin{proof}
    Due to the decomposition \eqref{decompose T2}, we shall do the estimates of $T_{3,j}$. 

    For $T_{3,1}$ and $T_{3,3}$, we use $|\lambda+sk^\f13z|\leq \lambda(1+z)(1+k^\f13)$ to find
    \begin{align*}
        &\|T_{3,1}\|_{L^2}+\|T_{3,3}\|_{L^2}\leq \f12 \|(1+z)z^3\hat{f_0}(k,z)\|_{L^\oo_k(L^1_z)} \| k^\f43(1+k^\f13) \int_1^\oo e^{-\sin\delta\, k^\f23\lambda t} \lambda \\
        &\times\Bigl(\|Ai\bigl(e^{-i(\f{\pi}6+\delta)}\lambda\bigr) Ai\bigl(e^{-i(\f{5\pi}6+\delta)}\lambda+e^{i\f{\pi}6}k^\f13y\bigr) \|_{L^2_y(\R_+)}\| \f{Ai'\bigl(e^{-i(\f{5\pi}6+\delta)}\lambda+e^{i\f\pi6 z}\bigr)}{Ai\bigl(e^{-i(\f{5\pi}6+\delta)}\lambda\bigr)} \|_{L^\oo_z(\R_+)}\\
        &\quad+ \|Ai\bigl(e^{i(\f{5\pi}6+\delta)}\lambda\bigr) Ai\bigl(e^{i(\f{\pi}6+\delta)}\lambda+e^{i\f{\pi}6}k^\f13y\bigr) \|_{L^2_y(\R_+)}\| \f{Ai'\bigl(e^{i(\f{\pi}6+\delta)}\lambda+e^{i\f\pi6}z\bigr)}{Ai\bigl(e^{i(\f{\pi}6+\delta)}\lambda\bigr)} \|_{L^\oo_z(\R_+)} \Bigr)d\lambda \|_{L^2_k(\R_+)}\\
        &\leq C_m \|f_0\|_{L^2(m)} \int_1^\oo \lambda^\f34\| k^\f76(1+k^\f13)e^{-\sin\delta\, k^\f23\lambda t}\|_{L^2_k(\R_+)} d\lambda.
    \end{align*}
    Thanks to 
    $$
        \int_1^\oo \lambda^\f34\| k^\f76(1+k^\f13)e^{-\sin\delta\, k^\f23\lambda t}\|_{L^2_k(\R_+)} d\lambda \leq C \int_1^\oo \lambda^\f34 (\lambda t)^{-\f52}\bigl(1+(\lambda t)^{-\f12}\bigr)d\lambda
        \leq C (t^{-\f52}+t^{-3}),
    $$
    we arrive at $\|T_{3,1}\|_{L^2}+\|T_{3,3}\|_{L^2}\leq C_m \|f_0\|_{L^2(m)}(t^{-\f52}+t^{-3}).$

    Similarly, we can prove $T_{3,2}$ satisfies
    \begin{align*}
        \|T_{3,2}\|_{L^2}&\leq C\|(1+z)z^3\hat{f_0}(k,z)\|_{L^\oo_k(L^1_z)} \|k^\f76(1+k^\f13)e^{-\sin\delta\, k^\f23 t}\|_{L^2_k(\R_+)}\\
        &\leq C_m \|f_0\|_{L^2(m)}(t^{-\f52}+t^{-3}),
    \end{align*}
    which finishes the proof.
\end{proof}

At last, we conclude the $L^2$ estimates achieved in this subsection:
\begin{prop}
    { Let $m>5$ and $f(t)$ solve \eqref{eq2.5} with $f_0\in L^2(m)$. Then, we have 
        \begin{equation}\label{eq2.33}
            \|f(t)\|_{L^2(\R^2_+)}\leq C_m (1+t)^{-\f32} \|f_0\|_{L^2(m)}.
        \end{equation}
    Moreover, if we assume in addition $M_2[f_0]=0$, then
    \begin{equation}\label{eq2.34}
        \|f(t)\|_{L^2(\R^2_+)}\leq C_m (1+t)^{-\f52} \|f_0\|_{L^2(m)}.
    \end{equation}
    }
\end{prop}
\begin{proof}
    First, by taking the $L^2$ inner production of \eqref{eq2.5} with $f$, we have 
    \begin{equation}\label{eq2.35}
        \|f(t)\|_{L^2}\leq \|f_0\|_{L^2}.
    \end{equation}
        
    From \eqref{eq2.15}, \eqref{eq2.17}, \eqref{eq2.18}, \eqref{eq2.29} and \eqref{eq2.32}, we have shown that for $t>1$,
    \begin{align*}
        \|\hat{f}(t,k,y)\|_{L^2(\R_+\times\R_+)} \leq C_m (1+t)^{-\f32}\|f_0\|_{L^2(m)}.
    \end{align*}
    It's easy to use symmetry to get the same conclusion on the region $k<0$. Therefore, together with \eqref{eq2.35}, we have proven \eqref{eq2.33}.

    When we assume $M_2[f_0]=0$, one can observe that \eqref{eq2.17} and \eqref{eq2.18} gives better time decay for $T_2$ and $T_1$, which enable us to prove the $t^{-\f52}$ decay in \eqref{eq2.34}. This finishes the proof.
\end{proof}

\subsection{The decay estimates in the weighted space}
The goal of this section is to prove Proposition \ref{prop1.1}.

\begin{lem}\label{lem2.6}
    {
    Let $m>5$ and $f(t)$ solve \eqref{eq2.5} with $f_0\in L^2(m)$. If there holds 
        \begin{equation}\label{eq2.36}
            \|f(t)\|_{L^2(\R^2_+)}\leq C_m (1+t)^{-\varsigma } \|f_0\|_{L^2(m)},
        \end{equation}
        with some $\varsigma<\f{m}2$,
    then for all $a_1,a_2\in \N$ with $a_1+a_2\leq m$, we have 
    \begin{equation}\label{eq2.37}
        \|x^{a_1} y^{a_2} f(t)\|_{L^2(\R_+^2)}\leq C_m (1+t)^{\f{3a_1+a_2}{2}-\varsigma} \|f_0\|_{L^2(m)}.
    \end{equation}
    }
\end{lem}
\begin{proof}
    We first consider the highest derivatives in both directions one by one.
    \begin{itemize}
        \item Estimates of the horizontal weighted norms:
    \end{itemize}

    By taking $L^2$ inner of \eqref{eq2.5} with $y^{2m} f$, we do integration by parts to compute
    $$
    \f12\f{d}{dt} \|y^m f\|_{L^2}^2 +\|y^m\p_y f\|_{L^2}^2 = -m\int_{\R^2_+} y^{2m-1} \p_y|f|^2 dxdy = m(2m-1) \|y^{m-1} f\|_{L^2}^2.
    $$
    Then, we apply Holder's inequality and Young's inequality to get that for any $\delta>0$,
    $$
    \|y^{m-1} f\|_{L^2}^2  \leq \|f\|_{L^2}^\f2m\|y^m f\|_{L^2}^\f{2(m-1)}m \leq \f\delta{m(2m-1)(1+t)} \|y^m f\|_{L^2}^2+C_{\delta,m} (1+t)^{m-1}\|f\|_{L^2}^2,
    $$ 
    which implies
    \begin{equation}\notag
        \f{d}{dt} \|y^m f\|_{L^2}^2-\f{2\delta}{1+t}\|y^m f\|_{L^2}^2 \leq C_{\delta,m} (1+t)^{m-1} \|f\|_{L^2}^2.
    \end{equation}
     Then, using Gronwall's inequality, we derive from \eqref{eq2.36} that 
    \begin{align*}
        \|y^m f(t)\|_{L^2}^2 &\leq (1+t)^{2\delta } \|y^mf_0(x,y)\|_{L^2}^2+C_{\delta,m}  (1+t)^{2\delta}\int_0^t (1+s)^{m-1-2\delta} \|f(s)\|_{L^2}^2 ds \\
        &\leq C_{\delta,m}  (1+t)^{2\delta } \|f_0\|_{L^2(m)}^2\bigl( 1+ 
        \int_0^t (1+s)^{m-1-2\delta-2\varsigma}ds\bigr)\\
        &\leq C_{\delta,m}  (1+t)^{m-2\varsigma } \|f_0\|_{L^2(m)}^2,
    \end{align*}
    where we take a small $\delta$ such that $m-2\delta-2\varsigma>0$.
    By taking the square root, we conclude that
    \begin{equation}\label{eq2.38}
        \|y^m f(t)\|_{L^2}\leq C_{m}  (1+t)^{\f{m}2-\varsigma} \|f_0\|_{L^2(m)}.
    \end{equation}

    \begin{itemize}
        \item Estimates of the tangential weighted norms:
    \end{itemize}

    We take $L^2$ inner of \eqref{eq2.5} with $x^{2m} f$ and do integration by parts to write
    $$
    \f12\f{d}{dt} \|x^m f\|_{L^2}^2 +\|x^m\p_y f\|_{L^2}^2 = m\int_{\R^2_+} x^{2m-1} y|f|^2 dxdy.
    $$
    Again, we apply Holder's inequality and Young's inequality to get that for any $\delta>0$,
    $$
    \int_{\R^2_+} x^{2m-1} y|f|^2 dxdy \leq \|x^m f\|_{L^2}^{2-\f1m}\|y^m f\|_{L^2}^\f1m \leq \f\delta{m(1+t)} \|x^m f\|_{L^2}^2+C_{\delta,m}(1+t)^{2m-1}  \|y^m f\|_{L^2}^2,
    $$ 
    which implies
    \begin{equation}\notag
        \f{d}{dt} \|x^m f\|_{L^2}^2-\f{2\delta}{1+t}\|x^m f\|_{L^2}^2 \leq C_{\delta,m} (1+t)^{2m-1} \|y^m f\|_{L^2}^2.
    \end{equation}
    Again, we use Gronwall's inequality to derive from \eqref{eq2.12} that
    \begin{align*}
        \|x^m f(t)\|_{L^2}^2 &\leq (1+t)^{2\delta } \|x^m f_0(x,y)\|_{L^2}^2+C_{\delta,m}  (1+t)^{2\delta}\int_0^t (1+s)^{2m-1-2\delta} \|y^m f(s)\|_{L^2}^2 ds \\
        &\leq C_{\delta,m}  (1+t)^{2\delta } \|f_0\|_{L^2(m)}^2\bigl( 1+ 
        \int_0^t (1+s)^{3m-1-2\varsigma-2\delta} ds\bigr)\\
        &\leq C_{\delta,m}  (1+t)^{3m-2\varsigma } \|f_0\|_{L^2(m)}^2,
    \end{align*}
    where we choose $\delta$ so small that $3m-2\varsigma-2\delta>0$.
    By taking the square root, we conclude that
    \begin{equation}\label{eq2.40}
        \|x^m f(t)\|_{L^2}\leq C_{m} (1+t)^{\f32m-\varsigma}  \|f_0\|_{L^2(m)}.
    \end{equation}

    By doing interpolation among \eqref{eq2.36}, \eqref{eq2.38} and \eqref{eq2.40}, we prove \eqref{eq2.37}.
\end{proof}

Now, we can present the proof of Proposition \ref{prop1.1}.
\begin{proof}[Proof of Proposition \ref{prop1.1}]
    By denoting $F(t)=e^{\ln(1+t) \cL}F_0$, we have $(1+t)\p_t F=\cL F$ with $F(0)=F_0$. Then, we introduce $f(t,x,y)$ by \eqref{eq2.6}, which solves \eqref{eq2.5} with $f_0=F_0$.

    From \eqref{eq2.6}, we have that for any $a_1,a_2\in \N$, 
    \begin{equation}\label{eq2.41}
        \begin{aligned}
            \|X^{a_1} Y^{a_2} F(t,X,Y)\|_{L^2(\R_+^2)}&=(1+t)^\f52 \|X^{a_1}Y^{a_2} f(t,(1+t)^\f32X,(1+t)^\f12Y)\|_{L^2_{X,Y}(\R_+^2)}\\
            &= (1+t)^{\f{3-3a_1-a_2}2}\|x^{a_1} y^{a_2} f(t,x,y)\|_{L^2(\R^2_+)}.
        \end{aligned}
    \end{equation}

    From \eqref{eq2.33}, we can apply Lemma \ref{lem2.6} with $\varsigma=\f32$, and use \eqref{eq2.41} to find
    \begin{align*}
        \|F(t)\|_{L^2(m)}&=\sum_{a_1+a_2\leq m}\| X^{a_1} Y^{a_2} F(t)\|_{L^2(\R^2_+)}=\sum_{a_1+a_2\leq m}(1+t)^{\f{3-3a_1-a_2}2}\| x^{a_1} y^{a_2} f(t,x,y)\|_{L^2(\R^2_+)}\\
        &\leq \sum_{a_1+a_2\leq m}C_m\| f_0\|_{L^2(m)} \leq C_m \|F_0\|_{L^2(m)},
    \end{align*}
    which implies \eqref{eq2.4}.
    
    For \eqref{eq2.5a}, the proof can be similarly derived from \eqref{eq2.34}, Lemma \ref{lem2.6} and \eqref{eq2.41}. This finishes the proof.
\end{proof}

\begin{rmk}
    In this section, we also get that 
    \begin{align*}
        \bigg\|\hat{f}(t,k,y)&+M_2[f_0] e^{i\f\pi6}k^\f23 \int_\R e^{ik^\f23\lambda t}\f{Ai(e^{i\f{5\pi}6}\lambda)}{Ai(e^{i\f{\pi}6}\lambda)} Ai\bigl(e^{i\f{\pi}6}(\lambda+k^\f13y)\bigr) Ai'\bigl(e^{i\f{\pi}6}\lambda\bigr)  d\lambda\\
        &+\f{M_2[f_0]}{\sqrt{4\pi t}}\bigl(\f{y}{2t}-\f{i}2 kt\bigr)\exp \bigl(-\f{y^2}{4t}-i\f{kt}2y-\f{k^2}{12}t^3\bigr)\bigg\|_{L^2}
        \leq C_m (1+t)^{-\f52}\|f_0\|_{L^2(m)}.
    \end{align*}
    This main object with coefficient $M_2[f_0]$ actually gives out the kernel of $\cL$. From the change of variable \eqref{eq2.6}, we know that the horizontal Fourier transform of $F(t)$ is
    $$\hat{F}(t,\ell,Y)=(1+t)\hat{f}(t,(1+t)^{-\f32}{\ell},(1+t)^\f12Y).$$
    By removing some decaying terms, we can define the $\cF[\bar{\Omega}](\ell,Y)$ on $\ell>0$ as
    \begin{equation}\label{eq2.42}
        \begin{aligned}
        \cF[\bar{\Omega}](\ell,Y)=&-e^{i\f\pi6}\ell^\f23 \int_\R e^{i\ell^\f23\lambda }\f{Ai(e^{i\f{5\pi}6}\lambda)}{Ai(e^{i\f{\pi}6}\lambda)} Ai\bigl(e^{i\f{\pi}6}(\lambda+\ell^\f13Y)\bigr) Ai'\bigl(e^{i\f{\pi}6}\lambda\bigr)  d\lambda\\
        &-\f{1}{\sqrt{4\pi }}\bigl(Y-\f{i}2 \ell\bigr)\exp \bigl(-\f{1}{4}Y^2-\f{i}2\ell Y-\f{1}{12}\ell^2\bigr),
    \end{aligned}
    \end{equation}
    and on the region $\ell<0$ by symmetry. A slight modification of our proof here shows $\|e^{\tau\cL}F_0-\bar{\Omega}\|_{L^2}\rightarrow 0$, as $\tau$ goes to infinity. However, it is difficult to prove that $\bar{\Omega}$ defined by the formula \eqref{eq2.42} belongs to $L^2(m)$, since the derivatives in $\ell$ look complicated and may cause a singularity near $\ell=0$. For such a reason, we will find another formula for the kernel in the next section.
\end{rmk}

\section{The kernel of $\cL$}\label{section 3}
In this section, we prove the existence of the one-dimensional kernel of $\cL$ on $L^2(m)$. More precisely, we have 
\begin{prop}\label{prop3.1}
    {  There exists a unique function $\bar{\Omega}$ in all $L^2(m)$ with $5<m\in\N$, which solves 
    \begin{equation}\label{eq3.28}
        \cL \bar{\Omega}=0, \quad \bar\Omega(X, 0)=0, \andf \int_{\R^2_+} Y\bar{\Omega}(X,Y)dXdY=1.
    \end{equation}
    }
\end{prop}

From the estimates of the semi-group $e^{t\cL}$ in \eqref{eq2.4} and \eqref{eq2.5}, it would be expected that $\cL$ has a kernel in $L^2(m)$, and the real part of all other eigenvalues of $\cL$ cannot be larger than $-1$ with associated eigen-functions with zero vertical momentum. 

\subsection{Uniqueness of the kernel}
In this section, we prove that if the kernel of $\mathcal{L}$ exists, then the associated eigen-space must be of dimension one. 
\begin{proof}[Uniqueness part of Proposition \ref{prop3.1}]
  For the uniqueness part, we assume by contradiction that there exist two such functions $\bar{\Omega}_1$ and $\bar{\Omega}_2$. Then $\Omega(X,Y)\eqdefa \bar{\Omega}_1-\bar{\Omega}_2$ is a steady solution of $\p_\tau\Omega=\cL\Omega$, and $\int_{\R^2_+} Y\Omega(X,Y)dXdY=0$. Therefore, we derive from \eqref{eq2.5a} that 
    $$
        \|\Omega\|_{L^2(6)}=\|e^{\tau\cL}\Omega\|_{L^2(6)}\leq C_\delta e^{-\tau}\|\Omega\|_{L^2(6)}.
    $$
    As $\tau$ goes to infinity, the only possibility for the above inequality to hold is that $\Omega=0$. This means $\bar{\Omega}_1=\bar{\Omega}_2$, which finishes the proof for uniqueness.
\end{proof}

\subsection{Formula in the Fourier variables}

In this section, we construct the kernel $\bar\Omega\in L^2(m)$ of $\mathcal{L}$. Let us recall the equation of the kernel
\begin{align*}
    \left\{\begin{aligned}
 &   \p_Y^2\bar\Omega+\frac32 X\p_X\bar\Omega+\frac12 Y\p_Y\bar\Omega+\frac52\bar\Omega-Y\p_X\bar\Omega=0,\quad X\in \mathbb{R}, \ Y>0\\
& \bar\Omega(X,0)=0,\quad \lim_{Y\to \infty}\bar\Omega=0. 
        \end{aligned}
        \right.
\end{align*}
We define the central symmetric extension of $\bar{\Omega}$ to the entire space by setting $\bar\Omega^-(X, Y) = -\bar\Omega(-X, -Y)$ for $Y<0$. It follows from the symmetry that $\bar{\Omega}^-$ solves the same equation: 
\begin{align*}
    \left\{\begin{aligned}
 &   \p_Y^2\bar\Omega^-+\frac32 X\p_X\bar\Omega^-+\frac12 Y\p_Y\bar\Omega^-+\frac52\bar\Omega^--Y\p_X\bar\Omega^-=0,\quad X\in \mathbb{R}, \  Y<0\\
& \bar\Omega^-(X,0)=0,\quad \lim_{Y\to \pm\infty}\bar\Omega^-=0. 
        \end{aligned}
        \right.
\end{align*}
Let $\tilde\Omega=\left\{\begin{aligned}
    &\bar{\Omega}, \quad &&Y>0,\\
    &0,\quad &&Y=0,\\
    &\bar{\Omega}^-,\quad &&Y<0.
\end{aligned}\right. $
 It is easy to verify that in the sense of distribution, $\tilde{\Omega}$ satisfies 
\begin{align}\label{eq: extend kernel}
    \left\{\begin{aligned}
 &   \p_Y^2\tilde\Omega+\frac32 X\p_X\tilde\Omega+\frac12 Y\p_Y\tilde\Omega+\frac52\tilde\Omega-Y\p_X\tilde\Omega=G(X)\delta_0(Y),\quad (X,Y)\in \mathbb{R}^2\\
& \tilde\Omega(X,0)=0,\quad \lim_{|Y|\to \infty}\tilde\Omega=0. 
        \end{aligned}
        \right.
\end{align}
where $G(X)=\lim\limits_{Y\to 0+}\partial_Y\bar{\Omega}(X,Y)-\lim\limits_{Y\to 0-}\partial_Y\bar{\Omega}^-(X,Y)= \lim\limits_{Y\to 0+}\bigl(\partial_Y\bar{\Omega}(X,Y)-\p_Y \bar{\Omega}(-X,Y)\bigr)$ is an odd function representing the jump of $\p_Y\tilde{\Omega}$ near the boundary. 

Let us introduce the Fourier transform of $\tilde\Omega$ as
$$
f(k,\eta) \eqdefa \int_{\R^2} \tilde\Omega (X,Y)e^{-ikX-i\eta Y} dX dY.
$$
We note that the Fourier transform should be in the sense of distribution. Later, once we show $f\in H^m(\mathbb{R}^2)$, the Fourier transform here is in the classic sense. 

In $(k,\eta)$ variables, the equation becomes
\begin{subequations}
    
\begin{equation}\label{eq3.1}
    \f32k\p_k f+(\f12\eta-k)\p_\eta f+(\eta^2-\f12)f=g,
\end{equation}
where $g(k)$ is a function depending only on $k$ which is the Fourier transform of $G(X)$, namely, 
\begin{align*}
    g(k)=\int_{\mathbb{R}}G(X)e^{-ikX}dX. 
\end{align*}
Here, our solution has to satisfy the boundary condition $\bar\Omega(X,0)=0$, which is equivalent to 
\begin{equation}\label{eq3.2}
    \int_{\R} f(k,\eta) d\eta =0, \qquad \forall k\in \R.
\end{equation}
From the central symmetry, one also have
\begin{equation}\label{sym f g}
    f(k,\eta)=-f(-k,-\eta) \andf g(k)=-g(-k).
\end{equation}

Together with $\|\bar\Omega\|_{L^2(m)}=C\|f\|_{H^m}$ from the Plancherel equality, our goal is to search for $(f,g)$ solving \eqref{eq3.1} and \eqref{eq3.2} such that $f$ belongs to $H^m$. 
The smoothness of $(f,g)$ implies $g(0)=0$ and 
\begin{equation}\label{eq3.3}
    \p_\eta  f(0,0)=-i\int_{\R^2} Y\tilde\Omega(X,Y)dXdY=-2i \int_{\R^2_+} Y\bar{\Omega}(X,Y)dXdY=-2i,
\end{equation}
where we used \eqref{sym f g} and normalize $\int_{\R^2_+} Y\bar\Omega(X,Y)dXdY=1$. 
\end{subequations}

To find smooth $(f,g)\in H^m(\R^2)\times H^m(\R)$ satisfying \eqref{eq3.1}-\eqref{eq3.3}, we have the following two lemmas to give a solution formula: 
\begin{lem}\label{lem3.11}
     Let $g(k)$ be an odd function. Then,
    \begin{equation}\label{eq3.7}
    \begin{aligned}
        f(k,\eta) 
    =&-2i(k+\eta) e^{-(k+\eta)^2+k(k+\eta)-\f13k^2}+\f23\int_0^{|k|} (\f{|k|}\ell)^\f13 \exp \Bigl( -(k+\eta)^2(1-(\f{\ell}{|k|})^\f23) \\
        &\qquad\qquad\qquad+k(k+\eta)(1-(\f{\ell}{|k|})^\f43)-\f13(k^2-\ell^2)\Bigr) \f{g\bigl(\sgn(k)\ell\bigr)}{\ell}d\ell,
    \end{aligned}
\end{equation}
solves \eqref{eq3.1}, \eqref{sym f g} and \eqref{eq3.3}. 
\end{lem}
To make $f(k,\eta)$ satisfy \eqref{eq3.2}, we have the following lemma. 
\begin{lem}\label{lem3.12}
    Assume that $g(k)=-g(-k)=k\,h(k^{2/3})$ for $k\geq0$, where $h(s)$ is a function on $\R_+$ given by the following convolution equation
\begin{equation}\label{eq3.9}
    \int_0^\tau (\tau-s)^{-\f12}e^{-\f1{12}(\tau-s)^3} h(s)ds =i \tau^\f12 e^{-\f1{12}\tau^3},\qquad\forall \tau>0.
\end{equation}
Then, $f(k,\eta)$ defined by \eqref{eq3.7} satisfies \eqref{eq3.2}.
\end{lem}

\begin{proof}[Proof of Lemma \ref{lem3.11}] 
One can easily check by direct computations that the formula \eqref{eq3.7} satisfies \eqref{eq3.1} and \eqref{eq3.3}. Here, we present a more inspiring proof by first solving the transport equation \eqref{eq3.1} and then using $g(0)=0$ and \eqref{eq3.3} to simplify the expression. In fact, we shall prove that given an odd function $g(k)$, the only possible function $f\in H^m(\R^2)$ solving \eqref{eq3.1}, \eqref{sym f g} and \eqref{eq3.3} are given by our formula \eqref{eq3.7}.

To solve \eqref{eq3.1}, we first introduce the characteristic lines $\{ (k+\eta)^3=Ck \}$. Due to the central symmetry, we first consider the regime $k>0$.

Denote $k(t,k_0)=e^{\f32 t}k_0$ and $\eta(t,k_0,\eta_0)=e^{\f12 t}(k_0+\eta_0)-e^{\f32t}k_0$ which solves
$$
    \f{d}{dt}k=\f32 k\andf \f{d}{dt}\eta=\f12\eta-k.
$$
Therefore, we deduce from \eqref{eq3.1} that
$$
    \f{d}{dt}f\bigl(k(t,k_0),\eta(t,k_0,\eta_0)\bigr)=\bigl(\f12-\eta^2(t,k_0,\eta_0)\bigr)f\bigl(k(t,k_0,),\eta(t,k_0,\eta_0)\bigr)+g\bigl(k(t,k_0)\bigr).
$$
By solving the ODE, one has the following solution formula
\begin{equation}\notag
    f\bigl(k(t),\eta(t)\bigr) = e^{\f12 t -\int_0^t \eta^2(s)ds} \Bigl( f(k_0,\eta_0)+\int_0^t e^{-\f12s +\int_0^s \eta^2(s')ds'} g(k(s))ds\Bigr).
\end{equation}
By computing out the integration:
$$
    \int_0^t \eta^2(s)ds=(k_0+\eta_0)^2 (e^t-1)-k_0(k_0+\eta_0)(e^{2t}-1)+\f13 k_0^2(e^{3t}-1),
$$
we arrive at
\begin{equation}\label{eq3.4}
    \begin{aligned}
        f\bigl(k(t),\eta(t)\bigr) = &e^{\f12 t -\bigl(k(t)+\eta(t)\bigr)^2+k(t)\bigl(k(t)+\eta(t)\bigr)-\f13k^2(t)} \Bigl( e^{(k_0+\eta_0)^2-k_0(k_0+\eta_0)+\f13k_0^2)} f(k_0,\eta_0)\\
        &\qquad+\int_0^t e^{-\f12s +(k_0+\eta_0)^2e^s-k_0(k_0+\eta_0)e^{2s}+\f13k_0^2e^{3s}}g(e^{\f32s}k_0)ds\Bigr),\\
        =&e^{\f12 t -\bigl(k(t)+\eta(t)\bigr)^2+k(t)\bigl(k(t)+\eta(t)\bigr)-\f13k^2(t)} \Bigl( e^{(k_0+\eta_0)^2-k_0(k_0+\eta_0)+\f13k_0^2)} f(k_0,\eta_0)\\
        &\qquad+\f23\int_1^{k(t)} (\f{k_0}\ell)^\f13 e^{(k_0+\eta_0)^2 k_0^{-\f23}\ell^\f23-k_0^\f23(k_0+\eta_0)\ell^\f43+\f13 \ell^2}g(\ell)\f{d\ell}{\ell}\Bigr),
    \end{aligned}
\end{equation}
where we changed the variables from $s$ to $\ell=e^{\f32 s}k_0$.

Then, we notice that every characteristic line $\eta=Ck^\f13-k$ crosses the vertical line $k=1$ exactly once, and therefore we can choose $k_0=1$. By replacing $(k_0,\eta_0,t)$ in terms of $(k(t),\eta(t))$ by
$$
    k_0=1,\qquad \eta_0=\f{k+\eta}{k^\f13} -1 \andf t= \f23\ln k,
$$
we can reformulate \eqref{eq3.4} to 
\begin{equation}\label{eq3.5}
    \begin{aligned}
        f(k,\eta) 
        =&k^\f13 e^{-(k+\eta)^2+k(k+\eta)-\f13k^2} \Bigl( \exp \bigl((\f{k+\eta}{k^\f13})^2-\f{k+\eta}{k^\f13}+\f13\bigr) f(1,\f{k+\eta}{k^\f13}-1)\\
        &\qquad+\f23\int_1^{k} \ell^{-\f13} \exp \bigl((\f{k+\eta}{k^\f13})^2\ell^\f23-\f{k+\eta}{k^\f13}\ell^\f43+\f13\ell^2\bigr) \f{g(\ell)}{\ell}d\ell\Bigr).
    \end{aligned}
\end{equation}
Here, $f(1,\cdot)$ and $g(\cdot)$ should be carefully chosen to make $f$ smooth enough and satisfy \eqref{eq3.3}.

Noticing from $g(k)=-g(-k)$ that $g(0)=0$, the following integration 
$$
    \int_0^{1} \ell^{-\f13} \exp \bigl((\f{k+\eta}{k^\f13})^2\ell^\f23-\f{k+\eta}{k^\f13}\ell^\f43+\f13\ell^2\bigr) \f{g(\ell)}{\ell}d\ell
$$
has no singularity and can be viewed as a function of $\f{k+\eta}{k^\f13}$. Therefore, we introduce a new function 
\begin{align*}
    F(\f{k+\eta}{k^\f13}) =&\exp \bigl((\f{k+\eta}{k^\f13})^2-\f{k+\eta}{k^\f13}+\f13\bigr) f(1,\f{k+\eta}{k^\f13}-1) \\
    &-\f23 \int_0^{1} \ell^{-\f13} \exp \bigl((\f{k+\eta}{k^\f13})^2\ell^\f23-\f{k+\eta}{k^\f13}\ell^\f43+\f13\ell^2\bigr) \f{g(\ell)}{\ell}d\ell,
\end{align*}
which together with \eqref{eq3.5} gives a simplified formula
\begin{equation}\label{eq3.6}
    \begin{aligned}
        f(k,\eta) 
        =&k^\f13 e^{-(k+\eta)^2+k(k+\eta)-\f13k^2}  F(\f{k+\eta}{k^\f13})\\
        &+\f23\int_0^{k} (\f{k}{\ell})^{\f13} \exp \bigl(-({k+\eta})^2(1-(\f{\ell}{k})^\f23)+k(k+\eta)(1-(\f{\ell}{k})^\f43)-\f13(k^2-\ell^2)\bigr) \f{g(\ell)}{\ell}d\ell.
    \end{aligned}
\end{equation}

Now, we turn to determine $F$ from the smoothness of $f$ and \eqref{eq3.3}. First, one derives from \eqref{eq3.3} and \eqref{eq3.6} that for all $\xi\in \R$,
$$
    -2i=\p_\eta f(0,0)=\lim_{k\rightarrow0+} \p_\eta f(k,\xi k^\f13-k)=F'(\xi),
$$
and therefore, the only possibility for the choice $F$ is $F(\xi)=-2i\xi+C_F$. By taking $F(\xi)=-2i\xi+C_F$, \eqref{eq3.6} becomes 
\begin{align*}
    &f(k,\eta) 
    =-2i(k+\eta) e^{-(k+\eta)^2+k(k+\eta)-\f13k^2} +C_F k^\f13 e^{-(k+\eta)^2+k(k+\eta)-\f13k^2}  +I(k,\eta),
\end{align*}
where $I$ denote
\begin{align*}
    &I(k,\eta)=\f23\int_0^{k} (\f{k}\ell)^\f13 e^{ -(k+\eta)^2(1-(\f{\ell}{k})^\f23)+k(k+\eta)(1-(\f{\ell}{k})^\f43)-\f13(k^2-\ell^2)} \f{g(\ell)}{\ell}d\ell.
\end{align*}
Then, we fix $\eta\in\R$ and consider the limit $k\rightarrow0+$ for $\p_k f$. The exponential weight satisfies
\begin{equation}\label{eq3.6a}
    \begin{aligned}
        &-(k+\eta)^2(1-(\f{\ell}{k})^\f23)+k(k+\eta)(1-(\f{\ell}{k})^\f43)-\f13(k^2-\ell^2)\\
        =&-(1-(\f{\ell}k)^\f23)\Bigl(\eta^2-\bigl(1-(\f{\ell}k)^\f23\bigr)k\eta+\f13k^2 \bigl(1-(\f{\ell}k)^\f23\bigr)^2 \Bigr)\\
        \leq &-\f14(1-(\f{\ell}k)^\f23)\eta^2\leq 0,
\end{aligned}
\end{equation}
from which, we derive that as $k\rightarrow0+$,
\begin{align*}
    |\p_k I| \leq  \f23|\f{g(k)}k| + C_\eta\int_0^k \Bigl(k^{-\f23}\ell^{-\f13}+k^\f13\ell^{-\f13}+k\Bigr)d\ell \leq C_{g,\eta}.
\end{align*}
Together with the regularity of terms from $F$, we find that as $k\rightarrow0+$,
$$
    \p_k f(k,\eta)=\f13 C_F k^{-\f23}e^{-\eta^2}+O(1).
$$
Since our goal is to make $f$ belong to $H^m$, one should have $\p_k f(k,\eta)\in L^2$. Therefore, the only possibility is that the term of order $k^{-\f23}$, which is not square integrable, vanishes, which means $C_F=0$.
 Finally, we plug $F(\xi)=-2i\xi$ into\eqref{eq3.6} and arrive at the formula \eqref{eq3.7} for $k>0$.
 
 When $k<0$, we use the symmetry \eqref{sym f g} to extend \eqref{eq3.7} to $k<0$. This finishes the proof.
\end{proof}

\begin{proof}[Proof of Lemma \ref{lem3.12}] Due to the symmetry, we again only need to consider the case $k>0$. 
One directly computes
\begin{align*}
    &\int_\R -2i(k+\eta) e^{-(k+\eta)^2+k(k+\eta)-\f13k^2}d\eta 
    =-2i\int_\R (\eta+\f{k}2+\f{k}2)e^{-(\eta+\f{k}2)^2-\f{k^2}{12}}d\eta =-{i\sqrt{\pi}} ke^{-\f{k^2}{12}}. 
\end{align*}
Then, we change the order of integrations and use the equality in \eqref{eq3.6a} to derive 
\begin{align*}
    &\int_\R\f23\int_0^{k} (\f{k}\ell)^\f13 \exp \bigl( -(k+\eta)^2(1-(\f{\ell}{k})^\f23)+k(k+\eta)(1-(\f{\ell}{k})^\f43)-\f13(k^2-\ell^2)\bigr) \f{g(\ell)}{\ell}d\ell d\eta\\
    =& \f23 \int_0^{k} (\f{k}\ell)^\f13 \int_\R \exp \Bigl( -(1-(\f{\ell}k)^\f23)\bigl(\eta-\bigl(1-(\f{\ell}k)^\f23\bigr)\f{k}2\bigr)^2\Bigr)d\eta e^{-\f{1}{12}(k^\f23-\ell^\f23)^3} \f{g(\ell)}{\ell}d\ell \\
    =& \f23 \sqrt{\pi}\int_0^{k} (\f{k}\ell)^\f13 \bigl(1-(\f{\ell}k)^\f23\bigr)^{-\f12} e^{-\f{1}{12}(k^\f23-\ell^\f23)^3} \f{g(\ell)}{\ell}d\ell.
\end{align*}
Combining the above two computations, we deduce from \eqref{eq3.2} that $g$ solves
\begin{equation}\notag
    \int_0^{k} (k^\f23-\ell^\f23)^{-\f12} e^{-\f{1}{12}(k^\f23-\ell^\f23)^3} g(\ell)\ell^{-\f43}d\ell=\f{3i}2 k^\f13e^{-\f{k^2}{12}}.
\end{equation}

By introducing the variable transform $s=\ell^\f23$, we find that $g$ solves the following convolution-type integral equation
$$
    \int_0^{k^\f23}(k^\f23-s)^{-\f12}e^{-\f1{12}(k^\f23-s)^3}\f{g(s^\f32)}{s^\f32}ds =i k^\f13 e^{-\f1{12} k^2},\qquad \forall k>0.
$$
Therefore,we denote $\tau=k^\f23$ and $h(s)=\f{g(s^\f32)}{s^\f32}$ and arrive at \eqref{eq3.9}. This finishes the proof.
\end{proof}
\begin{rmk}
      Here we give some ideas to understand the solution formula \eqref{eq3.7} and \eqref{eq3.9}. The first term $-2i(k+\eta)e^{-(k+\eta)^2+k(k+\eta)-\f13k^2}$ is exactly the Fourier expression of the second eigenfunction $-2(\p_X+\p_Y)G_L$ on the whole plane $\R^2$. However, on the half plane $\R^2_+$, it doesn't satisfy the Dirichlet boundary condition at $y=0$. Our method is to find another group of solutions to \eqref{eq3.1} which are expressed by the integration of the Neumann boundary values $g$ in the second line of \eqref{eq3.7}, and then we choose a suitable $g$ as in \eqref{eq3.9} to make the Dirichlet boundary vanish.
\end{rmk}

\subsection{The regularity of the kernel}

In this subsection, we aim to use the formula \eqref{eq3.7} to prove that $f$ belongs to $H^m$, with some high regularity assumptions on the boundary value $g$ to be checked in the next subsection. 

\begin{lem}\label{lem3.1}
    {  Let $f$ be given by \eqref{eq3.7}. If $\f{g(\ell)}{\ell}$ is smooth enough on $l\in [0,+\oo)$, then it holds
    \begin{equation}\label{eq3.10}
        \begin{aligned}
            &\|f(k,\eta)\|_{H^m(k>0)}
            \leq  C_m \Bigl(1+ \|(\ell^{-\f16}+\ell^\f12) \f{g(\ell)}\ell\|_{L^2_\ell(\R_+)}+\sum_{j=1}^{m} \|(\ell^{-\f16}+\ell^{\f56})\p_\ell^{j}(\f{g(\ell)}{\ell})\|_{L^2_\ell(\R_+)}  \Bigr). 
        \end{aligned}
    \end{equation}
    }
\end{lem}
\begin{proof}
    It is obvious that $-2i(k+\eta)e^{-(k+\eta)^2+k(k+\eta)-\f13k^2}\in H^\oo$, and therefore we only need to consider the regularity for the integration of $g$. Let us introduce the notation of the following kernel with a variable $\xi\in(0,1)$:
    \begin{equation}\label{eq3.11}
        K(k,\eta,\xi)\eqdefa \xi^{-\f13} \exp \bigl( -(k+\eta)^2(1-\xi^\f23)+k(k+\eta)(1-\xi^\f43)-\f13k^2(1-\xi^2)\bigr).
    \end{equation}
    It remains to show 
    \begin{equation}\label{eq3.11a}
        \begin{aligned}
            &\|\int_0^k K(k,\eta,\f{\ell}k)\f{g(\ell)}{\ell}d\ell\|_{H^m(k>0)}\\
            \leq & C_m \Bigl( \|(\ell^{-\f16}+\ell^\f12) \f{g(\ell)}\ell\|_{L^2_\ell(\R_+)}+\sum_{j=1}^{m} \|(\ell^{-\f16}+\ell^{\f56})\p_\ell^{j}(\f{g(\ell)}{\ell})\|_{L^2_\ell(\R_+)}  \Bigr). 
        \end{aligned}
    \end{equation}

    \begin{itemize}
        \item[{\bf Step 1:}] The derivation rules.
    \end{itemize}

    For the $\p_\eta$ derivatives, it can be easily computed that 
    \begin{equation}\label{eq3.12}
        \p_\eta^\beta \int_0^k K(k,\eta,\f{\ell}k)\f{g(\ell)}{\ell}d\ell=\int_0^k \p_2^\beta K(k,\eta,\f{\ell}k)\f{g(\ell)}{\ell}d\ell.
    \end{equation}

    For the $\p_k$ derivatives, we shall prove by induction that 
    \begin{equation}\label{eq3.13}
        \p_k^\beta \int_0^k K(k,\eta,\f{\ell}k)\f{g(\ell)}{\ell}d\ell
        =\sum_{j=0}^\beta \binom{\beta}{j} \int_0^k \p_1^{j}K(k,\eta,\f\ell{k})(\f{\ell}{k})^{\beta-j}\f{\p_\ell^{\beta-j}g(\ell)}{\ell} d\ell.
    \end{equation}
    Obviously, when $\beta=0$, \eqref{eq3.13} is a trivial equality. We now assume that \eqref{eq3.13} is true for $\beta=N$, and show that the case $\beta=N+1$ holds.

    By directly computations, we see
    \begin{align*}
        &\p_k\int_0^k \p_1^{j}K(k,\eta,\f\ell{k})(\f{\ell}{k})^{N-j}\f{\p_\ell^{N-j}g(\ell)}{\ell} d\ell\\
        =&  \p_1^{j}K(k,\eta,1)\f{\p_k^{N-j}g(k)}{k} -\f{N-j}k\int_0^k \p_1^{j}K(k,\eta,\f\ell{k})(\f{\ell}{k})^{N-j}\f{\p_\ell^{N-j}g(\ell)}{\ell} d\ell\\
        &+\int_0^k \p_1^{j+1}K(k,\eta,\f\ell{k})(\f{\ell}{k})^{N-j}\f{\p_\ell^{N-j}g(\ell)}{\ell} d\ell-\int_0^k \f{\ell}{k^2}\p_1^{j}\p_3 K(k,\eta,\f\ell{k})(\f{\ell}{k})^{N-j}\f{\p_\ell^{N-j}g(\ell)}{\ell} d\ell.
    \end{align*}
    For the last integral above, we view $\f{\ell}{k^2}\p_1^{j}\p_3 K(k,\eta,\f\ell{k})=\f{\ell}k\f{d}{d\ell}\p_1^j K(k,\eta,\f\ell{k})$ and do integration by parts to find
    \begin{align*}
        &-\int_0^k \f{\ell}{k^2}\p_1^{j}\p_3 K(k,\eta,\f\ell{k})(\f{\ell}{k})^{N-j}\f{\p_\ell^{N-j}g(\ell)}{\ell} d\ell\\
        =& -\p_1^{j}K(k,\eta,1)\f{\p_k^{N-j}g(k)}{k} +\int_0^k \p_1^{j} K(k,\eta,\f\ell{k})\f{d}{d\ell}\Bigl((\f{\ell}{k})^{N-j+1}\f{\p_\ell^{N-j}g(\ell)}{\ell}\Bigr) d\ell\\
        =& -\p_1^{j}K(k,\eta,1)\f{\p_k^{N-j}g(k)}{k} +\int_0^k \p_1^{j} K(k,\eta,\f\ell{k})(\f{\ell}{k})^{N-j+1}\f{\p_\ell^{N-j+1}g(\ell)}{\ell} d\ell\\
        & +\f{N-j}k\int_0^k \p_1^{j} K(k,\eta,\f\ell{k})(\f{\ell}{k})^{N-j}\f{\p_\ell^{N-j}g(\ell)}{\ell} d\ell,
    \end{align*}
    where we used that $\ell^{N-j-1}\p_\ell^{N-j}g(\ell)=\ell^{N-j}\p_\ell^{N-j}(\f{g(\ell)}\ell)+(N-j)\ell^{N-j-1}\p_\ell^{N-j-1}(\f{g(\ell)}{\ell})$ has no singularity near $\ell=0$ and $\ell \p_1K(k,\eta,\f\ell{k})$ vanishes at $\ell=0$.
    By adding together the above two equalities, we arrive at
    \begin{align*}
        &\p_k\int_0^k \p_1^{j}K(k,\eta,\f\ell{k})(\f{\ell}{k})^{N-j}\f{\p_\ell^{N-j}g(\ell)}{\ell} d\ell\\
        =& \int_0^k \p_1^{j+1}K(k,\eta,\f\ell{k})(\f{\ell}{k})^{N-j}\f{\p_\ell^{N-j}g(\ell)}{\ell} d\ell+\int_0^k \p_1^{j} K(k,\eta,\f\ell{k})(\f{\ell}{k})^{N-j+1}\f{\p_\ell^{N-j+1}g(\ell)}{\ell} d\ell.
    \end{align*}
    Finally, we conclude that 
    \begin{align*}
        &\p_k^{N+1} \int_0^k K(k,\eta,\f{\ell}k)\f{g(\ell)}{\ell}d\ell
        =\sum_{j=0}^N \binom{N}{j}\p_k \int_0^k \p_1^{j}K(k,\eta,\f\ell{k})(\f{\ell}{k})^{N-j}\f{\p_\ell^{N-j}g(\ell)}{\ell} d\ell\\
        =& \sum_{j=0}^N \binom{N}{j}\Bigl(\int_0^k \p_1^{j+1}K(k,\eta,\f\ell{k})(\f{\ell}{k})^{N-j}\f{\p_\ell^{N-j}g(\ell)}{\ell} d\ell\\
        &\qquad\qquad\qquad\qquad+\int_0^k \p_1^{j} K(k,\eta,\f\ell{k})(\f{\ell}{k})^{N-j+1}\f{\p_\ell^{N-j+1}g(\ell)}{\ell} d\ell\Bigr)\\
        =&\sum_{j=0}^{N+1} (\binom{N}{j}+\binom{N}{j-1})\int_0^k \p_1^{j} K(k,\eta,\f\ell{k})(\f{\ell}{k})^{N-j+1}\f{\p_\ell^{N-j+1}g(\ell)}{\ell} d\ell\\
        =&\sum_{j=0}^{N+1} \binom{N+1}{j}\int_0^k \p_1^{j} K(k,\eta,\f\ell{k})(\f{\ell}{k})^{N+1-j}\f{\p_\ell^{N+1-j}g(\ell)}{\ell} d\ell,
    \end{align*}
    which gives \eqref{eq3.13} for $\beta=N+1$. By induction, we have already proven \eqref{eq3.13}. 

    \begin{itemize}
        \item[{\bf Step 2:}] Estimates for derivatives of $K(k,\eta,\xi)$.
    \end{itemize}

    As in \eqref{eq3.6a}, we write
    $$
        K(k,\eta,\xi)=\xi^{-\f13}\exp \bigl( -\f13(1-\xi^\f23)^3k^2-(1-\xi^\f23)^2k\eta-(1-\xi^\f23)\eta^2\bigr).
    $$
    By introducing new variables $\tilde{k}=(1-\xi^\f23)^\f32k$ and $\tilde{\eta}=(1-\xi^\f23)^\f12\eta$, we denote
    $$
        K(k,\eta,\xi)=\tilde{K}(\tilde{k},\tilde{\eta},\xi)\eqdefa \xi^{-\f13} \exp \bigl(-\f13 \tilde{k}^2+\tilde{k}\tilde{\eta}-\tilde{\eta}^2\bigr).
    $$
    Thanks to $-\f13 \tilde{k}^2+\tilde{k}\tilde{\eta}-\tilde{\eta}^2\leq -\f1{16}(\tilde{k}^2+\tilde{\eta}^2)$, we know that for any $a,b>0$,
    $$
        \p_{\tilde{k}}^a \p_{\tilde{\eta}}^b \tilde{K}(\tilde{k},\tilde{\eta},\xi)=\xi^{-\f13} P_{a,b}(\tilde{k},\tilde{\eta})\exp \bigl(-\f1{16}(\tilde{k}^2+\tilde{\eta}^2)\bigr)\leq C_{a,b} \xi^{-\f13} \exp \bigl(-\f1{20}(\tilde{k}^2+\tilde{\eta}^2)\bigr).
    $$
    Therefore, the change of variables gives us 
    \begin{equation}\label{eq3.14}
        \begin{aligned}
            \p_k^a\p_\eta^b K(k,\eta,\xi)&=(1-\xi^\f23)^\f{3a+b}2 \p_{\tilde{k}}^a \p_{\tilde{\eta}}^b \tilde{K}(\tilde{k},\tilde{\eta},\xi)\leq C_{a,b} \xi^{-\f13}(1-\xi^\f23)^\f{3a+b}2 \exp \bigl(-\f1{20}(\tilde{k}^2+\tilde{\eta}^2)\bigr)\\
            &\leq C_{a,b} \xi^{-\f13}(1-\xi^\f23)^\f{3a+b}2 \exp \bigl(-\f1{20}(1-\xi^\f23)^3k^2-\f1{20}(1-\xi^\f23)\eta^2\bigr)\\
            &\leq C_{a,b}  \exp \bigl(-\f1{20}(1-\xi^\f23)^3k^2-\f1{20}(1-\xi^\f23)\eta^2\bigr),
        \end{aligned}
    \end{equation}
    where we used $0<\xi<1$.

    \begin{itemize}
        \item[{\bf Step 3:}] Proof of \eqref{eq3.11a}.
    \end{itemize}

    First, we consider the $L^2$ norm. By applying \eqref{eq3.14} with $a=b=0$, one has
    \begin{align*}
        \|\int_0^k K(k,\eta,\f{\ell}k)\f{g(\ell)}{\ell}d\ell\|_{L^2(k>0)} 
        &\leq C \|\int_0^k(\f{\ell}k)^{-\f13} e^{-\f1{20}(k^\f23-\ell^\f23)^3} \|e^{-\f1{20}\bigl(1-(\f{\ell}k)^\f23\bigr)\eta^2}\|_{L^2_\eta}\f{g(\ell)}{\ell}d\ell \|_{L^2_k(\R_+)}\\
        &\leq C \|\int_0^k(\f{\ell}k)^{-\f13}\bigl(1-(\f{\ell}k)^\f23\bigr)^{-\f14} e^{-\f1{20}(k^\f23-\ell^\f23)^3} \f{g(\ell)}{\ell}d\ell \|_{L^2_k(\R_+)}
    \end{align*}
    Next, we introduce the new variables $\tau=k^\f23$ and $s=\ell^\f23$ to write
    \begin{align*}
        &\|\int_0^k(\f{\ell}k)^{-\f13}\bigl(1-(\f{\ell}k)^\f23\bigr)^{-\f14} e^{-\f1{20}(k^\f23-\ell^\f23)^3} \f{g(\ell)}{\ell}d\ell \|_{L^2_k(\R_+)}\\
        =& (\f32)^\f32 \|\tau^\f14\int_0^\tau (\f\tau{s})^{\f12} (\f{\tau}{\tau-s})^\f14 e^{-\f1{20}(\tau-s)^3} \f{g(s^\f32)}{s^\f32}\sqrt{s}ds\|_{L^2_\tau(\R_+)}\\
        =& (\f32)^\f32 \|\tau\int_0^\tau ({\tau-s})^{-\f14} e^{-\f1{20}(\tau-s)^3} \f{g(s^\f32)}{s^\f32}ds\|_{L^2_\tau(\R_+)}.
    \end{align*} 
    One can view $\tau=(\tau-s)+s$ and use Young's inequality to estimate
    \begin{align*}
        &\|\tau\int_0^\tau ({\tau-s})^{-\f14} e^{-\f1{20}(\tau-s)^3} \f{g(s^\f32)}{s^\f32}ds\|_{L^2_\tau(\R_+)}\\
        \leq & \|s^\f34 e^{-\f1{20}s^3}\|_{L^1_s(\R_+)} \|\f{g(s^\f32)}{s^\f32}\|_{L^2_s(\R_+)} + \|s^{-\f14} e^{-\f1{20}s^3}\|_{L^1_s(\R_+)} \|s\f{g(s^\f32)}{s^\f32}\|_{L^2_s(\R_+)}\\
        \leq & C \|(1+s) \f{g(s^\f32)}{s^\f32}\|_{L^2_s(\R_+)}.
    \end{align*}
    Combining the above estimates, we use $\|(1+s) \f{g(s^\f32)}{s^\f32}\|_{L^2_s(\R_+)}=(\f23)^\f12\|(\ell^{-\f16}+\ell^\f12) \f{g(\ell)}\ell\|_{L^2_\ell(\R_+)}$ to conclude 
    \begin{equation}\label{eq3.15}
        \|\int_0^k K(k,\eta,\f{\ell}k)\f{g(\ell)}{\ell}d\ell\|_{L^2(k>0)}\leq C \|(\ell^{-\f16}+\ell^\f12) \f{g(\ell)}\ell\|_{L^2_\ell(\R_+)}.
    \end{equation}

    By repeating exactly the same process together with \eqref{eq3.12}, one can prove
    \begin{equation}\label{eq3.16}
        \|\p_\eta^m\int_0^k K(k,\eta,\f{\ell}k)\f{g(\ell)}{\ell}d\ell\|_{L^2(k>0)}\leq C_m \|(\ell^{-\f16}+\ell^\f12) \f{g(\ell)}\ell\|_{L^2_\ell(\R_+)}.
    \end{equation}

    However, for the derivatives in $k$ direction, we shall carefully deal with the terms in \eqref{eq3.13}. When $j=m$, we derive again from the proof of \eqref{eq3.15} that
    \begin{equation}\label{eq3.17}
        \|\int_0^k \p_1^m K(k,\eta,\f{\ell}k)\f{g(\ell)}{\ell}d\ell\|_{L^2(k>0)}\leq C_m \|(\ell^{-\f16}+\ell^\f12) \f{g(\ell)}\ell\|_{L^2_\ell(\R_+)}.
    \end{equation}
    For the cases $0\leq j\leq m-1$, we apply \eqref{eq3.14} with $a=j$ and $b=0$ together with \eqref{eq3.13} to compute
    \begin{align*}
        &\|\int_0^k \p_1^j K(k,\eta,\f{\ell}k)(\f{\ell}k)^{m-j}\f{\p_\ell^{m-j}g(\ell)}{\ell}d\ell\|_{L^2(k>0)} \\
        \leq &C_j \|\int_0^k(\f{\ell}k)^{m-j-\f13} e^{-\f1{20}(k^\f23-\ell^\f23)^3} \|e^{-\f1{20}\bigl(1-(\f{\ell}k)^\f23\bigr)\eta^2}\|_{L^2_\eta}\f{\p_\ell^{m-j}g(\ell)}{\ell}d\ell \|_{L^2_k(\R_+)}\\
        \leq &C_j \|\int_0^k(\f{\ell}k)^{m-j-\f13}\bigl(1-(\f{\ell}k)^\f23\bigr)^{-\f14} e^{-\f1{20}(k^\f23-\ell^\f23)^3} \f{\p_\ell^{m-j}g(\ell)}{\ell}d\ell \|_{L^2_k(\R_+)}.
    \end{align*}
    In the new variables $\tau=k^\f23$ and $s=\ell^\f23$, we rewrite
    \begin{align*}
        &\|\int_0^k(\f{\ell}k)^{m-j-\f13}\bigl(1-(\f{\ell}k)^\f23\bigr)^{-\f14} e^{-\f1{20}(k^\f23-\ell^\f23)^3} \f{\p_\ell^{m-j}g(\ell)}{\ell}d\ell \|_{L^2_k(\R_+)}\\
        =& (\f32)^\f32 \|\tau^\f14\int_0^\tau (\f{s}\tau)^{\f32(m-j)-\f12} (\f{\tau}{\tau-s})^\f14 e^{-\f1{20}(\tau-s)^3} \f{(\p_\ell^{m-j}g)(s^\f32)}{s^\f32}\sqrt{s}ds\|_{L^2_\tau(\R_+)}\\
        \leq & (\f32)^\f32 \|\tau^{-\f12}\int_0^\tau ({\tau-s})^{-\f14} e^{-\f1{20}(\tau-s)^3} (\p_\ell^{m-j}g)(s^\f32)ds\|_{L^2_\tau(\R_+)},
    \end{align*}
    where we used $(\f{s}\tau)^{\f32(m-j)-\f12}\leq \f{s}\tau$ due to $m-j\geq 1$.

    When $\tau\leq 1$, we use $\tau^{-\f12}\leq \tau^{-\f38}s^{-\f18}$ and Holder's inequality to get
    \begin{align*}
        &\|\tau^{-\f12}\int_0^\tau ({\tau-s})^{-\f14} e^{-\f1{20}(\tau-s)^3} (\p_\ell^{m-j}g)(s^\f32)ds\|_{L^2_\tau((0,1])}\\
        \leq &\|\tau^{-\f38}\|_{L^2_\tau((0,1])} \|s^{-\f14}e^{-\f1{20}s^3}\|_{L^3_s(\R_+)}\|s^{-\f18}\|_{L^6_s((0,1])} \|(\p_\ell^{m-j}g)(s^\f32)\|_{L^2_s(\R_+)}\\
        \leq &C \|(\p_\ell^{m-j}g)(s^\f32)\|_{L^2_s(\R_+)}.
    \end{align*}
    On the other hand, for $\tau\geq 1$, we deduce from $\tau^{-\f12}\leq 1$ and Young's inequality that 
    \begin{align*}
        &\|\tau^{-\f12}\int_0^\tau ({\tau-s})^{-\f14} e^{-\f1{20}(\tau-s)^3} (\p_\ell^{m-j}g)(s^\f32)ds\|_{L^2_\tau([1,\oo))}\\
        \leq & \|s^{-\f14}e^{-\f1{20}s^3}\|_{L^1_s(\R_+)} \|(\p_\ell^{m-j}g)(s^\f32)\|_{L^2_s(\R_+)}\\
        \leq &C \|(\p_\ell^{m-j}g)(s^\f32)\|_{L^2_s(\R_+)}.
    \end{align*}
    Combining the above estimates, we use $\|(\p_\ell^{m-j}g)(s^\f32)\|_{L^2_s(\R_+)}=(\f23)^\f12\|\ell^{-\f16}\p_\ell^{m-j}g(\ell)\|_{L^2_\ell(\R_+)}$ to conclude that for $0\leq j\leq m-1$,
    \begin{equation}\label{eq3.18}
        \|\int_0^k \p_1^j K(k,\eta,\f{\ell}k)(\f{\ell}k)^{m-j}\f{\p_\ell^{m-j}g(\ell)}{\ell}d\ell\|_{L^2(k>0)}\leq C_m\|\ell^{-\f16}\p_\ell^{m-j}g(\ell)\|_{L^2_\ell(\R_+)}.
    \end{equation}

    Noticing $\p_\ell^{m-j}g(\ell)=\ell \p_\ell^{m-j}(\f{g(\ell)}\ell)+(m-j)\p_\ell^{m-j-1}(\f{g(\ell)}{\ell})$, we apply \eqref{eq3.13}, \eqref{eq3.17} and \eqref{eq3.18} to conclude that 
    \begin{equation}\label{eq3.19}
        \begin{aligned}
            &\|\p_k^m\int_0^k K(k,\eta,\f{\ell}k)\f{g(\ell)}{\ell}d\ell\|_{L^2(k>0)}
            \leq C_m \Bigl( \|(\ell^{-\f16}+\ell^\f12) \f{g(\ell)}\ell\|_{L^2_\ell(\R_+)}\\
            &\qquad\qquad\qquad+\sum_{j=0}^{m-1} \bigl(\|\ell^{\f56}\p_\ell^{m-j}(\f{g(\ell)}{\ell})\|_{L^2_\ell(\R_+)}  +\|\ell^{-\f16}\p_\ell^{m-j-1}(\f{g(\ell)}{\ell})\|_{L^2_\ell(\R_+)}\bigr)  \Bigr)\\
            \leq &C_m \Bigl( \|(\ell^{-\f16}+\ell^\f12) \f{g(\ell)}\ell\|_{L^2_\ell(\R_+)}+\sum_{j=1}^{m} \|(\ell^{-\f16}+\ell^{\f56})\p_\ell^{j}(\f{g(\ell)}{\ell})\|_{L^2_\ell(\R_+)}  \Bigr).
        \end{aligned}
    \end{equation}
    
Finally, we combine \eqref{eq3.15}, \eqref{eq3.16} and \eqref{eq3.19} to finish the proof of \eqref{eq3.11a}. This finishes the proof.
\end{proof}

\subsection{The regularity for the boundary value}\label{subsection3.3}
In this subsection, we check the smoothness of $\f{g(\ell)}{\ell}$ where $g$ is given by \eqref{eq3.9} in terms of $h$. Therefore, we rewrite the norms in \eqref{eq3.10} in $s=\ell^\f23$ variables for the unknown $h(s)=\f{g(s^\f32)}{s^\f32}$ as
\begin{equation}\label{eq3.20a}
    \begin{aligned}
    &\|(\ell^{-\f16}+\ell^\f12) \f{g(\ell)}\ell\|_{L^2_\ell(\R_+)}+\sum_{j=1}^{m} \|(\ell^{-\f16}+\ell^{\f56})\p_\ell^{j}(\f{g(\ell)}{\ell})\|_{L^2_\ell(\R_+)}  \\
    = & (\f32)^\f12 \Bigl( \|(1+s) h(s)\|_{L^2_s(\R_+)}+\sum_{j=1}^{m} \|(1+s^{\f32})(\f23s^{-\f12}\p_s)^{j}h(s)\|_{L^2_s(\R_+)}  \Bigr).
\end{aligned}
\end{equation}

Before studing these norms of $h(s)$, we denote $H_{\kappa}(\tau)=\tau^\kappa e^{-\f1{12}\tau^3}$ and turns \eqref{eq3.9} into
\begin{equation}\label{eq3.21a}
    \int_0^\tau H_{-\f12}(\tau-s) h(s)ds=i H_\f12(\tau),\qquad \forall \tau>0.
\end{equation}

By applying the Laplace transform, we get 
$$
    \ccL[H_{-\f12}](\lambda) \ccL[h](\lambda)=i \ccL[H_\f12](\lambda).
$$
Since $H_{-\f12}(\tau)$ is a strictly decreasing function on $\R_+$, $\cL[H_{-\f12}](\lambda)$ has no zero on the half complex plane $\{\text{Re} \lambda\geq0\}$. Therefore, we get a solution formula for the integral equation in the Laplacian variable:
\begin{equation}\label{eq3.21}
    \ccL[h](\lambda)=i \f{\ccL[H_\f12](\lambda)}{\ccL[H_{-\f12}](\lambda) }.
\end{equation}

First, we prove the following lemma showing that $h(s)$ belongs to the $(1+s^2)^2$-weighted $L^2$ space.

\begin{lem}\label{lem3.2}
    {  Let $h(s)$ solve the integral equation \eqref{eq3.9}. Then, $(1+s^2)h(s)\in L^2_s(\R_+)$.
    }
\end{lem}
\begin{proof}
    First, for the $L^2$ norm, we can apply Plancherel's equality to get 
    $$
        \|h\|_{L^2(\R_+)}\leq C \|\ccL[h]\|_{L^2(i\R)}.
    $$
    From the fact that $\ccL[H_{-\f12}](\lambda)$ has no zeros in $\text{Re} \lambda\geq 0$, and the continuity of $\ccL[H_\f12](\lambda)$ and $\ccL[H_{-\f12}](\lambda)$, we know from \eqref{eq3.21} that $\ccL[h](\lambda)$ is bounded on the half ball $\{ \text{Re}\lambda\geq 0,|\lambda|\leq N \}$. At infinity, we can use \eqref{eq:lem A} to find that for $|\arg \lambda|<\f23\pi-\delta$ and $\lambda>>1$,
    $$
        |\ccL[h](\lambda)|\leq C_\delta \f{|\lambda|^{-\f32}}{|\lambda|^{-\f12}}\leq C_\delta|\lambda|^{-1},
    $$
    which gives the $L^2$ integrability on the imaginal line $i\R$, by taking $\delta<\f\pi6$.

    For the $s^2$-weighted $L^2$ norm, we use the derivation property of the Laplace transform to compute from \eqref{eq3.21} that 
    \begin{align*}
        \ccL[s^2h(s)](\lambda)&=\p_\lambda^2\ccL[h](\lambda)=i \p_\lambda^2\f{\ccL[H_\f12](\lambda)}{\ccL[H_{-\f12}](\lambda) }=i \p_\lambda\Bigl(-\f{\ccL[H_\f32](\lambda)}{\ccL[H_{-\f12}](\lambda) }+(\f{\ccL[H_\f12](\lambda)}{\ccL[H_{-\f12}](\lambda) })^2\Bigr)\\
        &=i \Bigl(\f{\ccL[H_\f52](\lambda)}{\ccL[H_{-\f12}](\lambda) }-3\f{\ccL[H_\f32](\lambda)}{\ccL[H_{-\f12}](\lambda) }\f{\ccL[H_\f12](\lambda)}{\ccL[H_{-\f12}](\lambda) }+(\f{\ccL[H_\f12](\lambda)}{\ccL[H_{-\f12}](\lambda) })^3\Bigr),
    \end{align*}
    where we used the fact that $\p_\lambda \ccL[H_\kappa](\lambda)=-\ccL[H_{\kappa+1}](\lambda)$. Then, it's similar to estimate that for $\lambda\in i\R$,
    $$
        |\ccL[s^2h(s)](\lambda)| \leq C (1+|\lambda|)^{-3},
    $$
    which is $L^2$ integrable. This, together with Plancherel's equality, finishes the proof.
\end{proof}

\begin{rmk}
    {  By following this proof, one can easily find that $h$ belongs to any polynomial-weighted $L^2$ space.
    }
\end{rmk}

For the higher regularity norms, we first apply $\p_\tau $ derivatives to \eqref{eq3.21a} and derive the equations for $\p_s^m h$.

\begin{lem}\label{lem3.3}
    {  Let $h(s)$ solve the integral equation \eqref{eq3.9}. Then, for any $m\in \N$, there exist constants $(a^{3m}_j)_{j=0}^{2m}$, $(a^{3m+1}_j)_{j=0}^{2m}$ and $(a^{3m+2}_j)_{j=0}^{2m+1}$ such that $\p_s^{3m}h(0)=\f12 a_0^m$, $\p_s^{3m+1}h(0)=\p_s^{3m+2}h(0)=0$ and
    \begin{subequations}\label{eq:lem3.3}
        \begin{align}\label{eq:lem3.3a}
            &\int_0^\tau H_{-\f12}(\tau -s) \p_s^{3m} h(s)ds = \sum_{j=0}^{2m} a_j^{3m} H_{\f12+3j}(\tau),\\
            &\int_0^\tau H_{-\f12}(\tau -s) \p_s^{3m+1} h(s)ds = \sum_{j=0}^{2m} a_j^{3m+1} H_{\f52+3j}(\tau),\label{eq:lem3.3b}\\
            &\int_0^\tau H_{-\f12}(\tau -s) \p_s^{3m+2} h(s)ds = \sum_{j=0}^{2m+1} a_j^{3m+2} H_{\f32+3j}(\tau).\label{eq:lem3.3c}
        \end{align}
    \end{subequations}
    }
\end{lem}
\begin{proof}
    First, we show the following equality
    \begin{equation}\label{eq3.23}
        \p_\tau\int_0^\tau H_{\kappa}(\tau-s) f(s)ds= H_\kappa(\tau)f(0)+\int_0^{\tau} H_{\kappa}(\tau-s) \p_sf(s)ds.
    \end{equation}
    For any $\delta>0$, we do integration by parts to see
    \begin{align*}
        \p_\tau \int_0^{\tau-\delta} H_{\kappa}(\tau-s) f(s)ds &= H_\kappa(\delta)f(\tau-\delta)+\int_0^{\tau-\delta} H'_{\kappa}(\tau-s) f(s)ds\\
        &=H_\kappa(\delta)f(\tau-\delta)-\int_0^{\tau-\delta} \f{d}{ds} \bigl(H_{\kappa}(\tau-s) \bigr) f(s)ds\\
        &=H_\kappa(\tau)f(0)+\int_0^{\tau-\delta} H_{\kappa}(\tau-s) \p_sf(s)ds.
    \end{align*}
    By taking $\delta\rightarrow0+$, we finish the proof of \eqref{eq3.23}.

    By induction, we try to show \eqref{eq:lem3.3a}. When $m=0$, this equation reduce to \eqref{eq3.21a} with $a^0_0=i$. Now, we assume the case $m=M$ that 
    $$
        \int_0^\tau H_{-\f12}(\tau -s) \p_s^{3M} h(s)ds = \sum_{j=0}^{2M} a_j^{3M} H_{\f12+3j}(\tau),
    $$
    and show the equation for the case $m=M+1$.

    By applying $\p_\tau$ and using \eqref{eq3.23}, we find that 
    \begin{align*}
        H_{-\f12}(\tau)\p_s^{3M}h(0)&+\int_0^\tau H_{-\f12}(\tau -s) \p_s^{3M+1} h(s)ds
        =\sum_{j=0}^{2M} a_j^{3M} \bigl( (\f12+3j)H_{-\f12+3j}(\tau)-\f14 H_{\f52+3j}(\tau) \bigr)\\
        &=\f12 a^M_0 H_{-\f12}(\tau) +\sum_{j=0}^{2M-1} (-\f14 a_j^{3M}+(\f72+3j)a_{j+1}^{3M})H_{\f52+3j}(\tau)-\f14 a_{2M}^{3M} H_{\f52+6M}(\tau).
    \end{align*}
    Then, by comparing the limit as $\tau\rightarrow0+$, we find that $\p_s^{3M}h(0)=\f12 a_0^M$ and therefore,
    \begin{equation}
        \begin{aligned}
            \int_0^\tau H_{-\f12}(\tau -s) \p_s^{3M+1} h(s)ds
            &=\sum_{j=0}^{2M} a_{j}^{3M+1}H_{\f52+3j}(\tau),
        \end{aligned}
    \end{equation}
    with $a_j^{3M+1}=-\f14 a_j^{3M}+(\f72+3j)a_{j+1}^{3M}$ for $0\leq j\leq 2M-1$ and $a^{3M+1}_{2M}=-\f14 a^{3M}_{2M}$.

    Then, we apply $\p_\tau$ and use \eqref{eq3.23} again to get
    \begin{align*}
        &\qquad H_{-\f12}(\tau)\p_s^{3M+1}h(0)+\int_0^\tau H_{-\f12}(\tau -s) \p_s^{3M+2} h(s)ds \\
        &=\sum_{j=0}^{2M} a_j^{3M+1}\bigl( (\f52+3j)H_{\f32+3j}(\tau)-\f14 H_{\f92+3j}(\tau)\bigr)\\
        &=\f52 a^{3M+1}_0 H_{\f32}(\tau) +\sum_{j=1}^{2M} \bigl((\f{5}2+3j)a_{j}^{3M+1}-\f14a_{j-1}^{3M+1}\bigr)H_{\f92+3j}(\tau)-\f14 a_{2M}^{3M+1} H_{\f92+6M}(\tau).
    \end{align*}
    By comparing the limit as $\tau\rightarrow0+$, one derives $\p_s^{3M+1}h(0)=0$ and therefore, 
    \begin{equation}
        \begin{aligned}
            \int_0^\tau H_{-\f12}(\tau -s) \p_s^{3M+2} h(s)ds
            &=\sum_{j=0}^{2M+1} a_{j}^{3M+2}H_{\f32+3j}(\tau),
        \end{aligned}
    \end{equation}
    with $a_0^{3M+2}=\f52 a^{3M+1}_0$, $a_j^{3M+2}=(\f52+3j)a_j^{3M+1}-\f14 a_{j-1}^{3M+1}$ for $1\leq j\leq 2M$ and $a_{2M+1}^{3M+2}=-\f14 a_{2M}^{3M+1}$.

    Similarly, we take $\p_\tau$ derivatives, use \eqref{eq3.23} and comparing the limit as $\tau\rightarrow0+$ to conclude $\p_s^{3M+2}h(0)=0$ and 
    \begin{equation}
        \begin{aligned}
            \int_0^\tau H_{-\f12}(\tau -s) \p_s^{3(M+1)} h(s)ds
            &=\sum_{j=0}^{2(M+1)} a_{j}^{3(M+1)}H_{\f12+3j}(\tau),
        \end{aligned}
    \end{equation}
    with $a_0^{3(M+1)}=\f32 a^{3M+2}_0$, $a_j^{3(M+1)}=(\f32+3j)a_j^{3M+2}-\f14 a_{j-1}^{3M+2}$ for $1\leq j\leq 2M+1$ and $a_{2(M+1)}^{3(M+1)}=-\f14 a_{2M+1}^{3M+2}$.
        
    This finishes the induction proof of \eqref{eq:lem3.3a}, and \eqref{eq:lem3.3b} and \eqref{eq:lem3.3c} are already derived from \eqref{eq:lem3.3a} during the above proof.
\end{proof}

\begin{col}\label{col3.1}
    {  Let $h(s)$ solve the integral equation \eqref{eq3.9}. Then, for any $m\in \N$, $(1+s^2)\p_s^m h(s)\in L^2_s(\R_+)$.
    }
\end{col}
\begin{proof}
    By using the formulas in \eqref{eq:lem3.3}, one can easily repeat the proof of Lemma \ref{lem3.2} to prove this high regularity estimate. We omit the details here.
\end{proof}

From Lemma \ref{lem3.3}, we see that $\p_s^m h(0)=0$ for all $3\nmid m $, which implies $h(s)$ behaves like $\sum_{j=0}^\oo a_j s^{3j}$ as $s$ small enough. Such a polynomial will not become singular under the $(s^{-\f12}\p_s)^m$ derivatives. To present the detailed control for $s$ small, we need the following Lemma:

\begin{lem}\label{lem3.4}
    {  Let $h(s)$ solve the integral equation \eqref{eq3.9}. Then, for any $m\in \N$, $(s^{-\f12}\p_s)^m h(s)\in L^2_s(0,1)$.
    }
\end{lem}
\begin{proof}
    From Corollary \ref{col3.1}, we have that $h$ is a smooth function. Therefore, we use Taylor expansion and the boundary condition of $\p_s^m$ derivatives obtained in Lemma \ref{lem3.3} to write that for $0<s\leq1$,
    $$
        h(s)= \sum_{j=0}^m \f{\p_s^{3j}h(0)}{(3j)!} s^{3j} +\f1{(3m+2)!}\int_0^s (s-s')^{3m+2} \p_s^{3m+3} h(s') ds' .
    $$

    For the polynomial part, we directly compute that
    $$
        \| ( s^{-\f12}\p_s)^m \sum_{j=0}^m \f{\p_s^{3j}h(0)}{(3j)!} s^{3j}\|_{L^2_s(0,1)}
        \leq C_{m} \sum_{j=0}^{m} \|s^{\f32 j}\|_{L^2_s(0,1)} \leq C_{m}.
    $$

    For the integration part, we denote $\tilde{h}(s)=\f1{(3m+2)!}\int_0^s (s-s')^{3m+2} \p_s^{3m+3} h(s') ds'$ which is a $C^{3m+3}([0,1])$ function satisfying $\p_s^j \tilde{h}=0$ for all $j<3m+3$. Therefore, we deduce from the weighted Hardy inequality $\|s^{\alpha} f\|_{L^2(0,1)}\leq C \|s^{\alpha+1}\p_sf\|_{L^2(0,1)}$ for $\alpha<-\f12$ and $f(0)=0$ that 
    $$
        \|( s^{-\f12}\p_s)^m \tilde{h}\|_{L^2(0,1)}\leq C_{m} \sum_{j=0}^m \| s^{-\f{m}2-j}\p_s^{m-j}\tilde{h}\|_{L^2(0,1)} \leq C_m  \| \p_s^{\f32m+1}\tilde{h}\|_{L^2(0,1)} \leq C_{m}. 
    $$

    Combining the above estimates, we arrive at 
    $$
        \|( s^{-\f12}\p_s)^m h\|_{L^2(0,1)} \leq C_{m},
    $$
    which finishes the proof.
\end{proof}

\subsection{Existence of the kernel}
In this subsection, we prove the existence of the kernel $\bar{\Omega}$ and finish the proof of Theorem \ref{thm1}.

Let us first complete the proof of Proposition \ref{prop3.1}. 
\begin{proof}[Existence part of Proposition \ref{prop3.1}]
    The existence of such a function solving $\cL \bar{\Omega}=0$ and $\bar{\Omega}(X,0)=0$ is given by the inverse Fourier transform of \eqref{eq3.7}. 
    From \eqref{eq3.3}, we also have $\int_{\R^2_+} Y \bar{\Omega}(X,Y)dXdY=1.$
    
    It remains to check that $\bar{\Omega}$ belongs to $L^2(m)$. By Lemma \ref{lem3.1} and \eqref{eq3.20a}, we only need to show the regularity of $h(s)$. From Lemma \ref{lem3.2}, we see that  $$
    \|(1+s) h(s)\|_{L^2_s(\R_+)}\leq C.
    $$
    Also, we deduce from Corollary \ref{col3.1} that 
    $$
        \|(1+s^{\f32})(\f23s^{-\f12}\p_s)^{m}h(s)\|_{L^2_s(1,\oo)} \leq C_m \sum_{j=1}^{m}\|(1+s)\p_s^j h\|_{L^2_s(\R^2_+)}\leq C_m,
    $$
    which together with Lemma \ref{lem3.4} implies
    $$
        \|(1+s^{\f32})(\f23s^{-\f12}\p_s)^{m}h(s)\|_{L^2_s(\R_+^2)} \leq C_m.
    $$
    Combining the above estimates, we already checked all the norms needed and show $\bar{\Omega}\in L^2(m)$ for all $m\in\N$. 
     This finishes the proof.
\end{proof}

\begin{col}\label{col3.2}
    {  Let $\bar{\Omega}$ be given by Proposition \ref{prop3.1}. Then, for any $\alpha_1,\alpha_2,m\in \N$, $\p_X^{\alpha_1}\p_Y^{\alpha_2} \bar{\Omega}$ belongs to $L^2(m)$. }
\end{col}
\begin{proof}
    From $\cL\bar{\Omega}=0$, we write back in the original variables (with $\nu=1$) that $\bar{\omega}(t,x,y) =(1+t)^{-\f52}\bar{\Omega}((1+t)^{-\f32}x, (1+t)^{-\f12}y)$ solves
    $$
        \p_t \bar{\omega} -\p_y^2 \bar{\omega}+y\p_x \bar{\omega}=0 \andf \bar{\omega}|_{t=0}=\bar{\Omega}(x,y),
    $$
    with Dirichlet boundary conditions. By standard hypo-ellipticity theory (see for instance \cite{hormander1967hypoelliptic,villani2009hypocoercivity}), one can prove that 
    $$
        \|\p_x^{\alpha_1}\p_y^{\alpha_2}\bar{\omega}(t)\|_{L^2(m)}\leq C_m t^{-\f{3\alpha_1+\alpha_2}2} \|\bar{\Omega}\|_{L^2(m)},
    $$
    which we omit the detailed proof. Therefore, we conclude that 
    $$
        \|\p_X^{\alpha_1}\p_Y^{\alpha_2} \Omega\|_{L^2(m)}\leq C_{{\alpha_1},{\alpha_2}}  \|\p_x^{\alpha_1} \p_y^{\alpha_2}\bar{\omega}(1)\|_{L^2(m)}\leq C_{m,{\alpha_1},{\alpha_2}} \|\bar{\Omega}\|_{L^2(m)},
    $$
    which finishes the proof. We remark that one can see subsection \ref{subsection5.2} for a proof of the smoothing effect of \eqref{eqs:Om} from the hypo-ellipticity of $\cL_t$.
\end{proof}

\begin{proof}[Proof of Theorem \ref{thm1}]
    From Proposition \ref{prop3.1} and Corollary \ref{col3.2}, we have already proven the existence, uniqueness, and regularity of $\bar{\Omega}\in L^2(m)$ solving $\cL\bar{\Omega}=0$. Then, for any $F_0\in L^2(m)$, it follows from $$\int_{\R^2_+}Y \bigl(F_0-M_2[F_0]\bar{\Omega}\bigr)dXdY=M_2[F_0]-M_2[F_0]=0$$
    and \eqref{eq2.5a} that 
    $$
        \|e^{\tau\cL}\bigl(F_0-M_2(F_0)\bar{\Omega}\bigr)\|_{L^2(m)} \leq C_{m,\delta}e^{-(1-\delta)\tau} \|F_0-M_2(F_0)\bar{\Omega}\|_{L^2(m)},
    $$
    which gives \eqref{eq:thm1}. This finishes the proof.
\end{proof}

\section{Estimates in the weighted $L^2(m)$ space}\label{section 4}

In this section, we aim to derive some a priori estimates for $\Omega(t)$ and its derivatives in $L^2(m)$. To do so, we first give some decay estimates for $\omega(t)$ in $L^p$ by Moser's method to avoid the difficulties from nonlinearity, and then use the $L^p$ norms to estimate the evolution of $\Omega(t)$ in $L^2(m)$. Finally, we shall use hypo-ellipticity to control $\p_X^{\alpha_1}\p_Y^{\alpha_2}\Omega(t)$ in $L^2(m)$ for ${\alpha_1}+{\alpha_2}\leq2$ or $(\alpha_1,\alpha_2) =(3,0)$.

\subsection{Estimates in the $L^p$ space}

For any $L^1$ initial data $\omega_0$, the system \eqref{eqs:w} should be local-wellposed, since the influence of the Couette flow is ignorable for some time scale, see \eqref{eq0.3}. Therefore, in this paper, we omit the details for the local-well-posedness theory in $L^1$ and present that the $L^1$ norm is always decreasing:

\begin{prop}\label{prop4.1}
    { 
    For any non-trivial initial data $\omega_0\in L^1$, there exists a global solution $\omega(t)$ of \eqref{eqs:w}, whose $L^1$ norm decreases.
    }
\end{prop}
\begin{proof}
    Given a local $L^1$ solution of \eqref{eqs:w}, we decompose the initial data $\omega_0=\omega_0^+-\omega_0^-$ with $\omega^\pm_0 \eqdefa \f12(|\omega_0|\pm \omega_0 )$, and consider the functions $\omega^\pm$ solving the following transport-diffusion equation on half space 
    \begin{equation}\label{eqs:w +-}
        \quad \left\{\begin{array}{l}
        \displaystyle \pa_t \om^\pm-\nu\Delta \omega^\pm +y\p_x \omega^\pm +u\cdot\nabla \omega^\pm=0, \qquad (t,x,y)\in\R^+\times\R^2_+, \\
        \displaystyle \omega^\pm|_{y=0}=0,\\
        \displaystyle  \omega^\pm|_{t=0}=\omega_{0}^\pm(x,y).
        \end{array}\right.
    \end{equation} 
    Due to strong maximal principle, we know $\om^\pm(t,x,y)\geq 0$ for all $t> 0$ and $y>0$, and $\p_y \omega^\pm (t,x,0)\geq 0$ for all $t>0$, and the two inequalities becomes equalities if and only if $\omega^\pm=0$. Therefore, we integrate the equations over $\R^2_+$ to find that for all $t>0$,
    $$
    \f{d}{dt}\int_{\R^2_+} \omega^\pm(t,x,y) dxdy =-\nu \int_\R \p_y \om^\pm (t,x,0)dx \leq 0,
    $$
    which implies the $L^1$ norms for $\omega^\pm$ are non-increasing, in particularly, 
    $$\int_{\R^2_+}\omega^\pm(t,x,y) dxdy\leq \int_{\R^2_+}\om^\pm_0(x,y)dxdy.$$
    Finally, we arrive at
    $$
    \|\omega(t)\|_{L^1} \leq \|\omega^+(t)\|_{L^1}+\|\omega^-(t)\|_{L^1}\leq \|\omega^+_0\|_{L^1}+\|\omega^-_0\|_{L^1}=\|\omega_0\|_{L^1}.
    $$
    One can check conditions for equality in the above inequalities to see that the $L^1$ norm always decreases except when the initial data is trivial. 
\end{proof}

By Moser's method, one has the decay for the $L^p$ norm from the scale of the heat equation:
\begin{lem}\label{S2lem3}
    {  For any $\omega_0\in L^1$, the solution of \eqref{eqs:w} satisfies that for $1\leq p\leq \oo$,
    \begin{equation}\label{eq4.2}
        \|\omega(t)\|_{L^p}\lesssim (\nu t)^{\f1p-1}\|\omega_0\|_{L^1}.
    \end{equation}
    }
\end{lem}
\begin{proof}
    The estimate is standard, and we omit the details here. One can follow the proof of Lemma 2.3 in \cite{liu2026nonlinear} to complete the proof.
\end{proof}

\subsection{The estimates of the $L^2(m)$ norm}
In this subsection, we give an upper bound estimate of the $L^2(m)$ norm of $\Omega(t)$ under time evolution.

\begin{prop}\label{prop5.1}
    {  Let $m>6$. For any $\Omega_0\in L^2(m)$, the solution of \eqref{eqs:Om} satisfies
    \begin{equation}\label{eq5.1}
        \|\Omega(t)\|_{L^2(m)}\leq C_m (1+t) \bigl(1+\nu^{-\f32}\|\Omega_0\|_{L^2(m)}\bigr)^m \|\Omega_0\|_{L^2(m)}.
    \end{equation}
    }
\end{prop}
\begin{proof}
    We first remark that $L^2(m)$ can be embedded into $L^1$, when $m>1$.
    Therefore, the Moser type estimate \eqref{eq4.2} gives
    \begin{equation}\label{eq5.2}
        \begin{aligned}
            \|\Omega(t)\|_{L^p} &= \nu^\f32 (1+t)^\f52\|\omega (t, \nu^\f12 (1+t)^\f32 X, \nu^\f12 (1+t)^\f12 Y)\|_{L^p}\\&= \nu^{\f32-\f1p}(1+t)^{\f52-\f2p}\|\omega(t)\|_{L^p}
            \leq C \nu^\f12 (1+t)^{\f52-\f2p}t^{\f1p-1} \|\omega_0\|_{L^1} \\
            &\leq C (1+t)^{\f52-\f2p}t^{\f1p-1}\|\Omega_0\|_{L^1}\leq C_m(1+t)^{\f52-\f2p}t^{\f1p-1}\|\Omega_0\|_{L^2(m)}.
        \end{aligned}
    \end{equation}
    On the other hand, the $L^p$ norms of $\omega(t)$ are non-increasing as time $t$ grows, which implies that for $1\leq p\leq 2$,
    \begin{align*}
        \|\Omega(t)\|_{L^p} &= \nu^{\f32-\f1p}(1+t)^{\f52-\f2p}\|\omega(t)\|_{L^p}\leq  \nu^{\f32-\f1p}(1+t)^{\f52-\f2p}\|\omega_0\|_{L^p} \\
        &= (1+t)^{\f52-\f2p}\|\Omega_0\|_{L^p}\leq C_m(1+t)^{\f52-\f2p}\|\Omega_0\|_{L^2(m)}.
    \end{align*}
    By discussing whether $t$ is small or not, one concludes that for $1\leq p\leq 2$,
    \begin{equation}\label{eq5.3}
        \|\Omega(t)\|_{L^p}\leq C_m (1+t)^{\f32-\f1p} \|\Omega_0\|_{L^2(m)}
    \end{equation}
    Similarly as \eqref{eq4.2}, for $2\leq p\leq \oo$, Moser's method shows that $\|\omega(t)\|_{L^p}\lesssim (\nu t)^{\f1p-\f12}\|\omega_0\|_{L^2}$, which together with the non-increasing property of $\|\omega(t)\|_{L^2}$ imply
    \begin{align*}
        \|\Omega(t)\|_{L^p} &= \nu^{\f32-\f1p}(1+t)^{\f52-\f2p}\|\omega(t)\|_{L^p}\leq C \nu (1+t)^{\f52-\f2p} t^{\f1p-\f12}\|\omega(t)\|_{L^2}\\
        &\leq C \nu(1+t)^{\f52-\f2p}t^{\f1p-\f12}\|\omega_0\|_{L^2} \leq C (1+t)^{\f52-\f2p}t^{\f1p-\f12}\|\Omega_0\|_{L^2}.
    \end{align*}
    Again, we discuss whether $t$ is small to prove that for $2\leq p\leq \oo$,
    \begin{equation}\label{eq5.4}
        \|\Omega(t)\|_{L^p}\leq C_m (1+t)^{2-\f2p} t^{\f1p-\f12}\|\Omega_0\|_{L^2(m)}. 
    \end{equation}

    Now, let us consider the estimates with the homogeneous weight $a^m(X,Y)\eqdefa (X^2+Y^2)^\f{m}2$. By using integration by parts and $\Omega(X,0)=0$, we have that
    \begin{equation}
            \begin{aligned}\label{eq5.4a}
        &\int_{\R^2_+} \Delta_t \Omega \Omega a^{2m} dXdY= -\f1{(1+t)^2}\int_{\R^2_+} |\p_X\Omega|^2 a^{2m} dXdY-\int_{\R^2_+} |\p_Y\Omega|^2 a^{2m} dXdY \\
        &\qquad\qquad+\f{1}{2(1+t)^2}\int_{\R^2_+} \Omega^2 \p_X^2 a^{2m} dXdY+\int_{\R^2_+} \Omega^2 \p_Y^2 a^{2m} dXdY\\
       & = -\f1{(1+t)^2}\|\p_X \bigl(a^m\Omega\bigr) \|_{L^2}^2-\|\p_Y \bigl(a^m\Omega\bigr) \|_{L^2}^2 +\f1{(1+t)^2}\| \Omega \p_X a^m\|_{L^2}^2+\| \Omega \p_Y a^m\|_{L^2}^2,
    \end{aligned}
    \end{equation}
    where $\Delta_t =\f1{(1+t)^2}\p_X^2+\p_Y^2$. Also, we compute that
    \begin{align*}
        &\int_{\R^2_+} \bigl(\f32X\p_X+\f12Y\p_Y+\f52-Y\p_X\bigr)\Omega \cdot \Omega a^{2m} dXdY
        = -\f34 \int_{\R^2_+} \Omega^2 \p_X\bigl( X a^{2m}\bigr) dXdY \\
        &\qquad\qquad-\f{1}{4}\int_{\R^2_+} \Omega^2 \p_Y \bigl(Ya^{2m}\bigr) dXdY+\f52\|\Omega a^m\|_{L^2}^2 +\f12\int_{\R^2_+} \Omega^2 Y\p_X a^{2m} dXdY\\
        & =\f32 \|\Omega a^m\|_{L^2}^2 -\f32 m\|\Omega X a^{m-1}\|_{L^2}^2 -\f{m}{2}\| \Omega Y a^{m-1}\|_{L^2}^2+m\int_{\R_+^2}\Omega^2 XY a^{2m-2} dXdY.
    \end{align*}
    Combining the above two estimates, we use $\p_Xa^m=mXa^{m-2}$ and $\p_Ya^m=mYa^{m-2}$ to get
    \begin{align*}
        &\int_{\R^2_+} \cL_t \Omega \cdot \Omega a^{2m} dXdY= -\f1{(1+t)^2}\|\p_X \bigl(a^m\Omega\bigr) \|_{L^2}^2-\|\p_Y \bigl(a^m\Omega\bigr) \|_{L^2}^2 \\
        &\qquad\qquad+\f{m^2}{(1+t)^2}\| \Omega X a^{m-2}\|_{L^2}^2+m^2\| \Omega Y a^{m-2}\|_{L^2}^2+\f32 \|\Omega a^m\|_{L^2}^2\\
        &\qquad\qquad-\f32 m\|\Omega X a^{m-1}\|_{L^2}^2 -\f{m}{2}\| \Omega Y a^{m-1}\|_{L^2}^2+m\int_{\R_+^2}\Omega^2 XY a^{2m-2} dXdY\\
        &\leq  \bigl(\f32-(1-\f{\sqrt{2}}{2})m\bigr)\|\Omega a^m\|_{L^2}^2+m^2\| \Omega a^{m-1}\|_{L^2}^2.
    \end{align*}
    Now, we use the interpolation inequality that for $\delta>0$
    $$
        m^2\| \Omega a^{m-1}\|_{L^2}^2 \leq \delta m\|\Omega a^m\|_{L^2}^2 +C \delta^m m^{m+1} \|\Omega \|_{L^2}^2
    $$
    to conclude
    \begin{equation}\label{eq5.5}
        \int_{\R^2_+} \cL_t \Omega \cdot \Omega a^{2m} dXdY \leq \bigl(\f32-(1-\f{\sqrt{2}}{2}-\delta)m\bigr)\|\Omega a^m\|_{L^2}^2+ C_{\delta,m} \|\Omega \|_{L^2}^2.
    \end{equation}

    For the nonlinear parts in \eqref{eqs:Om}, we compute by using integration by parts and Holder's inequality that  
    \begin{align*}
        &\int_{\R^2_+} \cN_t \Omega \cdot \Omega a^{2m} dXdY  = \nu^{-\f32}(1+t)^{-\f52}\int_{\R^2_+} \bigl(\p_Y\Delta_t^{-1}\Omega \p_X \Omega -\p_X\Delta_t^{-1}\Omega \p_Y \Omega\bigr) \cdot \Omega a^{2m} dXdY\\
         &\qquad = \nu^{-\f32}(1+t)^{-\f52}\int_{\R^2_+} \bigl(\p_X\Delta_t^{-1}\Omega \p_Y a^{2m} -\p_Y\Delta_t^{-1}\Omega \p_X a^{2m}\bigr) \Omega^2 dXdY\\
         &\qquad \leq C_m \nu^{-\f32}(1+t)^{-\f52} \|(\p_X,\p_Y)\Delta_t^{-1}\Omega\|_{L^\oo}\|\Omega a^{m-1}\|_{L^2} \|\Omega a^m\|_{L^2}.
    \end{align*}
    Here, we use the re-scaled version of the elliptic estimates on the half plane, \eqref{eq5.3} and \eqref{eq5.4} to find
    \begin{equation}\label{eq5.5a}
    \begin{aligned}
        &\|\p_X\Delta_t^{-1}\Omega\|_{L^\oo} +(1+t)\|\p_Y \Delta_t^{-1}\Omega\|_{L^\oo} \leq C (1+t)^\f32 \|\Omega\|_{L^{\f{2m}{m-1}}}^\f12 \|\Omega\|_{L^{\f{2m}{m+1}}}^\f12\\
        &\qquad\leq C_m (1+t)^\f32 \bigl( (1+t)^{1+\f1m}t^{-\f1{2m}}\|\Omega_0\|_{L^2(m)}\bigr)^\f12 \bigl((1+t)^{1-\f1{2m}}\|\Omega_0\|_{L^2(m)}\bigr)^\f12 \\
        &\qquad\leq C_m (1+t)^{\f52+\f1{4m}}t^{-\f1{4m}} \|\Omega_0\|_{L^2(m)},
    \end{aligned}
    \end{equation}
    which together with the interpolation inequality $\|\Omega a^{m-1}\|_{L^2} \leq \|\Omega a^m\|_{L^2}^{\f{m-1}{m}} \|\Omega\|_{L^2}^{\f{1}{m}}$ implies that 
    \begin{equation}\label{eq5.6}
        \begin{aligned}
            &\int_{\R^2_+} \cN_t \Omega \cdot \Omega a^{2m} dXdY \leq  C_m \nu^{-\f32}(1+t)^{\f1{4m}}t^{-\f1{4m}} \|\Omega_0\|_{L^2(m)} \|\Omega \|_{L^2}^\f1m\|\Omega a^m\|_{L^2}^{2-\f1m}\\
            &\qquad\leq \delta m \|\Omega a^m\|_{L^2}^2 + C_{\delta,m} \nu^{-3m}(1+t)^{\f1{2}}t^{-\f1{2}} \|\Omega_0\|_{L^2(m)}^{2m} \|\Omega\|_{L^2}^2.
        \end{aligned}
    \end{equation}
       
    Finally, we take the inner product of \eqref{eqs:Om} with $\Omega a^{2m}$ and use \eqref{eq5.5} and \eqref{eq5.6} with $\delta=\f{3-2\sqrt{2}}8$ to obtain
    \begin{align*}
        \f{1+t}2 \f{d}{dt} \|\Omega a^m\|_{L^2}^2 &\leq  (\f32-\f{m}4)\|\Omega a^m\|_{L^2}^2 + C_m \bigl(1+\nu^{-3m}(1+t)^{\f12}t^{-\f1{2}} \|\Omega_0\|_{L^2(m)}^{2m}\bigr) \|\Omega\|_{L^2}^2.
    \end{align*}
    Then, we apply Gronwall's inequality and use the decay estimate \eqref{eq5.3} to conclude that for $m>2$,
    \begin{align*}
        \|\Omega (t) a^m\|_{L^2}^2 &\leq (1+t)^{3-\f{m}2}\|\Omega_0 a^m\|_{L^2}^2 + C_m (1+t)^{3-\f{m}2} \int_0^t (1+s)^{\f{m}2-4} \\
        &\qquad\qquad\qquad\times \bigl(1+\nu^{-3m}(1+s)^{\f1{2}}s^{-\f1{2}} \|\Omega_0\|_{L^2(m)}^{2m}\bigr) (1+s)^2 \|\Omega_0\|_{L^2(m)}^2 ds\\
        &\leq C_m (1+t)^2 \bigl(1+\nu^{-3}\|\Omega_0\|_{L^2(m)}^2\bigr)^{m} \|\Omega_0\|_{L^2(m)}^2.
    \end{align*}
    By taking the square root and using \eqref{eq5.3}, we arrive at \eqref{eq5.1}, which finishes the proof.
\end{proof}

\subsection{Estimates of the higher regularities}\label{subsection5.2}
The goal of this subsection is to control the derivatives of $\Omega(t)$ in $L^2(m)$. When $t$ is finite, the rescaled heat operator $\Delta_t$ can gain regularity in both $X$ and $Y$ variables. Since we are considering the long-time behavior, one has to use the hypo-elliptic structure to gain horizontal regularities for large $t$.

Let us introduce the following energy functional:
\begin{align}\label{def:E}
E(t)\eqdefa &\|\Omega(t)\|_{L^2(m)}^2+c_1\bigl(\ln\f{1+t}{1+t_0}\bigr)\|\p_Y \Omega(t)\|_{L^2(m)}^2+c_2\bigl(\ln\f{1+t}{1+t_0}\bigr)^2 \bigl(\p_X \Omega (t)\big|\p_Y \Omega(t) \bigr)_{L^2(m)} \\
\nonumber&+c_3\bigl(\ln\f{1+t}{1+t_0}\bigr)^3\|\p_X \Omega(t)\|_{L^2(m)}^2+c_4\bigl(\ln\f{1+t}{1+t_0}\bigr)^2\|\p_Y^2\Omega(t)\|_{L^2(m)}^2\\
\nonumber&+c_5\bigl(\ln\f{1+t}{1+t_0}\bigr)^4\|\p_X\p_Y\Omega(t)\|_{L^2(m)}^2+c_6\bigl(\ln\f{1+t}{1+t_0}\bigr)^5 \bigl(\p_X^2 \Omega(t) \big|\p_X\p_Y \Omega(t) \bigr)_{L^2(m)}\\
\nonumber& +c_7\bigl(\ln\f{1+t}{1+t_0}\bigr)^6\|\p_X^2\Omega(t)\|_{L^2(m)}^2 +c_{8}\bigl(\ln\f{1+t}{1+t_0}\bigr)^7\|\p_X^2\p_Y\Omega(t)\|_{L^2(m)}^2\\
\nonumber&+c_{9}\bigl(\ln\f{1+t}{1+t_0}\bigr)^8 \bigl(\p_X^3 \Omega(t) \big|\p_X^2\p_Y \Omega(t) \bigr)_{L^2(m)} +c_{10}\bigl(\ln\f{1+t}{1+t_0}\bigr)^9\|\p_X^3\Omega(t)\|_{L^2(m)}^2,
\end{align}
and the associated dissipation energy functional:
\begin{align}\label{def:D}
D(t)\eqdefa &\tt^{-2}\|\p_X\Omega(t)\|_{L^2(m)}^2+\|\p_Y\Omega(t)\|_{L^2(m)}^2+c_1\bigl(\ln\f{1+t}{1+t_0}\bigr)\Bigl(\|\p_Y^2\Omega(t)\|_{L^2(m)}^2\\
\nonumber&+\tt^{-2}\|\p_X\p_Y\Omega(t)\|_{L^2(m)}^2\Bigr)
+c_2\bigl(\ln\f{1+t}{1+t_0}\bigr)^2\|\p_X\Omega(t)\|_{L^2(m)}^2\\
\nonumber&+c_3\bigl(\ln\f{1+t}{1+t_0}\bigr)^3\Bigl(\tt^{-2}\|\p_X^2\Omega(t)\|_{L^2(m)}^2+\|\p_X\p_Y\Omega(t)\|_{L^2(m)}^2\Bigr)\\
\nonumber&+c_4\bigl(\ln\f{1+t}{1+t_0}\bigr)^2\Bigl(\tt^{-2}\|\p_X\p_Y^2\Omega(t)\|_{L^2(m)}^2+\|\p_Y^3\Omega(t)\|_{L^2(m)}^2\Bigr) \\
\nonumber&+c_5\bigl(\ln\f{1+t}{1+t_0}\bigr)^4\Bigl(\tt^{-2}\|\p_X^2\p_Y\Omega(t)\|_{L^2(m)}^2+\|\p_X\p_Y^2\Omega(t)\|_{L^2(m)}^2\Bigr)\\
\nonumber&
+c_6\bigl(\ln\f{1+t}{1+t_0}\bigr)^5\|\p_X^2\Omega(t)\|_{L^2(m)}^2\\
\nonumber&+c_7\bigl(\ln\f{1+t}{1+t_0}\bigr)^6\Bigl(\tt^{-2}\|\p_X^3\Omega(t)\|_{L^2(m)}^2+\|\p_X^2\p_Y\Omega(t)\|_{L^2(m)}^2\Bigr)\\
\nonumber&+c_{8}\bigl(\ln\f{1+t}{1+t_0}\bigr)^7\Bigl(\tt^{-2}\|\p_X^3\p_Y\Omega(t)\|_{L^2(m)}^2+\|\p_X^2\p_Y^2\Omega(t)\|_{L^2(m)}^2\Bigr)\\
\nonumber&+c_{9}\bigl(\ln\f{1+t}{1+t_0}\bigr)^8\|\p_X^3\Omega(t)\|_{L^2(m)}^2\\
\nonumber&+c_{10}\bigl(\ln\f{1+t}{1+t_0}\bigr)^9\Bigl(\tt^{-2}\|\p_X^4\Omega(t)\|_{L^2(m)}^2+\|\p_X^3\p_Y\Omega(t)\|_{L^2(m)}^2\Bigr).
\end{align}

In the above energy functionals,  $t_0$ will be taken to be an {arbitrarily} large time and $t$ to belong to $[t_0,2t_0]$ so that $\ln \f{1+t}{1+t_0}$ is small. All the estimates below hold for $t_0\geq T_0$ for some universal $T_0$. The small constants from $c_1$ to $c_{10}$ are chosen to satisfy
\begin{equation}\label{assumptions on c1-7}
\quad \left\{\begin{array}{l}
\displaystyle 1\gg c_1\gg c_2=c_4 \gg c_3\gg c_5\gg c_6\gg c_7 \gg c_{8}\gg c_{9}\gg c_{10}, \\
\displaystyle  c_1^2\ll c_2, \quad c_2^2\ll c_1c_3, \quad c_5^2\ll c_3c_6,\quad c_6^2\ll c_5c_7,
\quad c_{8}^2\ll c_7 c_{9},\quad c_{9}^2\ll c_{8}c_{10},
\end{array}\right.
\end{equation}
where  $f\ll g$ means that there is a large constant $C$ so that $f\leq {C} g$.
Below we present one example of $c_1$ to $c_{10}$  which can satisfy \eqref{assumptions on c1-7}. We first take a large enough constant $A$ and then take
\begin{align*}
    &(c_1,c_2,c_3,c_4,c_5,c_6,c_7,c_8,c_9,c_{10})\\
=&(A^{-3},A^{-5},A^{-6},A^{-5} ,A^{-9},A^{-11},A^{-12},A^{-15},A^{-17},A^{-18}).
\end{align*}

Before proceeding, we first introduce some useful anisotropic inequalities.

\begin{lem}\label{S3lem1}
{  Let $m>1$ and $t>1$, one has
\begin{equation}\label{eq:lem3.1a}
\|\p_X\Delta_t^{-1}\Omega(t)\|_{L^\oo}+(1+t)\|\p_Y\Delta_t^{-1}\Omega(t)\|_{L^\oo}\lesssim (1+t)^\f52\|\Omega_0\|_{L^2(m)},
\end{equation}
and
\begin{equation}\label{eq:lem3.1b}
\|\p_X\p_Y\Delta_t^{-1}\Omega(t)\|_{L^2}+(1+t)\|\p_Y^2\Delta_t^{-1}\Omega(t)\|_{L^2}\lesssim (1+t)^{2} \|\Omega_0\|_{L^2(m)},
\end{equation}
}\end{lem}
\begin{proof}
 \eqref{eq:lem3.1a} follows from \eqref{eq5.5a} and $t>1$. By classical elliptic regularity theory, we have
 \begin{align*}
 \|\p_X\p_Y\Delta_t^{-1}\Omega(t)\|_{L^2}+\tt\|\p_Y^2\Delta_t^{-1}\Omega(t)\|_{L^2}\lesssim \tt\|\Omega(t)\|_{L^2},
 \end{align*}
 which together with \eqref{eq5.3} for $p=2$ ensures
 \eqref{eq:lem3.1b}.
\end{proof}

\begin{lem}\label{S3lem2}
{  For any $0<\sigma<\f12$, one has
\begin{equation}\label{eq5.14}
\|\p_X\Delta_t^{-1} f\|_{L^\oo}+\tt\|\p_Y\Delta_t^{-1}f\|_{L^\oo}
\leq C_\sigma \tt^{1+\sigma} \|f\|_{L^2(1)}^{\f12+\sigma}\|\p_X f\|_{L^2(1)}^{\f12-\sigma},
\end{equation}
and
\begin{equation}\label{eq5.14a}
    \begin{aligned}
        \|\p_Y^2\Delta_t^{-1}f\|_{L^\oo}
\leq C_\sigma \tt^{\sigma} \|\p_Y f\|_{L^2(1)}^{\f12+\sigma}\|\p_X\p_Y f\|_{L^2(1)}^{\f12-\sigma},\\
 \|\p_Y^3\Delta_t^{-1}f\|_{L^\oo}
\leq C_\sigma \tt^{\sigma} \|\p_Y^2 f\|_{L^2(1)}^{\f12+\sigma}\|\p_X\p_Y^2 f\|_{L^2(1)}^{\f12-\sigma}.
    \end{aligned}
\end{equation}
}\end{lem}
\begin{proof}
Since the operator $\Delta_t^{-1}$ is compatible with odd extension, we extend the function $f$ oddly across the boundary. Therefore, it suffices to prove the desired estimate in the whole space by the Fourier method.

Let us denote $\hat{f}(\xi,\eta)$ as the Fourier transform of $f(X,Y)$ on $R^2$.  It is easy to observe that
\begin{align*}
\|\p_X\Delta_t^{-1} f\|_{L^\oo}+(1+t)\|\p_Y\Delta_t^{-1} f\|_{L^\oo}
\leq & \bigl\|{\xi}\Bigl((1+t)^{-2}\xi^2+\eta^2\Bigr)^{-1}\hat{f}\bigr\|_{L^1}\\
&+(1+t)\bigl\|{\eta}\Bigl((1+t)^{-2}\xi^2+\eta^2\Bigr)^{-1}
\hat{f}\bigr\|_{L^1} \\
\lesssim &(1+t)^{1+\sigma}\bigl\|{|\xi|^{-\sigma}|\eta|^{\sigma-1}}{\hat{f}}\bigr\|_{L^1}.
\end{align*}

Notice that due to $\s\in (0,\f12),$ for any $R>0,$ we have
\begin{align*}
\||\xi|^{-\sigma}g\|_{L^1_{\xi}}\leq &\Bigl(\int_{|\xi|\leq R}|\xi|^{-2\s}\Bigr)^{\f12}\|g\|_{L^2_\xi}+
\Bigl(\int_{|\xi|\geq R}|\xi|^{-2(1+\s)}\Bigr)^{\f12}\|\xi g\|_{L^2_\xi}
\\
\lesssim &R^{\f12-\s}\|g\|_{L^2_\xi}+R^{-\f12-\s}\|\xi g\|_{L^2_\xi}.
\end{align*}
Taking $R=\f{\|\xi g\|_{L^2_\xi}}{\|g\|_{L^2_\xi}}$ in the above inequality leads to
\beq \label{S3eq1} \||\xi|^{-\sigma}g\|_{L^1_{\xi}}\lesssim \|g\|_{L^2_\xi}^{\f12+\sigma}\|\xi g\|_{L^2_\xi}^{\f12-\sigma},.
\eeq
Similarly, one can show that
\beq \label{S3eq2}
\||\eta|^{\sigma-1}g\|_{L^1_{\eta}} \lesssim \|g\|_{L^\oo_{\eta}}^{1-2\sigma}\|g\|_{L^2_\eta}^{2\sigma}.
\eeq

Thanks to \eqref{S3eq1} and \eqref{S3eq2}, we get, by using
 the fact that $L^2(1)$ can be continuously embedded into $L^1_Y(L^2_X),$ that
\begin{align*}
\bigl\|{\hat{f}}{|\xi|^{-\sigma}|\eta|^{\sigma-1}}\bigr\|_{L^1}
\lesssim &\bigl\| \|\hat{f}\|_{L^\oo_{\eta}}^{1-2\sigma}\|\hat{f}\|_{L^2_\eta}^{2\sigma}\bigr\|_{L^2_\xi}^{\f12+\sigma}
\bigl\| \xi \|\hat{f}\|_{L^\oo_{\eta}}^{1-2\sigma}\|\xi\hat{f}\|_{L^2_\eta}^{2\sigma}\bigr\|_{L^2_\xi}^{\f12-\sigma} \\
\lesssim &\|\hat{f}\|_{L^2_\xi(L^\oo_{\eta})}^{(1-2\sigma)(\f12+\sigma)}\|\hat{f}\|_{L^2}^{2\sigma(\f12+\sigma)}
\|\xi\hat{f}\|_{L^2_\xi(L^\oo_{\eta})}^{(1-2\sigma)(\f12-\sigma)}\|\xi\hat{f}\|_{L^2}^{2\sigma(\f12-\sigma)}\\
\lesssim &\|f\|_{L^1_Y(L^2_{X})}^{(1-2\sigma)(\f12+\sigma)}\|f\|_{L^2}^{2\sigma(\f12+\sigma)}
\|\p_X f\|_{L^1_Y(L^2_{X})}^{(1-2\sigma)(\f12-\sigma)}\|\p_X f\|_{L^2}^{2\sigma(\f12-\sigma)}\\
\lesssim &\|f\|_{L^2(1)}^{\f12+\sigma}\|\p_X f\|_{L^2(1)}^{\f12-\sigma}.
\end{align*}
This finishes the proof of the whole plane version of \eqref{eq5.14}, and the whole plane version of \eqref{eq5.14a} is the direct corollary. By restricting the odd extension back to $\R^2_+$, we finish the proof.
\end{proof}

\begin{rmk}
    { 
    Due to $\Delta_t^{-1}$ does not commute with $\p_Y$, \eqref{eq5.14a} can not be directly derived from \eqref{eq5.14}, although the whole plane versions are related.
    }
\end{rmk}

Let us turn to the estimates of the terms appearing in \eqref{def:E}.

\begin{lem}\label{S3lem3}
Let $m>1$ and $t_0$ be a large enough positive constant. Then for $t\geq t_0$, one has
\begin{equation}\label{eq:lem1}
    \begin{aligned}
        &(1+t)\f{d}{dt}\|\Omega(t) \|_{L^2(m)}^2 +2(1+t)^{-2}\|\p_X\Omega\|_{L^2(m)}^2+2\|\p_Y\Omega\|_{L^2(m)}^2 \\
&\qquad\qquad\qquad\leq C_{m} \bigl(1+{\nu}^{-\f32}{\|\Omega_0\|_{L^2(m)}}\bigr) E(t).
    \end{aligned}
\end{equation}
\end{lem}
\begin{proof}
The proof of this lemma will follow along the same line as that of   Proposition \ref{prop5.1}. For simplicity, we shall denote $\langle X,Y\rangle$ by $b(X,Y).$ We first get, by a similar derivation of \eqref{eq5.4a}, that
\begin{align*}
\int_{\R^2_+} \Delta_t\Omega \, \Omega \, b^{2m}dXdY=&-\bigl(1+{t}\bigr)^{-2}\bigl\|\p_X\Omega\bigr\|_{L^2(m)}^2-2m(1+t)^{-2}\int_{\R^2_+} b^{2m-2}X\p_X\Omega \,\Omega dXdY\\
&-\|\p_Y\Omega\|_{L^2(m)}^2-2m \int_{\R^2_+} b^{2m-2} Y\p_Y \Omega\,\Omega dXdY\\
\leq&-(1+t)^{-2}\|\p_X\Omega\|_{L^2(m)}^2-\|\p_Y\Omega\|_{L^2(m)}^2
+C_m\|\Omega\|_{L^2(m-1)}^2.
\end{align*}
Then, we use $\int_{\R^2_+} (\f32 X\p_X+\f12Y\p_Y-Y\p_X+\f52)\Omega\,\Omega\,b^{2m}dXdY\leq C_m\|\Omega\|_{L^2(m)}^2 $ to find
\begin{equation}\label{eq5.18}
\int_{\R^2_+} \cL_t\Omega \, \Omega \, b^{2m}dXdY\leq-(1+t)^{-2}\|\p_X\Omega\|_{L^2(m)}^2-\|\p_Y\Omega\|_{L^2(m)}^2+C_m\|\Omega\|_{L^2(m)}^2.
\end{equation}

While we get, by using  integration by parts, that
\begin{align*}
\int_{\R^2_+}\cN_t \Omega\, \Omega\, b^{2m}dXdY=& m\nu^{-\f32} (1+t^2)^{-\f52}\int_{\R^2_+} b^{2m-2}  \Omega^2\bigl( Y\p_X\Delta_t^{-1}\Omega -X \p_Y \Delta_t^{-1}\Omega \bigr)dXdY\\
\leq &C_m \nu^{-\f32} (1+t^2)^{-\f52}\|\Omega\|_{L^2(m)}^2 \bigl( \|\p_X\Delta_t^{-1}\Omega \|_{L^\oo}+\|\p_Y\Delta_t^{-1}\Omega \|_{L^\oo} \bigr),
\end{align*}
from which and \eqref{eq:lem3.1a}, we infer
\begin{equation}\label{eq5.19}
\int_{\R^2_+}\cN_t \Omega\, \Omega\, b^{2m}dXdY \leq  C_m \nu^{-\f32}{\|\Omega_0\|_{L^2(m)}} \|\Omega\|_{L^2(m)}^2.
\end{equation}

Thanks to \eqref{eq5.18} and \eqref{eq5.19}, by taking $L^2(m)$ inner product of \eqref{eqs:Om} with $\Omega,$
we conclude the proof of \eqref{eq:lem1}.
\end{proof}

\begin{lem}\label{S3lem4}
{  Under the assumptions of Lemma \ref{S3lem3}, for $t_0\leq t\leq 2t_0 $, we have
\begin{equation}\label{eq:lem2}
\begin{aligned}
&(1+t) \f{d}{dt}\Bigl(\bigl(\ln\f{1+t}{1+t_0}\bigr)\|\p_Y\Omega\|_{L^2(m)}^2\Bigr) +2\bigl(\ln\f{1+t}{1+t_0}\bigr)\bigl(\tt^{-2}\|\p_X\p_Y\Omega\|_{L^2(m)}^2+\|\p_Y^2\Omega\|_{L^2(m)}^2\bigr)\\
&\leq \|\p_Y\Omega\|_{L^2(m)}^2+2\bigl(\ln\f{1+t}{1+t_0}\bigr)\|\p_X\Omega\|_{L^2(m)}\|\p_Y\Omega\|_{L^2(m)}\\
&\quad+C_m \bigl(1+\nu^{-1}{\|\Omega(1)\|_{L^2(m)}}\bigr)\bigl(\ln\f{1+t}{1+t_0}\bigr)^{\f12}  D(t).
\end{aligned}\end{equation}
}\end{lem}
\begin{proof}
By applying $\p_Y$ to \eqref{eqs:Om} and then taking  $L^2(m)$ inner product of the resulting equation with $\ln(\f{t}{t_0})\p_Y\Omega$, we compute
\begin{equation}\label{eq5.21}
    \begin{aligned}
        &\f{1+t}2 \f{d}{dt}\Bigl(\bigl(\ln\f{1+t}{1+t_0}\bigr)\|\p_Y\Omega(t)\|_{L^2(m)}^2\Bigr) =\f12 \|\p_Y\Omega\|_{L^2(m)}^2\\
        &\qquad\qquad\qquad\qquad+\bigl(\ln\f{1+t}{1+t_0}\bigr)\int_{\R^2_+} \p_Y\bigl(\cL_t\Omega+\cN_t \Omega\bigr)\, \p_Y\Omega\, b^{2m}dXdY.
    \end{aligned}
\end{equation}

It is easy to observe
\begin{equation}\label{[py;L]}
\p_Y\cL_t=\cL_t\p_Y -\p_X+\f12\p_Y,
\end{equation}
from which and \eqref{eq5.18}, we deduce that
\begin{equation}\label{eq5.23}
\begin{aligned}
\int_{\R^2_+} \p_Y\cL_t\Omega \,\p_Y\Omega\, b^{2m}dXdY
&\leq \int_{\R^2_+} \cL_t\p_Y\Omega\, \p_Y\Omega\, b^{2m} dXdY  \\
&\quad+\int_{\R^2_+} \bigl(|\p_X\Omega|+|\p_Y\Omega| \bigr)\,|\p_Y\Omega|\, b^{2m}dXdY\\
&\leq -(1+t)^{-2}\|\p_X\p_Y\Omega\|_{L^2(m)}^2-\|\p_Y^2\Omega\|_{L^2(m)}^2\\
&\quad +C_m \|\p_Y\Omega\|_{L^2(m)}^2+\|\p_X\Omega\|_{L^2(m)}\|\p_Y\Omega\|_{L^2(m)},
\end{aligned}
\end{equation}
where we used the boundary condition $\p_Y^2 \Omega|_{Y=0}=0$ for the integration by parts following the proof of \eqref{eq5.18}.

To deal with  the nonlinear part in \eqref{eq5.21}, we first decompose it to be
\begin{align*}
&\int_{\R^2_+} \p_Y \cN_t\Omega \,\p_Y\Omega\, b^{2m}dXdY\\
=&\nu^{-\f32}(1+t)^{-\f52} \Bigl(\int_{\R^2_+} \bigl(\p_Y \Delta_t^{-1}\Omega\,\p_X \p_Y\Omega
- \p_X\Delta_t^{-1}\Omega\, \p_Y^2\Omega\bigr)\,\p_Y\Omega\, b^{2m}dXdY\\
&+\int_{\R^2_+} \bigl(\p_Y^2 \Delta_t^{-1}\Omega\,\p_X \Omega
- \p_X\p_Y\Delta_t^{-1}\Omega\, \p_Y\Omega\bigr)\,\p_Y\Omega\, b^{2m}dXdY\Bigr)
\eqdef I_1+I_2.
\end{align*}
 It follows from a similar derivation of \eqref{eq5.19} that
$$
I_1\leq {C_{m}}{\nu}^{-\f32} {\|\Omega_0\|_{L^2(m)}} \|\p_Y\Omega\|_{L^2(m)}^2 \leq {C_{m}}{\nu}^{-\f32} {\|\Omega_0\|_{L^2(m)}} D(t),
$$
where we used the boundary condition $\Delta_t^{-1}\Omega|_{Y=0}=0$ for the integration by parts.
For the second integral $I_2$, we write
\begin{align*}
I_2 \leq \nu^{-\f32} (1+t)^{-\f52}\Bigl(&\|\p_Y^2 \Delta_t^{-1}\Omega\|_{L^2}\|b^m\p_X\Omega\|_{L^2_X(L^\oo_Y)}\|b^m\p_Y\Omega\|_{L^2_Y(L^\oo_X)}\\
&+\|\p_X\p_Y\Delta_t^{-1}\Omega\|_{L^2}\|b^m\p_Y\Omega\|_{L^2_X(L^\oo_Y)}\|b^m\p_Y\Omega\|_{L^2_Y(L^\oo_X)}\Bigr).
\end{align*}
Whereas  we get, by using one-dimensional interpolation inequality, that
\begin{equation}\label{Sobolev in x}
\begin{aligned}
\|b^m f\|_{L^2_Y(L^\oo_X)}&\leq C \|b^m f\|_{L^2}^\f12
\|\p_X(b^m f)\|_{L^2}^\f12\\
&\leq C_m \|f\|_{L^2(m)}^\f12\Bigl(\|f\|_{L^2(m-1)}+\|\p_X f\|_{L^2(m)} \Bigr)^\f12\\
&\leq C_m \|f\|_{L^2(m)}^\f12\|\p_X f\|_{L^2(m)}^\f12
\end{aligned}
\end{equation}
and similarly
\begin{equation}\label{Sobolev in y}
\| b^m f\|_{L^2_X(L^\oo_Y)}\leq C_m \|f\|_{L^2_m}^\f12\|\p_Y f\|_{L^2_m}^\f12
\end{equation}
which together with \eqref{eq:lem3.1b} ensures that
\begin{align*}
I_2  \leq &C_m \nu^{-\f32} (1+t)^{-\f52} \Bigl( \tt \|\Omega_0\|_{L^2(m)}\|\p_X\Omega\|_{L^2(m)}^\f12\|\p_Y\Omega\|_{L^2(m)}^\f12 \|\p_X\p_Y\Omega\|_{L^2(m)}\\
&\quad + \tt^2 \|\Omega_0\|_{L^2(m)}\|\p_Y\Omega\|_{L^2(m)}\|\p_X\p_Y\Omega\|_{L^2(m)}^\f12\|\p_Y^2\Omega\|_{L^2(m)}^\f12\Bigr)\\
\leq & C_m \nu^{-\f32}\|\Omega_0\|_{L^2(m)} \bigl(\ln\f{1+t}{1+t_0}\bigr)^{-\f12}  \Bigl( \bigl(\tt^{-1}\|\p_X\Omega\|_{L^2(m)}\bigr)^\f12\|\p_Y\Omega\|_{L^2(m)}^\f12 \\
&\qquad\qquad\qquad\qquad\qquad\qquad\times\bigl( (\ln\f{1+t}{1+t_0})^\f12\tt^{-1}\|\p_X\p_Y\Omega\|_{L^2(m)}\bigr)\\
&\quad + \|\p_Y\Omega\|_{L^2(m)}\bigl( (\ln\f{1+t}{1+t_0})^\f12\tt^{-1}\|\p_X\p_Y\Omega\|_{L^2(m)}\bigr)^\f12\bigl((\ln\f{1+t}{1+t_0})^\f12\|\p_Y^2\Omega\|_{L^2(m)}\bigr)^\f12\Bigr)\\
\leq &C_m \nu^{-\f32}\|\Omega_0\|_{L^2(m)} \bigl(\ln\f{1+t}{1+t_0}\bigr)^{-\f12}  D(t).
\end{align*}
By summarizing the estimates of $I_1$ and $I_2$, we arrive at
\begin{equation}\label{eq5.26}
\begin{aligned}
&\int_{\R^2_+} \p_Y \cN_t\Omega\, \p_Y\Omega\, b^{2m}dXdY \leq {C_m}\nu^{-\f32} \|\Omega_0\|_{L^2(m)} \bigl(\ln{t}/{t_0}\bigr)^{-\f12}  D(t).
\end{aligned}
\end{equation}

By substituting \eqref{eq5.23} and \eqref{eq5.26} into \eqref{eq5.21},  we conclude the proof of \eqref{eq:lem2}.
\end{proof}

\begin{lem}\label{S3lem5}
{  Under the assumptions of Lemma \ref{S3lem4}, we have
\begin{equation}\label{eq:lem3}
\begin{aligned}
&(1+t)\f{d}{dt}\Bigl( \bigl(\ln\f{1+t}{1+t_0}\bigr)^2 \bigl(\p_X \Omega \big|\p_Y \Omega \bigr)_{L^2(m)} \Bigr)+\bigl(\ln\f{1+t}{1+t_0}\bigr)^2\|\p_X\Omega\|_{L^2(m)}^2\\
&\leq 2\bigl(\ln\f{1+t}{1+t_0}\bigr)\bigl(\p_X \Omega \big|\p_Y \Omega \bigr)_{L^2(m)}+2\bigl(\ln\f{1+t}{1+t_0}\bigr)^2\Bigl(\tt^{-2}\|\p_X^2\Omega\|_{L^2(m)}\|\p_X\p_Y\Omega\|_{L^2(m)}
\\
&\qquad+\|\p_Y^2\Omega\|_{L^2(m)}\|\p_X\p_Y\Omega\|_{L^2(m)}\Bigr) +C_m\bigl(1+\nu^{-\f32}{\|\Omega_0\|_{L^2(m)}}\bigr)\bigl(\ln\f{1+t}{1+t_0}\bigr)^\f12 D(t) .
\end{aligned}\end{equation}
}\end{lem}
\begin{proof} We first get, by a direct computation of the time derivative and using the equation of \eqref{eqs:Om}, that
\begin{equation}\label{eq5.28}
\begin{aligned}
&(1+t)\f{d}{dt}\Bigl( \bigl(\ln\f{1+t}{1+t_0}\bigr)^2 \bigl(\p_X \Omega \big|\p_Y \Omega \bigr)_{L^2(m)} \Bigr)=2\bigl(\ln\f{1+t}{1+t_0}\bigr)\bigl(\p_X \Omega \big|\p_Y \Omega \bigr)_{L^2(m)} \\
&+\bigl(\ln\f{1+t}{1+t_0}\bigr)^2 \Bigl(\int_{\R^2_+}\p_X \bigl(\cL_t\Omega+\cN_t\Omega \bigr) \,\p_Y\Omega \, b^{2m}dXdY
+ \int_{\R^2_+}\p_X \Omega\, \p_Y\bigl(\cL_t\Omega+\cN_t\Omega \bigr)\,b^{2m}dXdY\Bigr).
\end{aligned}
\end{equation}

It is easy to observe 
\begin{equation}\label{[px;L]}
\p_X \cL_t=\cL_t\p_X +\f32 \p_X ,
\end{equation}
from which and \eqref{[py;L]}, we infer
\begin{align*}
\int_{\R^2_+}&\p_X \cL_t\Omega\,\p_Y\Omega\, b^{2m}dXdY
+ \int_{\R^2_+}\p_X \Omega\, \p_Y\cL_t\Omega\, b^{2m} dXdY\\
=&\int_{\R^2_+}\cL_t\p_X \Omega\,\p_Y\Omega\, b^{2m}dXdY
+ \int_{\R^2_+}\p_X \Omega\,\cL_t \p_Y\Omega\, b^{2m}dXdY \\
&-\|\p_X\Omega\|_{L^2(m)}^2 +2\int_{\R^2_+} \p_X\Omega\, \p_Y\Omega\, b^{2m}dXdY.
\end{align*}
While we get, by using integration by parts, that
\begin{align*}
&\int_{\R^2_+}\cL_t\p_X \Omega\, \p_Y\Omega\, b^{2m}dXdY
+ \int_{\R^2_+}\p_X \Omega\,\cL_t \p_Y\Omega\, b^{2m}dXdY\\
=& -2(1+t)^{-2} \int_{\R^2_+}\p_X^2\Omega \, \p_X\p_Y\Omega\, b^{2m} dXdY -2\int_{\R^2_+} \p_Y^2\Omega\, \p_X\p_Y\Omega\, b^{2m} dX dY \\
&-(1+t)^{-2}\int_{\R^2_+}\p_X \bigl(\p_X\Omega \, \p_Y\Omega\bigr)\, \p_X b^{2m} dXdY-\int_{\R^2_+}\p_Y \bigl(\p_X\Omega \, \p_Y\Omega\bigr)\, \p_Y b^{2m} dXdY\\
&-\int_{\R^2_+} \p_X\Omega \, \p_Y\Omega\,\bigl(\f52 b^{2m} -\f32\p_X(Xb^{2m})-\f12 \p_Y(Yb^{2m})+Y\p_X b^{2m} \bigr)  dXdY\\
\leq &2\tt^{-2}\|\p_X^2\Omega\|_{L^2(m)}\|\p_X\p_Y\Omega\|_{L^2(m)}+2\|\p_Y^2\Omega\|_{L^2(m)}\|\p_X\p_Y\Omega\|_{L^2(m)}\\
&+C_m\|\p_X\Omega\|_{L^2(m)}\|\p_Y\Omega\|_{L^2(m)} .
\end{align*}
As a result, we use $\|\p_X\Omega\|_{L^2(m)}\|\p_Y\Omega\|_{L^2(m)} \leq C\bigl(\f{1+t}{1+t_0}\bigr)^{-1}D(t) $ to obtain
\begin{equation}\label{eq5.30}
\begin{aligned}
\int_{\R^2_+}&\p_X \cL_t\Omega\,\p_Y\Omega\, b^{2m}dXdY+ \int_{\R^2_+}\p_X \Omega\, \p_Y\cL_t\Omega\, b^{2m}dXdY\\
\leq & -\|\p_X\Omega\|_{L^2(m)}^2+C_m\bigl(\f{1+t}{1+t_0}\bigr)^{-1}D(t)\\
&+2\bigl(\tt^{-2}\|\p_X^2\Omega\|_{L^2(m)}\|\p_X\p_Y\Omega\|_{L^2(m)}+\|\p_Y^2\Omega\|_{L^2(m)}\|\p_X\p_Y\Omega\|_{L^2(m)}\bigr).
\end{aligned}
\end{equation}

For the nonlinear terms in \eqref{eq5.28}, we get, by using integration by parts and Poincare's inequality $\|f\|_{L^2(m-1)}\leq C_m \min(\|\p_Xf\|_{L^2(m)}, \|\p_Yf\|_{L^2(m)})$, that
\begin{align*}
&\int_{\R^2_+}\p_X \cN_t\Omega\,\p_Y\Omega\, b^{2m}dXdY
+ \int_{\R^2_+}\p_X \Omega\, \p_Y\cN_t\Omega\, b^{2m}dXdY \\
\leq &C_m\|\cN_t\Omega\|_{L^2(m)}\bigl(\|\p_X\p_Y\Omega\|_{L^2(m)}+\|\p_X\Omega\|_{L^2(m-1)}+\|\p_Y\Omega\|_{L^2(m-1)} \bigr)\\
\leq &C_m \nu^{-\f32}\tt^{-\f52}\bigl(\|\p_X\Delta_t^{-1}\Omega\|_{L^\oo}\|\p_Y\Omega\|_{L^2(m)}+\|\p_Y\Delta_t^{-1}\Omega\|_{L^\oo}\|\p_X\Omega\|_{L^2(m)} \bigr)\|\p_X\p_Y\Omega\|_{L^2(m)},
\end{align*}
from which and \eqref{eq:lem3.1a}, we deduce that
\begin{equation}\label{eq5.31}
\begin{aligned}
\int&\p_X \cN_t\Omega\p_Y\Omega b^{2m}
+ \int\p_X \Omega \p_Y\cN_t\Omega b^{2m}\\
\leq &  {C_m}{\nu}^{-\f32}{\|\Omega_0\|_{L^2(m)}}\bigl(\|\p_Y\Omega\|_{L^2(m)}+\tt^{-1}\|\p_X\Omega\|_{L^2(m)} \bigr)\|\p_X\p_Y\Omega\|_{L^2(m)}\\
\leq &C_m {\nu}^{-\f32}\|\Omega_0\|_{L^2(m)}\bigl(\ln\f{1+t}{1+t_0}\bigr)^\f12\Bigl(\|\p_Y\Omega\|_{L^2(m)}^2+\tt^{-2}\|\p_X\Omega\|_{L^2(m)}^2\Bigr)^\f12 \\
&\qquad \times\Bigl( \bigl(\ln\f{1+t}{1+t_0}\bigr)^3\|\p_X\p_Y\Omega\|_{L^2(m)}^2\Bigr)^\f12\\
&\leq C_m {\nu}^{-\f32}\|\Omega_0\|_{L^2(m)}\bigl(\ln\f{1+t}{1+t_0}\bigr)^\f12 D(t).
\end{aligned}
\end{equation}

By substituting \eqref{eq5.30} and \eqref{eq5.31} into \eqref{eq5.28}, we achieve \eqref{eq:lem3}. This finishes the proof of Lemma \ref{S3lem5}.
\end{proof}

\begin{lem}\label{S3lem6}
{  Under the assumptions of Lemma \ref{S3lem4}, for  $0<\sigma<\f12,$ we have
\begin{equation}\label{eq:lem4}
\begin{aligned}
(1+t)& \f{d}{dt}\Bigl(\bigl(\ln\f{1+t}{1+t_0}\bigr)^3\|\p_X\Omega\|_{L^2(m)}^2\Bigr) +\bigl(\ln\f{1+t}{1+t_0}\bigr)^3\tt^{-2}\|\p_X^2\Omega\|_{L^2(m)}^2\\
&+\bigl(\ln\f{1+t}{1+t_0}\bigr)^3\|\p_X\p_Y\Omega\|_{L^2(m)}^2
\leq 3\bigl(\ln\f{1+t}{1+t_0}\bigr)^2\|\p_X\Omega\|_{L^2(m)}^2\\
&+C_m\bigl(1+\nu^{-\f32}{\|\Omega_0\|_{L^2(m)}}\bigr)E(t)+C_{\sigma}\nu^{-\f32} \tt^{-1} \bigl(\ln\f{1+t}{1+t_0}\bigr)^{\f\sigma2+\f{1}4} E(t)^\f12 D(t).
\end{aligned}\end{equation}
}\end{lem}
\begin{proof} We get,
by first applying $\p_X$ to \eqref{eqs:Om} and then taking  $L^2(m)$ inner product of the resulting equation with $\bigl(\ln\f{1+t}{1+t_0}\bigr)^3\p_X\Omega$, that
\begin{equation}\label{eq5.33}\begin{aligned}
\f{1+t}2 \f{d}{dt}&\Bigl(\bigl(\ln\f{1+t}{1+t_0}\bigr)^3\|\p_X\Omega(t)\|_{L^2(m)}^2\Bigr) \\
=&\f32 \bigl(\ln\f{1+t}{1+t_0}\bigr)^2\|\p_X\Omega\|_{L^2(m)}^2+\bigl(\ln\f{1+t}{1+t_0}\bigr)^3\int_{\R^2_+} \p_X\bigl(\cL_t\Omega+\cN_t \Omega\bigr)\, \p_X\Omega\, b^{2m}dXdY.
\end{aligned}
\end{equation}
It follows from \eqref{[px;L]} and \eqref{eq5.18} that 
\begin{equation}\label{eq5.34}
\begin{aligned}
&\int_{\R^2_+} \p_X\cL_t\Omega \,\p_X\Omega\, b^{2m}dXdY
= \int_{\R^2_+} \cL_t\p_X\Omega\, \p_X\Omega\, b^{2m}dXdY +\f32\|\p_X \Omega \|_{L^2(m)}^2 \\
&\quad\leq  -\tt^{-2}\|\p_X^2\Omega\|_{L^2(m)}^2-\|\p_X\p_Y\Omega\|_{L^2(m)}^2+C_m \|\p_X\Omega\|_{L^2(m)}^2.
\end{aligned}
\end{equation}

For the nonlinear part in \eqref{eq5.33}, we first decompose it as
\begin{align*}
&\int_{\R^2_+} \p_X \cN_t\Omega\, \p_X\Omega\, b^{2m}dXdY\\
=&\nu^{-\f32}(1+t)^{-\f52} \Bigl(\int_{\R^2_+} \bigl(\p_Y \Delta_t^{-1}\Omega\,\p_X^2\Omega
- \p_X\Delta_t^{-1}\Omega\,\p_Y\p_X\Omega\bigr)\,\p_X\Omega\, b^{2m}dXdY\\
&\qquad+\int_{\R^2_+} \bigl(\p_X\p_Y \Delta_t^{-1}\Omega\,\p_X \Omega
- \p_X^2\Delta_t^{-1}\Omega\, \p_Y\Omega\bigr)\,\p_X\Omega\, b^{2m}dXdY \Bigr)
\eqdef A_1+A_2.
\end{align*}
We deduce from a similar derivation of \eqref{eq5.19} that
$$
A_1\leq C_m \nu^{-\f32} \|\Omega_0\|_{L^2(m)} \|\p_X\Omega\|_{L^2(m)}^2.
$$
For $A_2$, we write
\begin{align*}
A_2 \leq \nu^{-\f32} \tt^{-\f52}\bigl(\|\p_X\p_Y \Delta_t^{-1}\Omega\|_{L^\oo}\|\p_X\Omega\|_{L^2(m)}^2+\|\p_X^2\Delta_t^{-1}\Omega\|_{L^\oo}\|\p_Y\Omega\|_{L^2(m)}\|\p_X\Omega\|_{L^2(m)}\bigr),
\end{align*}
from which and   \eqref{eq5.14}, we infer
\begin{align*}
A_2  \leq &C_\sigma\nu^{-\f32} \tt^{\sigma-\f32} \Bigl(  \|\p_Y\Omega\|_{L^2(1)}^{\f12+\sigma}\|\p_X\p_Y\Omega\|_{L^2(1)}^{\f12-\sigma}\|\p_X\Omega\|_{L^2(m)}^2\\
& \qquad\qquad\qquad+ \|\p_X\Omega\|_{L^2(1)}^{\f12+\sigma}\|\p_X^2\Omega\|_{L^2(1)}^{\f12-\sigma}\|\p_Y\Omega\|_{L^2(m)}\|\p_X\Omega\|_{L^2(m)}\Bigr)\\
\leq &C_\sigma \nu^{-\f32} \tt^{-1}\bigl(\ln\f{1+t}{1+t_0}\bigr)^{\f\sigma2-\f{11}4} \Bigl(  \bigl(\bigl(\ln\f{1+t}{1+t_0}\bigr)^\f12\|\p_Y\Omega\|_{L^2(m)}\bigr)^{\f12+\sigma}\\
&\qquad\times\bigl(\bigl(\ln\f{1+t}{1+t_0}\bigr)^\f12\tt^{-1}\|\p_X\p_Y\Omega\|_{L^2(m)}\bigr)^{\f12-\sigma}\bigl(\bigl(\ln\f{1+t}{1+t_0}\bigr)^\f32\|\p_X\Omega\|_{L^2(m)}\bigr)^{\f12-\sigma}\\
&\quad\times\bigl(\bigl(\ln\f{1+t}{1+t_0}\bigr)\|\p_X\Omega\|_{L^2(m)}\bigr)^{\f32+\sigma}+\bigl(\bigl(\ln\f{1+t}{1+t_0}\bigr) \|\p_X\Omega\|_{L^2(m)}\bigr)^{\f32+\sigma}\\
&\times\bigl(\bigl(\ln\f{1+t}{1+t_0}\bigr)^\f32\tt^{-1}\|\p_X^2\Omega\|_{L^2(m)}\bigr)^{\f12-\sigma}\bigl(\ln\f{1+t}{1+t_0}\bigr)^\f12\|\p_Y\Omega\|_{L^2(m)}\Bigr)\\
\leq &C_\sigma \nu^{-\f32} \tt^{-1} \bigl(\ln\f{1+t}{1+t_0}\bigr)^{\f\sigma2-\f{11}4} \Bigl( \bigl(\ln\f{1+t}{1+t_0}\bigr)\|\p_Y\Omega\|_{L^2(m)}^2+ \bigl(\ln\f{1+t}{1+t_0}\bigr)^3\|\p_X\Omega\|_{L^2(m)}^2 \Bigr)^\f12 \\
&\quad \times \Bigl(\bigl(\ln\f{1+t}{1+t_0}\bigr)\tt^{-2}\|\p_X\p_Y\Omega\|_{L^2(m)}^2+\bigl(\ln\f{1+t}{1+t_0}\bigr)^2 \|\p_X\Omega\|_{L^2(m)}^2\\
&\qquad\qquad+ \bigl(\ln\f{1+t}{1+t_0}\bigr)^3\tt^{-2}\|\p_X^2\Omega\|_{L^2(m)}^2\Bigr).
\end{align*}
By combining the estimates of $A_1$ and $A_2$, we arrive at
\begin{equation}\label{eq5.35}
\begin{aligned}
\int_{\R^2_+} \p_X \cN_t\Omega \,\p_X\Omega\, b^{2m}dXdY \leq &C_m \nu^{-\f32}{\|\Omega_0\|_{L^2(m)}}\|\p_X\Omega\|_{L^2(m)}^2\\
&+{C_{\sigma}}{\nu}^{-\f32}\tt^{-1} \bigl(\ln{t}/{t_0}\bigr)^{\f\sigma2-\f{11}4} E(t)^\f12 D(t).
\end{aligned}
\end{equation}

By substituting  \eqref{eq5.34} and \eqref{eq5.35} into \eqref{eq5.33}, we conclude
\begin{align*}
(1+t) \f{d}{dt}&\Bigl(\bigl(\ln\f{1+t}{1+t_0}\bigr)^3\|\p_X\Omega\|_{L^2(m)}^2\Bigr) +\bigl(\ln\f{1+t}{1+t_0}\bigr)^3\tt^{-2}\|\p_X^2\Omega\|_{L^2(m)}^2\\
&+\bigl(\ln\f{1+t}{1+t_0}\bigr)^3\|\p_X\p_Y\Omega\|_{L^2(m)}^2
\leq 3\bigl(\ln\f{1+t}{1+t_0}\bigr)^2\|\p_X\Omega\|_{L^2(m)}^2\\
&+C_m\bigl(1+\nu^{-\f32}{\|\Omega_0\|_{L^2(m)}}\bigr)\bigl(\ln\f{1+t}{1+t_0}\bigr)^3\|\p_X\Omega\|_{L^2(m)}^2\\
&+C_{\sigma} \nu^{-\f32} \tt^{-1} \bigl(\ln\f{1+t}{1+t_0}\bigr)^{\f\sigma2+\f{1}4} E(t)^\f12 D(t),
\end{align*}
which gives rise to \eqref{eq:lem4}. This finishes the proof of Lemma \ref{S3lem6}.
\end{proof}

\begin{lem}\label{S3lem7}
{   Under the assumptions of Lemma \ref{S3lem6}, we have
\begin{equation}\label{eq:lem5}
\begin{aligned}
(1+t) &\f{d}{dt}\Bigl(\bigl(\ln\f{1+t}{1+t_0}\bigr)^2\|\p_Y^2\Omega\|_{L^2(m)}^2\Bigr) +2\bigl(\ln\f{1+t}{1+t_0}\bigr)^2\tt^{-2}\|\p_X\p_Y^2\Omega\|_{L^2(m)}^2\\
&+2\bigl(\ln\f{1+t}{1+t_0}\bigr)^2\|\p_Y^3\Omega\|_{L^2(m)}^2
\leq 2\bigl(\ln\f{1+t}{1+t_0}\bigr)\|\p_Y^2\Omega\|_{L^2(m)}^2\\
&+4 \bigl(\ln\f{1+t}{1+t_0}\bigr)^2\|\p_X\p_Y\Omega\|_{L^2(m)}\|\p_Y^2\Omega\|_{L^2(m)}\\
&+C_m\bigl(1+{\nu}^{-\f32}{\|\Omega_0\|_{L^2(m)}}\bigr) \bigl(\ln \f{1+t}{1+t_0}\bigr)^\f12 D(t) \\
&+C_{\sigma}{\nu}^{-1}\tt^{-1} \bigl(\ln \f{1+t}{1+t_0}\bigr)^{\f\sigma2+\f14} E(t)^\f12D(t).
\end{aligned}\end{equation}
}\end{lem}
\begin{proof}
By applying $\p_Y^2$ to \eqref{eqs:Om} and then taking $L^2(m)$ inner product of the resulting equation with $\bigl(\ln\f{1+t}{1+t_0}\bigr)^2\p_Y^2\Omega$, we compute
\begin{equation}\label{eq5.37}
\begin{split}
\f{1+t}2 \f{d}{dt}\Bigl(\bigl(\ln\f{1+t}{1+t_0}\bigr)^2&\|\p_Y\Omega\|_{L^2(m)}^2\Bigr) =\bigl(\ln\f{1+t}{1+t_0}\bigr)\|\p_Y^2\Omega\|_{L^2(m)}^2\\
&+\bigl(\ln\f{1+t}{1+t_0}\bigr)^2\int_{\R^2_+} \p_Y^2\Bigl(\cL_t\Omega+\cN_t \Omega\Bigr)\, \p_Y^2\Omega\, b^{2m}dXdY.
\end{split}
\end{equation}

By using \eqref{[py;L]} twice, we find
$$
\p_Y^2\cL_t=\cL_t\p_Y^2 -2\p_X\p_Y+\p_Y^2,
$$
from which and \eqref{eq5.18}, we deduce that 
\begin{equation}\label{eq5.38}
\begin{aligned}
&\int_{\R^2_+} \p_Y^2\cL_t\Omega\, \p_Y^2\Omega\, b^{2m}dXdY\\
\leq &\int_{\R^2_+} \cL_t\p_Y^2\Omega\, \p_Y^2\Omega\, b^{2m}dXdY+2\|\p_X\p_Y\Omega\|_{L^2(m)}\|\p_Y^2\Omega\|_{L^2(m)}+\|\p_Y^2\Omega\|_{L^2(m)}^2\\
\leq &-\tt^{-2}\|\p_X\p_Y^2\Omega\|_{L^2(m)}^2-\|\p_Y^3\Omega\|_{L^2(m)}^2+C_m \|\p_Y^2\Omega\|_{L^2(m)}^2\\
&+2\|\p_X\p_Y\Omega\|_{L^2(m)}\|\p_Y^2\Omega\|_{L^2(m)}.
\end{aligned}
\end{equation}
where we used again $\p_Y^2\Omega|_{Y=0}=0$ to do integration by parts.

To deal with the nonlinear part in \eqref{eq5.37}, we  decompose it as
\begin{align*}
&\int_{\R^2_+} \p_Y^2 \cN_t\Omega \,\p_Y^2\Omega\, b^{2m}dXdY\\=&\nu^{-\f32}(1+t)^{-\f52} \Bigl(\int_{\R^2_+} \bigl(\p_Y \Delta_t^{-1}\Omega\,\p_X\p_Y^2\Omega
- \p_X\Delta_t^{-1}\Omega\, \p_Y^3\Omega\bigr)\,\p_Y^2\Omega\,b^{2m}dXdY\\
&+2\int_{\R^2_+} \bigl(\p_Y^2 \Delta_t^{-1}\Omega\,\p_X\p_Y \Omega
- \p_X\p_Y\Delta_t^{-1}\Omega\, \p_Y^2\Omega\bigr)\,\p_Y^2\Omega\, b^{2m}dXdY\\
&+\int_{\R^2_+} \bigl(\p_Y^3 \Delta_t^{-1}\Omega\,\p_X \Omega
- \p_X\p_Y^2\Delta_t^{-1}\Omega\, \p_Y\Omega\bigr)\,\p_Y^2\Omega\, b^{2m}dXdY \Bigr)
\eqdef B_1+B_2+B_3.
\end{align*}
Along the same line as the derivation of \eqref{eq5.19}, we find
$$
B_1\leq C_{m}{\nu}^{-\f32} {\|\Omega_0\|_{L^2(m)}} \|\p_Y^2\Omega\|_{L^2(m)}^2.
$$
For  $B_2$, we write
\begin{align*}
B_2 \leq 2\nu^{-\f32} \tt^{-\f52}\|b^m\p_Y^2\Omega\|_{L^2_Y(L^\oo_X)}\bigl(&\|\p_Y^2 \Delta_t^{-1}\Omega\|_{L^2}\|b^m\p_X\p_Y\Omega\|_{L^2_X(L^\oo_Y)}\\
&+\|\p_X\p_Y\Delta_t^{-1}\Omega\|_{L^2}\|b^m\p_Y^2\Omega\|_{L^2_X(L^\oo_Y)}\bigr),
\end{align*}
from which,  \eqref{eq:lem3.1b}, \eqref{Sobolev in x} and \eqref{Sobolev in y}, we infer
\begin{align*}
B_2  \leq  &{C}\nu^{-\f32} \tt^{-\f52} \|\p_Y^2\Omega\|_{L^2(m)}^\f12\|\p_X\p_Y^2\Omega\|_{L^2(m)}^\f12
\Bigl( \tt\|\Omega_0\|_{L^2(m)}\|\p_X\p_Y\Omega\|_{L^2(m)}^{\f12}\|\p_X\p_Y^2\Omega\|_{L^2(m)}^{\f12}\\
&\qquad+  \tt^2\|\Omega_0\|_{L^2(m)} \|\p_Y^2\Omega\|_{L^2(m)}^\f12\|\p_Y^3\Omega\|_{L^2(m)}^\f12\Bigr)\\
\leq  &{C}\nu^{-\f32}{\|\Omega_0\|_{L^2(m)}} \bigl(\ln\f{1+t}{1+t_0}\bigr)^{-\f32}\Bigl( \bigl(\ln\f{1+t}{1+t_0}\bigr)^\f12 \|\p_Y^2\Omega\|_{L^2(m)}\Bigr)^\f12\Bigl(\bigl(\ln\f{1+t}{1+t_0}\bigr)\tt^{-1}\\
&\quad\times\|\p_X\p_Y^2\Omega\|_{L^2(m)}\Bigr)^\f12\Bigl( \bigl((\ln\f{1+t}{1+t_0})^\f12\tt^{-1}\|\p_X\p_Y\Omega\|_{L^2(m)}\bigr)^{\f12}\bigl( (\ln\f{1+t}{1+t_0})\tt^{-1}\\
&\qquad\times\|\p_X\p_Y^2\Omega\|_{L^2(m)}\bigr)^{\f12}+  \bigl( (\ln\f{1+t}{1+t_0})^\f12 \|\p_Y^2\Omega\|_{L^2(m)}\bigr)^\f12\bigl( (\ln\f{1+t}{1+t_0})\|\p_Y^3\Omega\|_{L^2(m)}\bigr)^\f12\Bigr)\\
\leq &{ C_m}\nu^{-\f32}  {\|\Omega_0\|_{L^2(m)}}  \bigl(\ln\f{1+t}{1+t_0}\bigr)^{-\f32} D(t).
\end{align*}
Finally, we get, by using \eqref{eq5.14a}, that
\begin{align*}
B_3  \leq & \nu^{-\f32} \tt^{-\f52}\bigl(\|\p_Y^3 \Delta_t^{-1}\Omega\|_{L^\oo}\|\p_X\Omega\|_{L^2(m)}+\|\p_X\p_Y^2\Delta_t^{-1}\Omega\|_{L^\oo}\|\p_Y\Omega\|_{L^2(m)}\bigr)\|\p_Y^2\Omega\|_{L^2(m)}\\
\leq &{C_\sigma}\nu^{-\f32} \tt^{\sigma-\f32} \|\p_Y^2\Omega\|_{L^2(m)}\|\p_Y^2\Omega\|_{L^2(1)}^{\f12+\sigma}\|\p_X\p_Y^2\Omega\|_{L^2(1)}^{\f12-\sigma}\\
&\times\bigl(  \tt^{-1}\|\p_X\Omega\|_{L^2(m)}+\|\p_Y\Omega\|_{L^2(m)}\bigr)\\
\leq &{C_\sigma}\nu^{-\f32} \tt^{-1}\bigl(\ln\f{1+t}{1+t_0}\bigr)^{\f\sigma2-\f{7}4} \Bigl(\bigl(\ln\f{1+t}{1+t_0}\bigr)\|\p_Y^2\Omega\|_{L^2(m)}\Bigr)
 \Bigl(\bigl(\ln\f{1+t}{1+t_0}\bigr)^\f12\|\p_Y^2\Omega\|_{L^2(m)}\Bigr)^{\f12+\sigma}\\
&\quad\times\Bigl(\bigl(\ln\f{1+t}{1+t_0}\bigr)\tt^{-1}\|\p_X\p_Y^2\Omega\|_{L^2(m)}\Bigr)^{\f12-\sigma}\bigl(  \tt^{-1}\|\p_X\Omega\|_{L^2(m)}+\|\p_Y\Omega\|_{L^2(m)} \bigr)\\
\leq &{C_{\sigma}}\nu^{-\f32} \tt^{-1} \bigl(\ln\f{1+t}{1+t_0}\bigr)^{\f\sigma2-\f{7}4} E(t)^\f12 D(t).
\end{align*}
As a consequence, we obtain
\begin{equation}\label{eq5.39}
\begin{aligned}
\int_{\R^2_+} \p_Y^2 \cN_t\Omega\, \p_Y^2\Omega\, b^{2m}dXdY \leq & {C_m}{\nu}^{-\f32} {\|\Omega_0\|_{L^2(m)}}\bigl(\ln \f{1+t}{1+t_0}\bigr)^{-\f32} D(t) \\
&+{C_{\sigma}}{\nu}^{-\f32}\tt^{-1} \bigl(\ln \f{1+t}{1+t_0}\bigr)^{\f\sigma2-\f74} E(t)^\f12D(t) .
\end{aligned}
\end{equation}

By substituting \eqref{eq5.38} and \eqref{eq5.39} into \eqref{eq5.37}, we conclude the proof of \eqref{eq:lem5}.
\end{proof}

\begin{lem}\label{S3lem8}
{  Let $m>1$. For $t_0\leq t\leq 2t_0$, one has
\begin{equation}\label{eq:lem6}
\begin{aligned}
(&1+t) \f{d}{dt}\Bigl(\bigl(\ln\f{1+t}{1+t_0}\bigr)^4\|\p_X\p_Y\Omega\|_{L^2(m)}^2\Bigr) +2\bigl(\ln\f{1+t}{1+t_0}\bigr)^4\tt^{-2}\|\p_X^2\p_Y\Omega\|_{L^2(m)}^2\\
&+2\bigl(\ln\f{1+t}{1+t_0}\bigr)^4\|\p_X\p_Y^2\Omega\|_{L^2(m)}^2
\leq 2\bigl(\ln\f{1+t}{1+t_0}\bigr)^3\|\p_X\p_Y\Omega\|_{L^2(m)}^2\\
&+4\bigl(\ln\f{1+t}{1+t_0}\bigr)^4\|\p_X^2\Omega\|_{L^2(m)}\|\p_X\p_Y\Omega\|_{L^2(m)}+C_m \bigl(1+\nu^{-\f32}{\|\Omega_0\|_{L^2(m)}}\bigr) \bigl(\ln\f{1+t}{1+t_0}\bigr)^{\f12} D(t).
\end{aligned}\end{equation}
}\end{lem}
\begin{proof} We first get,
by applying $\p_X\p_Y$ to \eqref{eqs:Om} and then taking  $L^2(m)$ inner product of the resulting equation with $\bigl(\ln{t}/{t_0}\bigr)^4\p_X\p_Y\Omega$, that
\begin{equation}\label{eq5.41}
    \begin{aligned}
\f{1+t}2 \f{d}{dt}\Bigl(&\bigl(\ln\f{1+t}{1+t_0}\bigr)^4\|\p_X\p_Y\Omega(t)\|_{L^2(m)}^2\Bigr) =2 \bigl(\ln\f{1+t}{1+t_0}\bigr)^3\|\p_X\p_Y\Omega\|_{L^2(m)}^2\\
&+\bigl(\ln\f{1+t}{1+t_0}\bigr)^4\int_{\R^2_+} \p_X\p_Y\bigl(\cL_t\Omega+\cN_t \Omega\bigr)\, \p_X\p_Y\Omega \,b^{2m}dXdY.
\end{aligned}
\end{equation}

It follows from \eqref{[px;L]} and \eqref{[py;L]} that
\begin{equation}\label{[pypx;L]}
\begin{aligned}
\p_X\p_Y\cL_t=\cL_t\p_X\p_Y -\p_X^2+2\p_X\p_Y,
\end{aligned}
\end{equation}
from which and \eqref{eq5.18}, we deduce that
\begin{equation}\label{eq5.43}
\begin{aligned}
&\int_{\R^2_+} \p_X\p_Y\cL_t\Omega\, \p_X\p_Y\Omega\, b^{2m} dXdY\\
&\leq \int_{\R^2_+} \cL_t\p_X\p_Y\Omega\, \p_X\p_Y\Omega\, b^{2m}dXdY +\|\p_X^2\Omega\|_{L^2(m)}\|\p_X\p_Y\Omega\|_{L^2(m)}+2\|\p_X\p_Y\Omega\|_{L^2(m)}^2 \\
&\leq -\tt^{-2}\|\p_X^2\p_Y\Omega\|_{L^2(m)}^2-\|\p_X\p_Y^2\Omega\|_{L^2(m)}^2+C_m \|\p_X\p_Y\Omega\|_{L^2(m)}^2\\
&\quad+\|\p_X^2\Omega\|_{L^2(m)}\|\p_X\p_Y\Omega\|_{L^2(m)}.
\end{aligned}
\end{equation}

To deal with the nonlinear part in \eqref{eq5.41}, we get, by using integration by parts and Poincare's inequality, that
\begin{align*}
\int_{\R^2_+} \p_X \p_Y\cN_t\Omega\, \p_X\p_Y\Omega\, b^{2m}dXdY
&=-\int_{\R^2_+} \p_Y\cN_t\Omega \,\bigl(\p_X^2\p_Y\Omega b^{2m}+2m\,\p_X\p_Y \Omega\, b^{2m-2}X\bigr)dXdY\\
&\leq C_m\|\p_Y\cN_t\Omega \|_{L^2(m)} \bigl( \|\p_X^2\p_Y\Omega\|_{L^2(m)}+\|\p_X\p_Y\Omega\|_{L^2(m-1)} \bigr)\\
&\leq C_m \|\p_Y\cN_t\Omega \|_{L^2(m)}  \|\p_X^2\p_Y\Omega\|_{L^2(m)}.
\end{align*}
 Whereas we deduce from \eqref{eq:lem3.1a}, \eqref{eq:lem3.1b}, \eqref{Sobolev in x} and \eqref{Sobolev in y} that
\begin{align*}
\|\p_Y\cN_t\Omega\|_{L^2(m)}\lesssim &\nu^{-\f32}\tt^{-\f52} \Bigl( \|\p_X\Delta_t^{-1}\Omega\|_{L^\oo}\|\p_Y^2\Omega\|_{L^2(m)}+\|\p_Y\Delta_t^{-1}\Omega\|_{L^\oo}\|\p_X\p_Y\Omega\|_{L^2(m)} \\
&+\|\p_X\p_Y\Delta_t^{-1}\Omega\|_{L^2}\|b^m\p_Y\Omega\|_{L^\oo}+\|\p_Y^2\Delta_t^{-1}\Omega\|_{L^2}\|b^m\p_X\Omega\|_{L^\oo}\Bigr)\\
\lesssim & \nu^{-\f32}{\|\Omega_0\|_{L^2(m)}} \Bigl( \|\p_Y^2\Omega\|_{L^2(m)} +\tt^{-1}\|\p_X\p_Y\Omega\|_{L^2(m)}\\
&+\tt^{-\f12}\|\p_X\p_Y^2\Omega\|_{L^2(m)}^\f14 \|\p_X\p_Y\Omega\|_{L^2(m)}^\f14 \|\p_Y^2\Omega\|_{L^2(m)}^\f14 \|\p_Y\Omega\|_{L^2(m)}^\f14\\
&+\tt^{-\f32}\|\p_X^2\p_Y\Omega\|_{L^2(m)}^\f14 \|\p_X^2\Omega\|_{L^2(m)}^\f14 \|\p_X\p_Y\Omega\|_{L^2(m)}^\f14 \|\p_X\Omega\|_{L^2(m)}^\f14 \Bigr)\\
\lesssim & \nu^{-\f32} {\|\Omega_0\|_{L^2(m)}}\Bigl( \|\p_Y^2\Omega\|_{L^2(m)} +\tt^{-1}\|\p_X\p_Y\Omega\|_{L^2(m)}\\
&+\bigl(\ln\f{1+t}{1+t_0}\bigr)^\f12\tt^{-1}\|\p_X\p_Y^2\Omega\|_{L^2(m)}+\bigl(\ln\f{1+t}{1+t_0}\bigr)^{-\f12}\|\p_Y\Omega\|_{L^2(m)}\Bigr)\\
&+\nu^{-\f32}{\|\Omega_0\|_{L^2(m)}}\tt^{-1}\Bigl(\tt^{-1}\|\p_X^2\Omega\|_{L^2(m)}+\|\p_X\p_Y\Omega\|_{L^2(m)}\\
&+\bigl(\ln\f{1+t}{1+t_0}\bigr)^\f12\tt^{-1}\|\p_X^2\p_Y\Omega\|_{L^2(m)}+\bigl(\ln\f{1+t}{1+t_0}\bigr)^{-\f12}\|\p_X\Omega\|_{L^2(m)} \Bigr).
\end{align*}
As a result, it comes out
\begin{align*}
\int_{\R^2_+}& \p_X \p_Y\cN_t\Omega\, \p_X\p_Y\Omega\, b^{2m}dXdY \leq C_m \nu^{-\f32}\|\Omega_0\|_{L^2(m)} \bigl(\ln\f{1+t}{1+t_0}\bigr)^{-\f72}\\
 & \quad\times\Bigl(\bigl(\ln\f{1+t}{1+t_0}\bigr)^6 \|\p_X^2\p_Y\Omega\|_{L^2(m)}^2  \Bigr)^\f12 \Bigl(\|\p_Y\Omega\|_{L^2(m)}^2+\bigl(\ln\f{1+t}{1+t_0}\bigr) \|\p_Y^2\Omega\|_{L^2(m)}^2 \\
&\qquad+\bigl(\ln\f{1+t}{1+t_0}\bigr)\tt^{-2}\|\p_X\p_Y\Omega\|_{L^2(m)}^2+\bigl(\ln\f{1+t}{1+t_0}\bigr)^2\tt^{-2}\|\p_X\p_Y^2\Omega\|_{L^2(m)}^2\Bigr)^\f12\\
&+C_m \nu^{-\f32}{\|\Omega_0\|_{L^2(m)}} \bigl(\ln\f{1+t}{1+t_0}\bigr)^{-\f72} \Bigl(\bigl(\ln\f{1+t}{1+t_0}\bigr)^4 \tt^{-2} \|\p_X^2\p_Y\Omega\|_{L^2(m)}^2  \Bigr)^\f12 \\
&\quad\times \Bigl( \bigl(\ln\f{1+t}{1+t_0}\bigr)^3\bigl(\tt^{-2}\|\p_X^2\Omega\|_{L^2(m)}^2+\|\p_X\p_Y\Omega\|_{L^2(m)}^2 \bigr)\\
 &\qquad+\bigl(\ln\f{1+t}{1+t_0}\bigr)^2 \|\p_X\Omega\|_{L^2(m)}^2+\bigl(\ln\f{1+t}{1+t_0}\bigr)^4\tt^{-2}\|\p_X^2\p_Y\Omega\|_{L^2(m)}^2\Bigr)^\f12,
\end{align*}
from which and \eqref{def:D}, we infer
\begin{equation}\label{eq5.44}
\int_{\R^2_+} \p_X \p_Y\cN_t\Omega\, \p_X\p_Y\Omega\, b^{2m}dXdY \leq C_m \nu^{-\f32}{\|\Omega_0\|_{L^2(m)}}\bigl(\ln\f{1+t}{1+t_0}\bigr)^{-\f72} D(t).
\end{equation}

By substituting \eqref{eq5.43} and \eqref{eq5.44} into \eqref{eq5.41}, we  conclude the proof of \eqref{eq:lem6}.
\end{proof}

\begin{lem}\label{S3lem9}
{  Let $m>1$ and $0<\sigma<\f12$. Then for $t_0\leq t\leq 2t_0 $, we have
\begin{equation}\label{eq:lem7}
\begin{aligned}
&(1+t)\f{d}{dt}\Bigl( \bigl(\ln\f{1+t}{1+t_0}\bigr)^5 \bigl(\p_X^2 \Omega \big|\p_X\p_Y \Omega \bigr)_{L^2(m)} \Bigr)+\bigl(\ln\f{1+t}{1+t_0}\bigr)^5\|\p_X^2\Omega\|_{L^2(m)}^2\\
\leq & 5\bigl(\ln\f{1+t}{1+t_0}\bigr)^4\bigl(\p_X^2 \Omega \big|\p_X\p_Y \Omega \bigr)_{L^2(m)}+2\bigl(\ln\f{1+t}{1+t_0}\bigr)^5\tt^{-2}\|\p_X^3\Omega\|_{L^2(m)}\|\p_X^2\p_Y\Omega\|_{L^2(m)}\\
&+2\bigl(\ln\f{1+t}{1+t_0}\bigr)^5\|\p_X^2\p_Y\Omega\|_{L^2(m)}\|\p_X\p_Y^2\Omega\|_{L^2(m)}\\
&+C_m \bigl(1+{\nu}^{-\f32}{\|\Omega_0\|_{L^2(m)}} \bigr)\bigl(\ln \f{1+t}{1+t_0}\bigr)^{\f12} D(t)\\
&+C_{m,\sigma} \nu^{-\f32} \tt^{-1} \bigl(\ln \f{1+t}{1+t_0}\bigr)^{\f\sigma2+\f14} E(t)^\f12 D(t).
\end{aligned}\end{equation}
}\end{lem}
\begin{proof} In view of \eqref{eqs:Om}, we write
\begin{equation}\label{eq5.46}
\begin{aligned}
(1&+t)\f{d}{dt}\Bigl( \bigl(\ln\f{1+t}{1+t_0}\bigr)^5 \bigl(\p_X^2\Omega \big|\p_X\p_Y \Omega \bigr)_{L^2(m)} \Bigr)=5\bigl(\ln\f{1+t}{1+t_0}\bigr)^4\bigl(\p_X^2 \Omega \big|\p_X\p_Y \Omega \bigr)_{L^2(m)} \\
&+\bigl(\ln\f{1+t}{1+t_0}\bigr)^5 \Bigl(\int_{\R^2_+}\p_X^2 \bigl(\cL_t\Omega+\cN_t\Omega \bigr)\,\p_X\p_Y\Omega \,b^{2m}dXdY \\
&\qquad\qquad\qquad\qquad+ \int_{\R^2_+}\p_X^2 \Omega\, \p_X\p_Y\bigl(\cL_t\Omega+\cN_t\Omega \bigr)\,b^{2m}dXdY\Bigr).
\end{aligned}
\end{equation}

By applying \eqref{[px;L]} twice, we find
\begin{equation}\label{[pxpx;L]}
\p_X^2 \cL_t=\cL_t\p_X^2 +3 \p_X^2,
\end{equation}
from which and \eqref{[pypx;L]}, we infer
\begin{align*}
\int_{\R^2_+}&\p_X^2 \cL_t\Omega\,\p_X\p_Y\Omega\,b^{2m}dXdY
+ \int_{\R^2_+}\p_X^2 \Omega\, \p_X\p_Y\cL_t\Omega\, b^{2m}dXdY\\
=&\int_{\R^2_+}\cL_t\p_X^2 \Omega\,\p_X\p_Y\Omega\, b^{2m}dXdY
+ \int_{\R^2_+}\p_X^2 \Omega\,\cL_t \p_X\p_Y\Omega \,b^{2m}dXdY-\|\p_X^2\Omega\|_{L^2(m)}^2 \\
&+5\int_{\R^2_+} \p_X^2\Omega \,\p_X\p_Y\Omega\, b^{2m}dXdY.
\end{align*}
By using integration by parts, we have
\begin{align*}
&\int_{\R^2_+}\cL_t\p_X^2 \Omega\,\p_X\p_Y\Omega\, b^{2m}dXdY
+ \int_{\R^2_+}\p_X^2 \Omega\,\cL_t \p_X\p_Y\Omega \,b^{2m}dXdY\\
\leq &2\tt^{-2}\|\p_X^3\Omega\|_{L^2(m)}\|\p_X^2\p_Y\Omega\|_{L^2(m)}+2\|\p_X^2\p_Y\Omega\|_{L^2(m)}\|\p_X\p_Y^2\Omega\|_{L^2(m)}^2\\
&+C_m\|\p_X^2\Omega\|_{L^2(m)}\|\p_X\p_Y\Omega\|_{L^2(m)}.
\end{align*}
Therefore, for the linear terms, we conclude that 
\begin{equation}\label{eq5.48}
\begin{aligned}
\int_{\R^2_+}&\p_X^2 \cL_t\Omega\,\p_X\p_Y\Omega \,b^{2m}dXdY + \int_{\R^2_+}\p_X^2 \Omega\, \p_X\p_Y\cL_t\Omega\, b^{2m}dXdY\leq-\|\p_X^2\Omega\|_{L^2(m)}^2\\
 &+2\tt^{-2}\|\p_X^3\Omega\|_{L^2(m)}\|\p_X^2\p_Y\Omega\|_{L^2(m)}+2\|\p_X^2\p_Y\Omega\|_{L^2(m)}\|\p_X\p_Y^2\Omega\|_{L^2(m)}^2\\
&+C_m\|\p_X^2\Omega\|_{L^2(m)}\|\p_X\p_Y\Omega\|_{L^2(m)}.
\end{aligned}
\end{equation}

To deal with the nonlinear terms in \eqref{eq5.46}, we use integration by parts to write
\begin{align*}
\int_{\R^2_+}&\p_X^2 \cN_t\Omega\,\p_X\p_Y\Omega\, b^{2m}dXdy
+ \int_{\R^2_+}\p_X^2 \Omega\, \p_X\p_Y\cN_t\Omega\, b^{2m} dXdY\\
\leq &C_m\|\p_X\cN_t\Omega\|_{L^2(m)}\bigl(\|\p_X^2\p_Y\Omega\|_{L^2(m)}+\|\p_X\p_Y\Omega\|_{L^2(m-1)}+\|\p_X^2\Omega\|_{L^2(m-1)} \bigr)\\
\leq &C_m\|\p_X\cN_t\Omega\|_{L^2(m)}\|\p_X^2\p_Y\Omega\|_{L^2(m)}.
\end{align*}
Yet it is easy to observe from \eqref{eqs:Om} that
\begin{align*}
\|\p_X\cN_t\Omega\|_{L^2(m)}\lesssim &\nu^{-\f32}\tt^{-\f52} \Bigl( \|\p_X\Delta_t^{-1}\Omega\|_{L^\oo}\|\p_X\p_Y\Omega\|_{L^2(m)}+\|\p_Y\Delta_t^{-1}\Omega\|_{L^\oo}\|\p_X^2\Omega\|_{L^2(m)} \\
&\qquad+\|\p_X\p_Y\Delta_t^{-1}\Omega\|_{L^2}\|b^m\p_X\Omega\|_{L^\oo}+\|\p_X^2\Delta_t^{-1}\Omega\|_{L^\oo}\|\p_Y\Omega\|_{L^2(m)}\Bigr).
\end{align*}
It follows from  \eqref{eq:lem3.1a} that
\begin{align*}
\nu^{-\f32}&\tt^{-\f52} \bigl( \|\p_X\Delta_t^{-1}\Omega\|_{L^\oo}\|\p_X\p_Y\Omega\|_{L^2(m)}+\|\p_Y\Delta_t^{-1}\Omega\|_{L^\oo}\|\p_X^2\Omega\|_{L^2(m)}\bigr)\\
\leq & {C}\nu^{-\f32}{\|\Omega_0\|_{L^2(m)}}\bigl( \|\p_X\p_Y\Omega\|_{L^2(m)}+\tt^{-1}\|\p_X^2\Omega\|_{L^2(m)}\bigr)\\
\leq &{C}\nu^{-\f32} {\|\Omega_0\|_{L^2(m)}}  \bigl(\ln \f{1+t}{1+t_0}\bigr)^{-\f32} D(t)^\f12.
\end{align*}
Along the same line, we deduce from \eqref{eq:lem3.1b} that
\begin{align*}
\nu^{-\f32}&\tt^{-\f52} \|\p_X\p_Y\Delta_t^{-1}\Omega\|_{L^2}\|b^m\p_X\Omega\|_{L^\oo}\\
\leq & C\nu^{-\f32}{\|\Omega_0\|_{L^2(m)}} \tt^{-\f12} \|\p_X\Omega\|_{L^2(m)}^\f14\|\p_X\p_Y\Omega\|_{L^2(m)}^\f14\|\p_X^2\Omega\|_{L^2(m)}^\f14  \|\p_X^2\p_Y\Omega\|_{L^2(m)}^\f14\\
\leq & C\nu^{-\f32}{\|\Omega_0\|_{L^2(m)}} \bigl(\ln\f{1+t}{1+t_0}\bigr)^{-\f32} \Bigl( \bigl(\ln\f{1+t}{1+t_0}\bigr)\|\p_X\Omega\|_{L^2(m)}\Bigr)^\f14
\Bigl( \bigl(\ln\f{1+t}{1+t_0}\bigr)^\f32\|\p_X\p_Y\Omega\|_{L^2(m)}\Bigr)^\f14 \\
&\qquad\times\Bigl( \bigl(\ln\f{1+t}{1+t_0}\bigr)^\f32\tt^{-1}\|\p_X^2\Omega\|_{L^2(m)}\Bigr)^\f14
\Bigl( \bigl(\ln\f{1+t}{1+t_0}\bigr)^2\tt^{-1}\|\p_X^2\p_Y\Omega\|_{L^2(m)}\Bigr)^\f14\\
\leq &C_m\nu^{-\f32}  {\|\Omega_0\|_{L^2(m)}}\bigl(\ln {t}/{t_0}\bigr)^{-\f32} D(t)^\f12.
\end{align*}
Whereas we deduce from \eqref{eq5.14}  that for any $0<\sigma<\f12$,
\begin{align*}
\nu^{-\f32}&\tt^{-\f52} \|\p_X^2\Delta_t^{-1}\Omega\|_{L^\oo}\|\p_Y\Omega\|_{L^2(m)}\\
\leq & C_\sigma\nu^{-\f32} \tt^{\sigma-\f32} \|\p_X\Omega\|_{L^2(m)}^{\f12+\sigma}\|\p_X^2\Omega\|_{L^2(m)}^{\f12-\sigma}\|\p_Y\Omega\|_{L^2(m)}\\
\leq & C_\sigma \nu^{-\f32} \tt^{-1} \bigl(\ln \f{1+t}{1+t_0}\bigr)^{\f\sigma2-\f74} \Bigl( \bigl(\ln \f{1+t}{1+t_0}\bigr)\|\p_X\Omega\|_{L^2(m)} \Bigr)^{\f12+\sigma}  \\
&\qquad\times\Bigl( \bigl(\ln \f{1+t}{1+t_0}\bigr)^{\f32}\tt^{-1}\|\p_X^2\Omega\|_{L^2(m)} \Bigr)^{\f12-\sigma}\Bigl( \bigl(\ln \f{1+t}{1+t_0}\bigr)^\f12\|\p_Y\Omega\|_{L^2(m)}\Bigr)\\
\leq & C_{\sigma} \nu^{-\f32} \tt^{-1} \bigl(\ln \f{1+t}{1+t_0}\bigr)^{\f\sigma2-\f74} E(t)^\f12 D(t)^\f12.
\end{align*}
Noticing that
$$
\|\p_X^2\p_Y\Omega\|_{L^2(m)} \leq C \bigl(\ln {t}/{t_0}\bigr)^{-3} D(t)^\f12,
$$
we get, by summarizing the above estimates, that
\begin{equation}\label{eq5.49}
\begin{aligned}
&\int_{\R^2_+}\p_X^2 \cN_t\Omega\,\p_X\p_Y\Omega\, b^{2m}dXdY
+ \int_{\R^2_+}\p_X^2 \Omega\, \p_X\p_Y\cN_t\Omega\, b^{2m}dXdY\\
&\leq C_m {\nu}^{-\f32}{\|\Omega_0\|_{L^2(m)}} \bigl(\ln \f{1+t}{1+t_0}\bigr)^{-\f92} D(t)+C_{m,\sigma} \nu^{-\f32} \tt^{-1} \bigl(\ln \f{1+t}{1+t_0}\bigr)^{\f\sigma2-\f{19}4} E(t)^\f12 D(t).
\end{aligned}
\end{equation}

By substituting the estimates \eqref{eq5.48} and \eqref{eq5.49} into \eqref{eq5.46}, we arrive at \eqref{eq:lem7}. This
finishes the proof of Lemma \ref{S3lem9}.
\end{proof}

\begin{lem}\label{S3lem10}
{  Under the assumptions of Lemma \ref{S3lem9}, we have
\begin{equation}\label{eq:lem8}
\begin{aligned}
(1+t) &\f{d}{dt}\Bigl(\bigl(\ln\f{1+t}{1+t_0}\bigr)^6\|\p_X^2\Omega(t)\|_{L^2(m)}^2\Bigr) +2\bigl(\ln\f{1+t}{1+t_0}\bigr)^6\tt^{-2}\|\p_X^3\Omega\|_{L^2(m)}^2\\
&+2\bigl(\ln\f{1+t}{1+t_0}\bigr)^6\|\p_X^2\p_Y\Omega\|_{L^2(m)}^2
\leq 6\bigl(\ln\f{1+t}{1+t_0}\bigr)^5\|\p_X^2\Omega\|_{L^2(m)}^2\\
&+C_m\bigl(1+\nu^{-\f32}{\|\Omega_0\|_{L^2(m)}}\bigr)E(t)+C_{m,\sigma} \nu^{-\f32} \tt^{-1} \bigl(\ln\f{1+t}{1+t_0}\bigr)^{\f\sigma2+\f{1}4} E(t)^\f12 D(t).
\end{aligned}\end{equation}
}\end{lem}
\begin{proof} We first get,
by applying $\p_X^2$ to \eqref{eqs:Om} and then taking  $L^2(m)$ inner product of the resulting equation with $\bigl(\ln\f{1+t}{1+t_0}\bigr)^6\p_X^2\Omega$, that
\begin{equation}\label{eq5.51}\begin{aligned}
\f{1+t}2 \f{d}{dt}\Bigl(\bigl(\ln\f{1+t}{1+t_0}\bigr)^6&\|\p_X^2\Omega(t)\|_{L^2(m)}^2\Bigr) =3\bigl(\ln\f{1+t}{1+t_0}\bigr)^5\|\p_X^2\Omega\|_{L^2(m)}^2\\
&+\bigl(\ln\f{1+t}{1+t_0}\bigr)^6\int_{\R^2_+} \p_X^2\Bigl(\cL_t\Omega+\cN_t \Omega\Bigr)\, \p_X^2\Omega\, b^{2m}dXdY.
\end{aligned}
\end{equation}

By applying \eqref{[pxpx;L]} and \eqref{eq5.18}, we deduce  that 
\begin{equation}\label{eq5.52}
\begin{aligned}
\int_{\R^2_+} &\p_X^2\cL_t\Omega \,\p_X^2\Omega\, b^{2m}dXdY
= \int_{\R^2_+} \cL_t\p_X^2\Omega\, \p_X^2\Omega \,b^{2m}dXdY +3\|\p_X^2\Omega\|_{L^2(m)}\\
\leq &-\tt^{-2}\|\p_X^3\Omega\|_{L^2(m)}^2-\|\p_X^2\p_Y\Omega\|_{L^2(m)}^2+C_m \|\p_X^2\Omega\|_{L^2(m)}^2.
\end{aligned}
\end{equation}

For the nonlinear part in \eqref{eq5.51}, we first decompose it as
\begin{align*}
\int_{\R^2_+} \p_X^2 \cN_t\Omega \,&\p_X^2\Omega\, b^{2m}dXdY\\
=&\nu^{-\f32}(1+t)^{-\f52} \Bigl(\int_{\R^2_+} \bigl(\p_Y \Delta_t^{-1}\Omega\,\p_X^3\Omega
- \p_X\Delta_t^{-1}\Omega \,\p_Y\p_X^2\Omega\bigr)\,\p_X^2\Omega\, b^{2m}dXdY\\
&+2\int_{\R^2_+} \bigl(\p_X\p_Y \Delta_t^{-1}\Omega\,\p_X^2 \Omega
- \p_X^2\Delta_t^{-1}\Omega\, \p_X\p_Y\Omega\bigr)\,\p_X^2\Omega\,b^{2m}dXdY\\
&+\int_{\R^2_+} \bigl(\p_X^2\p_Y \Delta_t^{-1}\Omega\,\p_X \Omega
- \p_X^3\Delta_t^{-1}\Omega\, \p_Y\Omega\bigr)\,\p_X\Omega\, b^{2m}dXdY\Bigr)
\eqdef D_1+D_2+D_3.
\end{align*}
It follows from a similar derivation of \eqref{eq5.19} that
$$
D_1\leq C_{m} {\nu}^{-\f32}{\|\Omega_0\|_{L^2(m)}} \|\p_X^2\Omega\|_{L^2(m)}^2.
$$
For  $D_2$, we write
\begin{align*}
D_2 \leq 2\nu^{-\f32} \tt^{-\f52}\Bigl(&\|\p_X\p_Y \Delta_t^{-1}\Omega\|_{L^\oo}\|\p_X^2\Omega\|_{L^2(m)}^2\\
&+\|\p_X^2\Delta_t^{-1}\Omega\|_{L^\oo}\|\p_X\p_Y\Omega\|_{L^2(m)}\|\p_X^2\Omega\|_{L^2(m)}\Bigr),
\end{align*}
from which and  \eqref{eq5.14}, we infer
\begin{align*}
D_2  \leq &C_\sigma\nu^{-\f32} \tt^{\sigma-\f32} \Bigl(  \|\p_Y\Omega\|_{L^2(m)}^{\f12+\sigma}\|\p_X\p_Y\Omega\|_{L^2(m)}^{\f12-\sigma}\|\p_X^2\Omega\|_{L^2(m)}^2\\
& \qquad\qquad\qquad\qquad+ \|\p_X\Omega\|_{L^2(m)}^{\f12+\sigma}\|\p_X^2\Omega\|_{L^2(m)}^{\f12-\sigma}\|\p_X\p_Y\Omega\|_{L^2(m)}\|\p_X^2\Omega\|_{L^2(m)}\Bigr)\\
\leq &C_\sigma \nu^{-\f32} \tt^{-1}\bigl(\ln\f{1+t}{1+t_0}\bigr)^{\f\sigma2-\f{23}4} \Bigl(  \bigl((\ln\f{1+t}{1+t_0})^\f12\|\p_Y\Omega\|_{L^2(m)}\bigr)^{\f12+\sigma}\\
&\times\bigl((\ln\f{1+t}{1+t_0})^\f12\tt^{-1}\|\p_X\p_Y\Omega\|_{L^2(m)}\bigr)^{\f12-\sigma}\bigl((\ln\f{1+t}{1+t_0})^3\|\p_X^2\Omega\|_{L^2(m)}\bigr)^{\f12-\sigma}\\
&\times\bigl((\ln\f{1+t}{1+t_0})^\f52\|\p_X^2\Omega\|_{L^2(m)}\bigr)^{\f32+\sigma}+\bigl((\ln\f{1+t}{1+t_0}) \|\p_X\Omega\|_{L^2(m)}\bigr)^{\f12+\sigma}\\
&\times\bigl((\ln\f{1+t}{1+t_0})^\f32\tt^{-1}\|\p_X^2\Omega\|_{L^2(m)}\bigr)^{\f12-\sigma}\bigl((\ln\f{1+t}{1+t_0})^\f32\|\p_X\p_Y\Omega\|_{L^2(m)}\bigr)\\
&\times\bigl((\ln\f{1+t}{1+t_0})^3\|\p_X^2\Omega\|_{L^2(m)}\bigr)\Bigr)\\
\leq &C_\sigma \nu^{-\f32} \tt^{-1} \bigl(\ln\f{1+t}{1+t_0}\bigr)^{\f\sigma2-\f{23}4} \Bigl( \bigl(\ln\f{1+t}{1+t_0}\bigr)\|\p_Y\Omega\|_{L^2(m)}^2+ \bigl(\ln\f{1+t}{1+t_0}\bigr)^6\|\p_X^2\Omega\|_{L^2(m)}^2 \Bigr)^\f12 \\
& \times \Bigl(\bigl(\ln\f{1+t}{1+t_0}\bigr)\tt^{-2}\|\p_X\p_Y\Omega\|_{L^2(m)}^2+\bigl(\ln\f{1+t}{1+t_0}\bigr)^5 \|\p_X^2\Omega\|_{L^2(m)}^2\\
&+\bigl(\ln\f{1+t}{1+t_0}\bigr)^2 \|\p_X\Omega\|_{L^2(m)}^2+ \bigl(\ln\f{1+t}{1+t_0}\bigr)^3\tt^{-2}\|\p_X^2\Omega\|_{L^2(m)}^2\\
&+\bigl(\ln\f{1+t}{1+t_0}\bigr)^3\|\p_X\p_Y\Omega\|_{L^2(m)}^2\Bigr)\\
\leq &C_{\sigma} \nu^{-\f32} \tt^{-1} \bigl(\ln{t}/{t_0}\bigr)^{\f\sigma2-\f{23}4} E(t)^\f12 D(t).
\end{align*}
Along the same line, we get, by using
 \eqref{eq5.14}, that
\begin{align*}
D_3  \leq & \nu^{-\f32} \tt^{-\f52}\Bigl(\|\p_X^2\p_Y \Delta_t^{-1}\Omega\|_{L^\oo}\|\p_X\Omega\|_{L^2(m)}+\|\p_X^3\Delta_t^{-1}\Omega\|_{L^\oo}\|\p_Y\Omega\|_{L^2(m)}\Bigr)\|\p_X^2\Omega\|_{L^2(m)}\\
\leq &C_\sigma\nu^{-\f32} \tt^{\sigma-\f32} \|\p_X^2\Omega\|_{L^2(m)}\Bigl(  \|\p_X\p_Y\Omega\|_{L^2(m)}^{\f12+\sigma}\|\p_X^2\p_Y\Omega\|_{L^2(m)}^{\f12-\sigma}\|\p_X\Omega\|_{L^2(m)}\\
& \qquad\qquad\qquad\qquad\qquad\qquad\qquad+ \|\p_X^2\Omega\|_{L^2(m)}^{\f12+\sigma}\|\p_X^3\Omega\|_{L^2(m)}^{\f12-\sigma}\|\p_Y\Omega\|_{L^2(m)}\Bigr)\\
\leq &C_\sigma \nu^{-\f32} \tt^{-1}\bigl(\ln\f{1+t}{1+t_0}\bigr)^{\f\sigma2-\f{23}4} \Bigl(\bigl(\ln\f{1+t}{1+t_0}\bigr)^3\|\p_X^2\Omega\|_{L^2(m)}\Bigr) \\
&\times\Bigl(  \bigl((\ln\f{1+t}{1+t_0})^\f32\|\p_X\p_Y\Omega\|_{L^2(m)}\bigr)^{\f12+\sigma}\bigl((\ln\f{1+t}{1+t_0})^2\tt^{-1}\|\p_X^2\p_Y\Omega\|_{L^2(m)}\bigr)^{\f12-\sigma}\\
&\quad\times\bigl((\ln\f{1+t}{1+t_0})\|\p_X\Omega\|_{L^2(m)}\bigr)+\bigl((\ln\f{1+t}{1+t_0})^\f52 \|\p_X^2\Omega\|_{L^2(m)}\bigr)^{\f12+\sigma}\\
&\qquad\times\bigl((\ln\f{1+t}{1+t_0})^3\tt^{-1}\|\p_X^3\Omega\|_{L^2(m)}\bigr)^{\f12-\sigma}\|\p_Y\Omega\|_{L^2(m)}\Bigr)\\
\leq &C_\sigma \nu^{-\f32} \tt^{-1} \bigl(\ln\f{1+t}{1+t_0}\bigr)^{\f\sigma2-\f{23}4} \Bigl( \bigl(\ln\f{1+t}{1+t_0}\bigr)^6\|\p_X^2\Omega\|_{L^2(m)}^2 \Bigr)^\f12 \Bigl(\|\p_Y\Omega\|_{L^2(m)}^2+\\
&  +\bigl(\ln\f{1+t}{1+t_0}\bigr)^2 \|\p_X\Omega\|_{L^2(m)}^2+\bigl(\ln\f{1+t}{1+t_0}\bigr)^3\|\p_X\p_Y\Omega\|_{L^2(m)}^2+\bigl(\ln\f{1+t}{1+t_0}\bigr)^5 \|\p_X^2\Omega\|_{L^2(m)}^2\\
&+\bigl(\ln\f{1+t}{1+t_0}\bigr)^4\tt^{-2}\|\p_X^2\p_Y\Omega\|_{L^2(m)}^2+\bigl(\ln\f{1+t}{1+t_0}\bigr)^6\tt^{-2}\|\p_X^3\Omega\|_{L^2(m)}^2+\Bigr)\\
\leq &C_{\sigma} \nu^{-\f32} \tt^{-1} \bigl(\ln\f{1+t}{1+t_0}\bigr)^{\f\sigma2-\f{23}4} E(t)^\f12 D(t).
\end{align*}
By summarizing the above estimates, we arrive at
\begin{equation}\label{eq5.53}
\begin{aligned}
\int_{\R^2_+} \p_X^2 \cN_t\Omega\, \p_X^2\Omega\, b^{2m}dXdY \leq &C_m\nu^{-\f32}{\|\Omega_0\|_{L^2(m)}}\|\p_X^2\Omega\|_{L^2(m)}^2 \\
&+C_{m,\sigma} \nu^{-\f32} \tt^{-1} \bigl(\ln\f{1+t}{1+t_0}\bigr)^{\f\sigma2-\f{23}4} E(t)^\f12 D(t).
\end{aligned}
\end{equation}

By inserting the estimates \eqref{eq5.52} and \eqref{eq5.53} into \eqref{eq5.51}, we  conclude the proof of \eqref{eq:lem8}.
\end{proof}

For terms with three derivatives, one can easily repeat the above approaches to prove the following estimates.

\begin{lem}\label{S3lem13}
{   Under the assumptions of Lemma \ref{S3lem9}, we have
\begin{equation}\label{eq:lem11}
\begin{aligned}
(1+t) &\f{d}{dt}\Bigl(\bigl(\ln\f{1+t}{1+t_0}\bigr)^7\|\p_X^2\p_Y\Omega\|_{L^2(m)}^2\Bigr) +2\bigl(\ln\f{1+t}{1+t_0}\bigr)^7\tt^{-2}\|\p_X^3\p_Y\Omega\|_{L^2(m)}^2\\
&+2\bigl(\ln\f{1+t}{1+t_0}\bigr)^7\|\p_X^2\p_Y^2\Omega\|_{L^2(m)}^2
\leq 7\bigl(\ln\f{1+t}{1+t_0}\bigr)^6\|\p_X^2\p_Y\Omega\|_{L^2(m)}^2\\
&+2 \bigl(\ln\f{1+t}{1+t_0}\bigr)^7\|\p_X^3\Omega\|_{L^2(m)}\|\p_X^2\p_Y\Omega\|_{L^2(m)}\\
&+C_m\bigl(1+{\nu}^{-\f32}{\|\Omega_0\|_{L^2(m)}}\bigr) \bigl(\ln \f{1+t}{1+t_0}\bigr)^\f12 D(t) \\
&+C_{\sigma}{\nu}^{-1}\tt^{-1} \bigl(\ln \f{1+t}{1+t_0}\bigr)^{\f\sigma2+\f14} E(t)^\f12D(t).
\end{aligned}\end{equation}
}\end{lem}

\begin{lem}\label{S3lem14}
{  Under the assumptions of Lemma \ref{S3lem9}, we have
\begin{equation}\label{eq:lem12}
\begin{aligned}
&(1+t)\f{d}{dt}\Bigl( \bigl(\ln\f{1+t}{1+t_0}\bigr)^8 \bigl(\p_X^3 \Omega \big|\p_X^2\p_Y \Omega \bigr)_{L^2(m)} \Bigr)+\bigl(\ln\f{1+t}{1+t_0}\bigr)^8\|\p_X^3\Omega\|_{L^2(m)}^2\\
\leq & 8\bigl(\ln\f{1+t}{1+t_0}\bigr)^7\bigl(\p_X^3 \Omega \big|\p_X^2\p_Y \Omega \bigr)_{L^2(m)}+2\bigl(\ln\f{1+t}{1+t_0}\bigr)^8\tt^{-2}\|\p_X^4\Omega\|_{L^2(m)}\|\p_X^3\p_Y\Omega\|_{L^2(m)}\\
&+2\bigl(\ln\f{1+t}{1+t_0}\bigr)^8\|\p_X^3\p_Y\Omega\|_{L^2(m)}\|\p_X\p_Y^3\Omega\|_{L^2(m)}\\
&+C_m \bigl(1+{\nu}^{-\f32}{\|\Omega_0\|_{L^2(m)}} \bigr)\bigl(\ln \f{1+t}{1+t_0}\bigr)^{\f12} D(t)\\
&+C_{m,\sigma} \nu^{-\f32} \tt^{-1} \bigl(\ln \f{1+t}{1+t_0}\bigr)^{\f\sigma2+\f14} E(t)^\f12 D(t).
\end{aligned}\end{equation}
}\end{lem}

\begin{lem}\label{S3lem15}
{  Under the assumptions of Lemma \ref{S3lem9}, we have
\begin{equation}\label{eq:lem13}
\begin{aligned}
(1+t) &\f{d}{dt}\Bigl(\bigl(\ln\f{1+t}{1+t_0}\bigr)^9\|\p_X^3\Omega(t)\|_{L^2(m)}^2\Bigr) +2\bigl(\ln\f{1+t}{1+t_0}\bigr)^9\tt^{-2}\|\p_X^4\Omega\|_{L^2(m)}^2\\
&+2\bigl(\ln\f{1+t}{1+t_0}\bigr)^9\|\p_X^3\p_Y\Omega\|_{L^2(m)}^2
\leq 9\bigl(\ln\f{1+t}{1+t_0}\bigr)^8\|\p_X^3\Omega\|_{L^2(m)}^2\\
&+C_m\bigl(1+\nu^{-\f32}{\|\Omega_0\|_{L^2(m)}}\bigr)E(t)+C_{m,\sigma} \nu^{-\f32} \tt^{-1} \bigl(\ln\f{1+t}{1+t_0}\bigr)^{\f\sigma2+\f{1}4} E(t)^\f12 D(t).
\end{aligned}\end{equation}
}\end{lem}

Now we are in a position to complete the estimate for the energy functionals defined by \eqref{def:E} and \eqref{def:D}.

\begin{prop}\label{S3prop2}
{  Let $m>1$ and $0<\sigma<\f12$. Let $(c_1,\cdots,c_{10})$ satisfy \eqref{assumptions on c1-7}.
Then for $t_0\leq t\leq 2t_0,$ the energy functionals: $E(t)$ and $D(t),$ defined by \eqref{def:E} and \eqref{def:D} satisfy
\begin{equation}\label{eq5.54}
\begin{aligned}
(1+t)\f{d}{dt}E(t)+\f12D(t)\leq C_m &\bigl(1+\nu^{-\f32}{\|\Omega_0\|_{L^2(m)}}\bigr)\bigl( E(t)+\bigl(\ln\f{1+t}{1+t_0}\bigr)^\f12D(t)\bigr)\\
&+C_{m,\sigma} \nu^{-\f32} \tt^{-1} \bigl(\ln\f{1+t}{1+t_0}\bigr)^{\f\sigma2+\f{1}4} E(t)^\f12 D(t).
\end{aligned}
\end{equation}
}\end{prop}
\begin{proof} In view of \eqref{def:E} and \eqref{def:D}, we get,
by summarizing the estimates obtained in the previous lemmas in such a way that $\eqref{eq:lem1}+c_1\eqref{eq:lem2}+c_2\eqref{eq:lem3}+c_3\eqref{eq:lem4}+c_4\eqref{eq:lem5}+c_5\eqref{eq:lem6}+c_6\eqref{eq:lem7}+c_7\eqref{eq:lem8}+c_{8}\eqref{eq:lem11}+c_{9}\eqref{eq:lem12}+c_{10}\eqref{eq:lem13}$, that
\beq\label{S3eq6}
\begin{split}
(1+t)&\f{d}{dt}E(t)+\f34D(t)
\leq C_m \bigl(1+\nu^{-\f32}{\|\Omega_0\|_{L^2(m)}}\bigr)\bigl( E(t)+\bigl(\ln\f{1+t}{1+t_0}\bigr)^\f12D(t)\bigr)\\
&+C_{m,\sigma} \nu^{-\f32} \tt^{-1} \bigl(\ln\f{1+t}{1+t_0}\bigr)^{\f\sigma2+\f{1}4} E(t)^\f12 D(t)\\
&+2(c_1+c_2)\bigl(\ln\f{1+t}{1+t_0}\bigr)\|\p_X\Omega\|_{L^2(m)}\|\p_Y\Omega\|_{L^2(m)}\\
&+(2c_2+4c_4)\bigl(\ln\f{1+t}{1+t_0}\bigr)^2\bigl(\tt^{-2}\|\p_X^2\Omega\|_{L^2(m)}+\|\p_Y^2\Omega\|_{L^2(m)}\bigr)\|\p_X\p_Y\Omega\|_{L^2(m)}\\
&+(4c_5+5c_6)\bigl(\ln\f{1+t}{1+t_0}\bigr)^4\|\p_X^2\Omega\|_{L^2(m)}\|\p_X\p_Y\Omega\|_{L^2(m)}\\
&+2c_6\bigl(\ln\f{1+t}{1+t_0}\bigr)^5\bigl(\tt^{-2}\|\p_X^3\Omega\|_{L^2(m)}+\|\p_X\p_Y^2\Omega\|_{L^2(m)}\bigr)\|\p_X^2\p_Y\Omega\|_{L^2(m)}\\
&+(2c_{9}+8c_{10}) \bigl(\ln\f{1+t}{1+t_0}\bigr)^7\|\p_X^3\Omega\|_{L^2(m)}\|\p_X^2\p_Y\Omega\|_{L^2(m)}\\
&+2c_{9}\bigl(\ln\f{1+t}{1+t_0}\bigr)^8\bigl(\tt^{-2}\|\p_X^4\Omega\|_{L^2(m)}+\|\p_X\p_Y^3\Omega\|_{L^2(m)}\bigr)\|\p_X^3\p_Y\Omega\|_{L^2(m)},
\end{split}\eeq
where we used the facts:
$$
c_1\leq \f54, \, c_3\leq \f{c_2}{12}, \, c_4\leq \f{c_1}{4}, \, c_5\leq \f{c_3}{8}, \, c_7\leq \f{c_6}{24}, \, c_{8}\leq \f{c_7}{28}, \, c_{10}\leq \f{c_{9}}{36}.
$$

Next we handle the terms in \eqref{S3eq6} which do not contain  $E(t)$ or $D(t).$  If $c_1$ and $c_2$ satisfy moreover $64c_1^2\leq c_2$ and $64c_2\leq 1$, then one has
\begin{align*}
2(c_1+c_2)\bigl(\ln\f{1+t}{1+t_0}\bigr)\|\p_X\Omega\|_{L^2(m)}\|\p_Y\Omega\|_{L^2(m)}
\leq \f{1}4\|\p_Y\Omega\|_{L^2(m)}^2+\f{c_2}4\bigl(\ln{t}/{t_0}\bigr)^2\|\p_X\Omega\|_{L^2(m)}^2.
\end{align*}
In case $32(c_2+2c_4)^2\leq c_1c_3$, we have
\begin{align*}
&(2c_2+4c_4)\bigl(\ln\f{1+t}{1+t_0}\bigr)^2\bigl(\tt^{-2}\|\p_X^2\Omega\|_{L^2(m)}+\|\p_Y^2\Omega\|_{L^2(m)}\bigr)\|\p_X\p_Y\Omega\|_{L^2(m)}\\
 &\leq \f{c_3}{8}\bigl(\ln\f{1+t}{1+t_0}\bigr)^3\bigl( \tt^{-2}\|\p_X^2\Omega\|_{L^2(m)}^2+\|\p_X\p_Y\Omega\|_{L^2(m)}^2\bigr)\\
&\quad+\f{c_1}{4}\bigl(\ln\f{1+t}{1+t_0}\bigr)\bigl(\tt^{-2}\|\p_X\p_Y\Omega\|_{L^2(m)}^2+\|\p_Y^2\Omega\|_{L^2(m)}^2 \bigr).
\end{align*}
When $8(4c_5+5c_6)^2\leq c_3c_6$, there holds
\begin{align*}
&(4c_5+5c_6)\bigl(\ln\f{1+t}{1+t_0}\bigr)^4\|\p_X^2\Omega\|_{L^2(m)}\|\p_X\p_Y\Omega\|_{L^2(m)}\\
&\leq \f{c_6}4 \bigl(\ln\f{1+t}{1+t_0}\bigr)^5\|\p_X^2\Omega\|_{L^2(m)}^2+\f{c_3}{8}\bigl(\ln\f{1+t}{1+t_0}\bigr)^3\|\p_X\p_Y\Omega\|_{L^2(m)}^2.
\end{align*}
If $64c_6^2\leq c_5c_7$, one has
\begin{align*}
&2c_6\bigl(\ln\f{1+t}{1+t_0}\bigr)^5\bigl(\tt^{-2}\|\p_X^3\Omega\|_{L^2(m)}+\|\p_X\p_Y^2\Omega\|_{L^2(m)}\bigr)\|\p_X^2\p_Y\Omega\|_{L^2(m)}\\
& \leq \f{c_7}{8}\bigl(\ln\f{1+t}{1+t_0}\bigr)^6\bigl( \tt^{-2}\|\p_X^3\Omega\|_{L^2(m)}^2+\|\p_X^2\p_Y\Omega\|_{L^2(m)}^2\bigr)\\
&\quad+\f{c_5}{8}\bigl(\ln\f{1+t}{1+t_0}\bigr)^4\bigl(\tt^{-2}\|\p_X^2\p_Y\Omega\|_{L^2(m)}^2+\|\p_X\p_Y^2\Omega\|_{L^2(m)}^2 \bigr).
\end{align*}
Similarly, when $32(c_{8}+4c_{9})^2\leq c_7c_{11}$, there holds 
\begin{align*}
    &(2c_{8}+8c_{9}) \bigl(\ln\f{1+t}{1+t_0}\bigr)^7\|\p_X^3\Omega\|_{L^2(m)}\|\p_X^2\p_Y\Omega\|_{L^2(m)}\\
    &\leq \f{c_7}8 \bigl(\ln\f{1+t}{1+t_0}\bigr)^6\|\p_X^2\p_Y\Omega\|_{L^2(m)}^2 +\f{c_{9}}{4} \bigl(\ln\f{1+t}{1+t_0}\bigr)^8\|\p_X^3\Omega\|_{L^2(m)}^2.
\end{align*}
Also, if $64c_{9}^2\leq c_{8}c_{10}$, one has
\begin{align*}
    &2c_{9}\bigl(\ln\f{1+t}{1+t_0}\bigr)^8\bigl(\tt^{-2}\|\p_X^4\Omega\|_{L^2(m)}+\|\p_X\p_Y^3\Omega\|_{L^2(m)}\bigr)\|\p_X^3\p_Y\Omega\|_{L^2(m)}\\
    &\leq \f{c_{8}}4 \bigl(\ln\f{1+t}{1+t_0}\bigr)^7\bigl(\tt^{-2}\|\p_X^3\p_Y\Omega\|_{L^2(m)}^2+\|\p_X\p_Y^3\Omega\|_{L^2(m)}^2\bigr)\\
    &\quad+ \f{c_{10}}4 \bigl(\ln\f{1+t}{1+t_0}\bigr)^9\bigl(\tt^{-2}\|\p_X^4\Omega\|_{L^2(m)}^2+\|\p_X^3\p_Y\Omega\|_{L^2(m)}^2\bigr).
\end{align*}

Noticing that under the assumption \eqref{assumptions on c1-7}, all the previous assumptions on $
 c_1,\cdots,c_{10},$ are satisfied. Therefore, by substituting the above estimates into \eqref{S3eq6}, we conclude the
 proof of \eqref{eq5.54}.
\end{proof}

With Proposition \ref{S3prop2}, let us show the estimates for higher regularities:

\begin{prop}\label{S3prop3}
{  Let $m>6$ and $0<\sigma<\f12$. Suppose \eqref{eqs:Om} has a solution $\Omega(t)\in L^2(m_0)$. Then, for $|({\alpha_1},{\alpha_2})|\leq 2$ or $({\alpha_1},{\alpha_2})=(3,0)$, there holds
\begin{equation}\label{eq:prop3.3}
\begin{aligned}
\|\p_X^{\alpha_1}\p_Y^{\alpha_2}\Omega(t)\|_{L^2(m)}
\leq &C_{m,\sigma} \bigl( 1 +{\nu}^{-\f32}{\|\Omega_0\|_{L^2(m_0)}}\bigr)^{m+\f{2(3{\alpha_1}+{\alpha_2})(m+1)}{1+2\sigma}} \tt\|\Omega_0\|_{L^2(m)}.
\end{aligned}
\end{equation}
}\end{prop}
\begin{proof}
We first fix $t_0$ to be a large enough time and prove that \eqref{eq:prop3.3} holds near $t_0$. Let us define
\begin{equation}\label{def:T^*}
T^*\eqdefa \inf \bigl\{\ T>t_0 |\ \sup_{t\in [t_0,T]}E(t)<2E(t_0)  \ \bigr\}.
\end{equation}
Observing that $T^*>t_0$ is well-defined.
We define $T_1^*$ to be determined by
\begin{equation}\label{def:T_1^*}
C_m \bigl(1+\nu^{-\f32}{\|\Omega_0\|_{L^2(m)}}\bigr)\bigl(\ln\f{1+T_1^*}{1+t_0}\bigr)^\f12+C_{m,\sigma} \nu^{-\f32} (1+t_0)^{-1} \bigl(\ln\f{1+T_1^*}{1+t_0}\bigr)^{\f\sigma2+\f{1}4} \bigl(2E(t_0)\bigr)^\f12=\f12.
\end{equation}
Then, we get that for all $t_0\leq t\leq \min(T^*,T^*_1)$,
\begin{align*}
    &C_m \bigl(1+\nu^{-\f32}{\|\Omega_0\|_{L^2(m)}}\bigr)\bigl(\ln\f{1+t}{1+t_0}\bigr)^\f12D(t)\\
    &+C_{m,\sigma} \nu^{-\f32} (1+t)^{-1} \bigl(\ln\f{1+t}{1+t_0}\bigr)^{\f\sigma2+\f{1}4} \bigl(E(t)\bigr)^\f12 D(t)\leq \f12 D(t),
\end{align*}
from which and \eqref{eq5.18}, we infer
$$
(1+t)\f{d}{dt} E(t)\leq C_m \bigl(1+\nu^{-\f32}{\|\Omega_0\|_{L^2(m)}}\bigr)E(t), \qquad\forall \quad t\in [t_0,\min(T^*,T^*_1)].
$$
Applying Gronwall's inequality gives rise to
$$
E(t)\leq E(t_0)e^{C_m \bigl(1+\nu^{-\f32}{\|\Omega_0\|_{L^2(m)}}\bigr)\bigl(\ln \f{1+t}{1+t_0}\bigr)}, \qquad\forall \ t\in [t_0,\min(T^*,T^*_1)].
$$
We denote $T^*_2$ to be determined by
\begin{equation}\label{def:T_2^*}
C_m \bigl(1+\nu^{-\f32}{\|\Omega_0\|_{L^2(m)}}\bigr)\bigl(\ln \f{1+T_2^*}{1+t_0})=\ln \f32.
\end{equation}
As a consequence, we obtain
$$
E(t)\leq \f32 E(t_0), \qquad\forall \quad t\in [t_0,\min(T^*,T^*_1,T^*_2)].
$$
Comparing the above inequality with \eqref{def:T^*},  we get, by using a standard continuous argument, that $T^*>\min(T^*_1,T^*_2)$.

On the other hand, we deduce from  Proposition \ref{prop5.1} that,
$$
E(t_0)= \|\Omega(t_0)\|_{L^2(m)}^2\leq C_{m}(1+t_0)^2 \bigl( 1 +{\nu}^{-\f32}{\|\Omega_0\|_{L^2(m)}}\bigr)^{2m} \|\Omega_0\|_{L^2(m)}^2,
$$
from which we infer
$$
E(t)\leq  C_{m}(1+t_0)^2 \bigl( 1 +{\nu}^{-\f32}{\|\Omega_0\|_{L^2(m)}}\bigr)^{2m} \|\Omega_0\|_{L^2(m)}^2, \qquad \forall t\in [t_0,f_0(t_0)],
$$
where
\begin{align*}
f_0(t_0)\eqdefa (1+t_0) \exp \Bigl( \f1{C_{m,\sigma}} \bigl( 1 +{\nu}^{-\f32}{\|\Omega_0\|_{L^2(m)}}\bigr)^{-\f{4(m+1)}{1+2\sigma}} \Bigr)-1,
\end{align*}
so that for any $T\leq f_0(t_0),$ there holds
\begin{align*}
&C_m \bigl(1+\nu^{-\f32}{\|\Omega_0\|_{L^2(m)}}\bigr)\bigl(\ln\f{1+T}{1+t_0}\bigr)^\f12+C_{m,\sigma} \nu^{-\f32} (1+t_0)^{-1} \bigl(\ln\f{1+T}{1+t_0}\bigr)^{\f\sigma2+\f{1}4} \bigl(2E(t_0)\bigr)^\f12\leq \f12,\\
& C_m \bigl(1+\nu^{-\f32}{\|\Omega_0\|_{L^2(m)}}\bigr)\bigl(\ln \f{1+T}{1+t_0})\leq \ln \f32,
\end{align*}
which implies $f_0(t_0)\leq \min (T_1^*,T_2^*)$.

Now for any $t$ being large enough, we  choose some $t_0<t$ so that
\begin{equation}\label{eq3.55}
\ln\f{1+t}{1+t_0}={C_{m,\sigma}}^{-1}   \bigl( 1 +{\nu}^{-\f32}{\|\Omega_0\|_{L^2(m_0)}}\bigr)^{-\f{4(m+1)}{1+2\sigma}}.
\end{equation}
Due to the fact: $t_0<t\leq f_0(t_0)$,  we deduce that
$$
E(t)\leq C_{m} \bigl( 1 +{\nu}^{-1}{\|\Omega(1)\|_{L^2(m_0)}}\bigr)^{2m} (1+t_0)^{2}\|\Omega_0\|_{L^2(m_0)}^2,
$$
from which and \eqref{def:E},  we conclude the proof of \eqref{eq:prop3.3}.
\end{proof}

\begin{rmk}\label{Rmk: 4.1}
    { In Proposition \ref{S3prop3}, we only prove the estimates for ${\alpha_2}\leq 2$, because higher derivatives in $Y$ direction may cause additional difficulties from the boundary. However, the estimates we obtained are sufficient to prove the results in this paper.
    }
\end{rmk}

\section{Proof of the main theorem}\label{section 5}
In this section, we present the proof of  Theorem \ref{Thm2}.

\begin{proof}[Proof of Theorem \ref{Thm2}]
Let us denote
$$
\tilde{\Omega}\eqdef \Omega-\alpha\bar{\Omega},
$$
where $\alpha=\int_{\R^2_+}Y\Omega(X,Y)dXdY $ is the vertical momentum of $\Omega$, and $\bar{\Omega}$ is the kernel of $\cL$ obtained from Proposition \ref{prop3.1}. As we pointed out, $M_2(\Omega)$ is a conserved
quantity, so that $\alpha=M_2(\Omega(1)),$ which is independent of time. Then in view of \eqref{eqs:Om}  and $\cL \bar{\Omega}=0,$  we find
\beq \label{S4eq4}
(1+t)\p_t \tilde{\Omega}=\cL \tilde{\Omega}+(\cL_t-\cL)(\tilde{\Omega}+\alpha\bar{\Omega})+\cN_t (\tilde{\Omega}+\alpha \bar{\Omega}).
\eeq

Since we are investigating the long-time behavior of $\tilde{\Omega}(t),$ we focus on the time-span $[t_0,+\oo)$ beginning at some large time $t_0$. We write
\begin{equation}
\tilde{\Omega}(t)=e^{(\ln \f{1+t}{1+t_0})\cL}\tilde{\Omega}(t_0)+\int_{t_0}^t e^{\left(\ln \f{1+t}{1+s}\right)\cL_\oo} \bigl( (\cL_s-\cL) (\tilde{\Omega}(s)+\alpha \bar{\Omega})+\cN_s (\tilde{\Omega}(s)+\alpha \bar{\Omega}) \bigr)\,\f{ds}{1+s}.
\end{equation}

Due to $\int_{\R^2_+}Y\tilde{\Omega}(t,X,Y)dXdY=0$, we first derive from  \eqref{eq2.5a} and \eqref{eq5.1} that
\begin{equation}\label{eq6.3}
\begin{aligned}
&\bigl\|e^{(\ln \f{1+t}{1+t_0})\cL_\oo}\tilde{\Omega}(t_0)\bigr\|_{L^2(m)}
\leq C_{m} e^{-\ln \f{1+t}{1+t_0}}\bigl(\|{\Omega}(t_0)\|_{L^2(m)}+\alpha\|\bar{\Omega}\|_{L^2(m)}\bigr)\\
&\leq  C_{m} \bigl(\f{1+t_0}{1+t}\bigr)\Bigl(\bigl( 1 +\nu^{-\f32}{\|\Omega_0\|_{L^2(m_0)}}\bigr)^{m} (1+t_0)\|\Omega_0\|_{L^2(m)}+\alpha\Bigr)\\
&\leq  C_{m} (1+t)^{-1}(1+t_0)^2\bigl( 1 +{\nu}^{-\f32}{\|\Omega_0\|_{L^2(m)}}\bigr)^{m} \|\Omega_0\|_{L^2(m)}.
\end{aligned}\end{equation}

Since $\cL_t-\cL=(1+t)^{-2}\p_X^2$, we have 
$$
(\cL_s-\cL) (\tilde{\Omega}(s)+\alpha \bar{\Omega})=(1+t)^{-2}\p_X^2\Omega(s).
$$
Then we get, by using \eqref{eq2.5a} and \eqref{eq:prop3.3}, that 
\begin{equation}\notag
\begin{aligned}
\bigl\|\int_{t_0}^t& e^{\left(\ln \f{1+t}{1+s}\right)\cL} (\cL_s-\cL) \bigl(\tilde{\Omega}(s)+\alpha\bar{\Omega}\bigr)\,\f{ds}{1+s}\bigr\|_{L^2(m)}\\
&\leq  C_{m}\int_{t_0}^t e^{-\ln \f{1+t}{1+s}} (1+s)^{-2}  \|\p_X^2 {\Omega}(s)\|_{L^2(m)}\,\f{ds}{1+s}.
%\leq & C_{m,\sigma}\int_{t_0}^t (\f{1+t}{1+s})^{-1} (1+s)^{-2}
% \bigl( 1 +\nu^{-\f32}{\|\Omega_0\|_{L^2(m_0)}}\bigr)^{m+\f{12(m+1)}{1+2\sigma}} (1+s)\|\Omega_0\|_{L^2(m)}\f{ds}{1+s}\\
%\leq & C_{m,\sigma}(1+t)^{-1}\bigl( 1 +\nu^{-\f32}{\|\Omega_0\|_{L^2(m)}}\bigr)^{m+\f{12(m+1)}{1+2\sigma}}\|\Omega_0\|_{L^2(m)}{\color{red} \int_{t_0}^t  (1+s)^{-1}\, ds.}
\end{aligned}
\end{equation}
By the interpolation inequality $\|\p_X^2\Omega\|_{L^2(m)}\leq \|\p_X^3\Omega\|_{L^2(m)}^\f23\|\Omega\|_{L^2(m)}^\f13$ and
$$
    \|\Omega(s)\|_{L^2(m)}\leq \|\tilde{\Omega}(s)\|_{L^2(m)}+\alpha\|\bar{\Omega}\|_{L^2(m)}\leq (1+s)^{-\f12} \sup_{s\in[t_0,t]}\bigl((1+s)\|\tilde{\Omega}(s)\|_{L^2(m)}\bigr) +\|\Omega_0\|_{L^2(m)},
$$
we apply \eqref{eq:prop3.3} with $(\alpha_1,\alpha_2)=(3,0)$ to get
\begin{align*}
    \bigl\|\int_{t_0}^t& e^{\left(\ln \f{1+t}{1+s}\right)\cL} (\cL_s-\cL) \bigl(\tilde{\Omega}(s)+\alpha\bar{\Omega}\bigr)\,\f{ds}{1+s}\bigr\|_{L^2(m)}\\
    \leq & C_{m,\sigma} \int_{t_0}^t \f{1+s}{1+t} (1+s)^{-2} \Bigl( \bigl( 1 +{\nu}^{-\f32}{\|\Omega_0\|_{L^2(m_0)}}\bigr)^{m+\f{18(m+1)}{1+2\sigma}} (1+s)\|\Omega_0\|_{L^2(m)} \Bigr)^\f23\\
    &\qquad\qquad\times \Bigl((1+s)^{-\f12} \sup_{s\in[t_0,t]}\bigl((1+s)\|\tilde{\Omega}(s)\|_{L^2(m)}\bigr) +\|\Omega_0\|_{L^2(m)}\Bigr)^\f13 \f{ds}{1+s}\\
    \leq & C_{m,\sigma} (1+t)^{-1}   \bigl( 1 +{\nu}^{-\f32}{\|\Omega_0\|_{L^2(m_0)}}\bigr)^{\f23m+\f{12(m+1)}{1+2\sigma}}\|\Omega_0\|_{L^2(m)}^\f23  \\ &\qquad\qquad\times\Bigl(\sup_{s\in[t_0,t]}\bigl((1+s)\|\tilde{\Omega}(s)\|_{L^2(m)}\bigr)^\f13 \int_{t_0}^t (1+s)^{-\f32}ds+\|\Omega_0\|_{L^2(m)}^\f13 \int_{t_0}^t (1+s)^{-\f43}ds \Bigr) 
\end{align*}
Therefore, we use Young's inequality to achieve
\begin{equation}\label{eq6.4}
\begin{aligned}
\bigl\|\int_{t_0}^t& e^{\left(\ln \f{1+t}{1+s}\right)\cL} (\cL_s-\cL) \bigl(\tilde{\Omega}(s)+\alpha\bar{\Omega}\bigr)\,\f{ds}{1+s}\bigr\|_{L^2(m)}
\leq \f18 (1+t)^{-1}\sup_{s\in[t_0,t]}\bigl((1+s)\|\tilde{\Omega}(s)\|_{L^2(m)}\bigr)
\\
&\qquad\qquad +  C_{m,\sigma}(1+t)^{-1}\bigl( 1 +\nu^{-\f32}{\|\Omega_0\|_{L^2(m)}}\bigr)^{m+\f{18(m+1)}{1+2\sigma}} \|\Omega_0\|_{L^2(m)}.
\end{aligned}
\end{equation}

For the nonlinear terms, one can easily use integration by parts and $\Delta_t^{-1}\Omega|_{Y=0}=0$ to check that 
\begin{align*}
        &\int_{\R^2_+} Y \,\cN_t (\Omega)dXdY=\nu^{-\f32}(1+t)^{-\f52}\int_{\R^2_+} \p_X \Delta_t^{-1}\Omega\, \Omega dXdY \\
        &=-(1+t)^{-2}\int_{\R^2_+}\p_X^2 \Delta_t^{-1}\Omega \, \p_X\Delta_t^{-1}\Omega dXdY-\int_{\R^2_+}\p_X\p_Y \Delta_t^{-1}\Omega \, \p_Y\Delta_t^{-1}\Omega dXdY=0,
\end{align*}
which together with \eqref{eq2.5a} imply
\begin{align*}
\bigl\|\int_{t_0}^t& e^{\left(\ln \f{1+t}{1+s}\right)\cL} \cN_s (\Omega(s))\f{ds}{1+s}\bigr\|_{L^2(m)}
\leq C_{m}\int_{t_0}^t e^{-\ln \f{1+t}{1+s}} \|\cN_s (\Omega)\|_{L^2(m)}\f{ds}{1+s}\\
\leq & C_{m} \Bigl( \alpha^2 (1+t)^{-1}\int_{t_0}^t  \|\cN_s (\bar{\Omega})\|_{L^2(m)}ds \\
&+\alpha (1+t)^{-1}\nu^{-\f32}\int_{t_0}^t  \|\p_Y \Delta_s^{-1}\bar{\Omega}\p_X \tilde{\Omega}(s)- \p_X\Delta_s^{-1}\bar{\Omega}\p_Y \tilde{\Omega}(s)\|_{L^2(m)}(1+s)^{-\f52}ds\\
&+\alpha (1+t)^{-1}\nu^{-\f32}\int_{t_0}^t  \|\p_Y \Delta_s^{-1}\tilde{\Omega}(s)\p_X \bar{\Omega}- \p_X\Delta_s^{-1}\tilde{\Omega}(s)\p_Y\bar{\Omega}\|_{L^2(m)}(1+s)^{-\f52}ds\\
&+ (1+t)^{-1}\int_{t_0}^t  \|\cN_s ( \tilde{\Omega}(s))\|_{L^2(m)}ds\Bigr) \eqdefa \mathcal{M}_1+\mathcal{M}_2+\mathcal{M}_3+\mathcal{M}_4.
\end{align*}
Next, we estimate term by term above.

It follows from \eqref{eq5.14}, Corollary \ref{col3.2} and $\alpha\leq \|\Omega_0\|_{L^2(m)}$ that
\begin{equation}\notag
\begin{aligned}
\mathcal{M}_1 \leq &C_{m} \alpha^2 (1+t)^{-1} \nu^{-\f32} \int_{t_0}^t  \bigl(\|\p_Y \Delta_s^{-1}\bar{\Omega}\|_{L^\oo}\|\p_X \bar{\Omega}\|_{L^2(m)}\\
&\qquad\qquad\qquad\qquad\qquad\qquad+ \|\p_X\Delta_s^{-1}\bar{\Omega}\|_{L^\oo}\|\p_Y \bar{\Omega}\|_{L^2(m)}\bigr)(1+s)^{-\f52}\,ds\\
\leq & C_{m,\sigma_0}  \|\Omega_0\|_{L^2(m)}^2 (1+t)^{-1} \nu^{-\f32} \|(\bar{\Omega},\nabla \bar{\Omega})\|_{L^2(m)}^2\int_{t_0}^t  (1+s)^{-\f32+\sigma_0}\,ds  \\
\leq & C_{m,\sigma_0} (1+t)^{-1}  {\nu}^{-\f32} (1+t_0)^{\sigma_0-\f12} \|\Omega_0\|_{L^2(m)}^2,
\end{aligned}
\end{equation}
where $0<\sigma_0<\f12$ is a small constant. For simplicity, we just take $\sigma_0=\f14$ to conclude
\begin{equation}\label{eq6.5}
\mathcal{M}_1\leq C_{m}(1+t)^{-1} {\nu}^{-\f32}\|\Omega_0\|_{L^2(m)}^2.
\end{equation}

Similarly, we get, by using \eqref{eq5.14} and $\|\p_Yf\|_{L^2(m)}\lesssim \|f\|_{L^2(m)}^\f12\|\p_Y^2f\|_{L^2(m)}^\f12$ that
\begin{align*}
\mathcal{M}_2 \leq &C_{m} \alpha (1+t)^{-1} \nu^{-\f32} \int_{t_0}^t  \bigl(\|\p_Y \Delta_s^{-1}\bar{\Omega}\|_{L^\oo}\|\p_X \tilde{\Omega}(s)\|_{L^2(m)}\\
&\qquad\qquad\qquad\qquad\qquad\qquad+ \|\p_X\Delta_s^{-1}\bar{\Omega}\|_{L^\oo}\|\p_Y \tilde{\Omega}(s)\|_{L^2(m)}\bigr)(1+s)^{-\f52}ds\\
\leq &C_{m,\sigma_1}  (1+t)^{-1}\nu^{-\f32}\|\Omega_0\|_{L^2(m)} \int_{t_0}^t  \Bigl(\|\p_X \tilde{\Omega}(s)\|_{L^2(m)}\\
&\qquad\qquad\qquad\qquad\qquad+ (1+s)\|\tilde{\Omega}(s)\|_{L^2(m)}^\f12\|\p_Y^2\tilde{\Omega}(s)\|_{L^2(m)}\Bigr)(1+s)^{-\f52+\sigma_1}ds,
\end{align*}
where $0<\sigma_1<\f12$ is a small constant to be chosen. Then by applying \eqref{eq:prop3.3}and
\begin{equation}\label{eq6.6}
\| \tilde{\Omega}(s)\|_{L^2(m)} \leq (1+s)^{-1} \sup_{s\in[t_0,t]} \bigl((1+s)\|\tilde{\Omega}(s)\|_{L^2(m)}\bigr),
\end{equation}
 we find
\begin{equation}\notag
    \begin{aligned}
\mathcal{M}_2 \leq & C_{m,\sigma,\sigma_1} (1+t)^{-1}\nu^{-\f32}\|\Omega_0\|_{L^2(m)} \Bigl(  \bigl( 1 +\nu^{-\f32}\|\Omega_0\|_{L^2(m)}\bigr)^{m+\f{6(m+1)}{1+2\sigma}}\|\Omega_0\|_{L^2(m)}\\
&\qquad\times\int_{t_0}^t (1+s)^{-\f32+\sigma_1-1}ds +\bigl( 1 +\nu^{-\f32}\|\Omega_0\|_{L^2(m)}\bigr)^{\f{m}2+\f{2(m+1)}{1+2\sigma}}\|\Omega_0\|_{L^2(m)}^\f12\\
&\qquad\times \sup_{s\in [t_0,t]} \bigl((1+s)\|\tilde{\Omega}\|_{L^2(m)}\bigr)^\f12 \int_{t_0}^t (1+s)^{-\f32+\sigma_1}ds\Bigr)\\
\leq&C_{m,\sigma,\sigma_1} (1+t)^{-1} \Bigl( (1+t_0)^{-\f12+\sigma_1} \bigl( 1 +\nu^{-\f32}\|\Omega_0\|_{L^2(m)}\bigr)^{m+1+\f{6(m+1)}{1+2\sigma}}\|\Omega_0\|_{L^2(m)}\\
&\qquad +(1+t_0)^{-\f12+\sigma_1}\bigl( 1 +\nu^{-\f32}\|\Omega_0\|_{L^2(m)}\bigr)^{\f{m}2+1+\f{2(m+1)}{1+2\sigma}}\|\Omega_0\|_{L^2(m)}^\f12\\
&\qquad\qquad\times \sup_{s\in [t_0,t]} \bigl((1+s)\|\tilde{\Omega}\|_{L^2(m)}\bigr)^\f12 \Bigr).
\end{aligned}
\end{equation}
Therefore, we deduce that
\begin{equation}\label{eq6.7}
\begin{aligned}
\mathcal{M}_2 \leq &\f18(1+t)^{-1}\sup_{s\in[t_0,t]} \bigl((1+s)\|\tilde{\Omega}(s)\|_{L^2(m)}\bigr)+C_{m,\sigma,\sigma_1} (1+t)^{-1} \|\Omega_0\|_{L^2(m)} \\
&\times \Bigl(\bigl( 1 +\nu^{-\f32}\|\Omega_0\|_{L^2(m)}\bigr)^{m+1+\f{6(m+1)}{1+2\sigma}}(1+t_0)^{-\f12+\sigma_1}\\
 &\qquad+\bigl( 1 +\nu^{-\f32}\|\Omega_0\|_{L^2(m)}\bigr)^{m+2+\f{4(m+1)}{1+2\sigma}} (1+t_0)^{-1+2\sigma_1}\Bigr).
\end{aligned}
\end{equation}

For  $\mathcal{M}_3$, we get, by applying \eqref{eq5.14}, that
\begin{align*}
\mathcal{M}_3 \leq &C_{m} \alpha (1+t)^{-1} \nu^{-\f32} \int_{t_0}^t  \bigl(\|\p_Y \Delta_s^{-1}\tilde{\Omega}(s)\|_{L^\oo}\|\p_X \bar{\Omega}\|_{L^2(m)}\\
&\qquad\qquad\qquad\qquad\qquad\qquad+ \|\p_X\Delta_s^{-1}\tilde{\Omega}(s)\|_{L^\oo}\|\p_Y \bar{\Omega}\|_{L^2(m)}\bigr)(1+s)^{-\f52}ds\\
\leq &C_{m,\sigma_2}  (1+t)^{-1}\nu^{-\f32}\|\Omega_0\|_{L^2(m)} \int_{t_0}^t  \| \tilde{\Omega}(s)\|_{L^2(m)}^{\f12+\sigma_2}\|\p_X \tilde{\Omega}(s)\|_{L^2(m)}^{\f12-\sigma_2}(1+s)^{-\f32+\sigma_2}\,ds,
\end{align*}
from which, \eqref{eq6.6} and \eqref{eq:prop3.3}, we infer
\begin{align*}
\mathcal{M}_3 \leq  &C_{m,\sigma,\sigma_2}  (1+t)^{-1}  \bigl( 1 +\nu^{-\f32}\|\Omega_0\|_{L^2(m)}\bigr)^{1+\left(\f12-\sigma_2\right)\left(m+\f{6(m+1)}{1+2\sigma}\right)}\|\Omega_0\|_{L^2(m)}^{\f12-\sigma_2}\\
&\quad\times\sup_{s\in[t_0,t]} \bigl((1+s)\|\tilde{\Omega}(s)\|_{L^2(m)}\bigr)^{\f12+\sigma_2}\int_{t_0}^t  (1+s)^{-\f32-\sigma_2}\,ds\\
\leq  &C_{m,\sigma,\sigma_2}  (1+t)^{-1} (1+t_0)^{-\f12-\sigma_2}   \bigl( 1 +\nu^{-\f32}\|\Omega_0\|_{L^2(m)}\bigr)^{1+\left(\f12-\sigma_2\right)\left(m+\f{6(m+1)}{1+2\sigma}\right)}\\
&\quad\times\|\Omega_0\|_{L^2(m)}^{\f12-\sigma_2}\sup_{s\in[t_0,t]} \bigl((1+s)\|\tilde{\Omega}(s)\|_{L^2(m)}\bigr)^{\f12+\sigma_2},
\end{align*}
where $0<\sigma_2<\f12$ is some constant to be chosen.
By applying Young's inequality, we obtain
\begin{equation}\label{eq6.8}
\begin{aligned}
&\mathcal{M}_3 \leq \f{1}8 (1+t)^{-1}\sup_{s\in[t_0,t]} \bigl((1+s)\|\tilde{\Omega}(s)\|_{L^2(m)}\bigr) +C_{m,\sigma,\sigma_2} (1+t)^{-1}  \\
&\qquad\qquad \times(1+t_0)^{-\f2{1-2\sigma_2}+1} \bigl( 1 +\nu^{-\f32}\|\Omega_0\|_{L^2(m)}\bigr)^{\f{2}{1-2\sigma_2}+m+\f{6(m+1)}{1+2\sigma}}\|\Omega_0\|_{L^2(m_0)}.
\end{aligned}\end{equation}

Finally, we consider the fully nonlinear part $M_4$. By applying \eqref{eq5.14} with $0<\sigma_3<\f12$ and the interpolation inequality, we find
\begin{align*}
\mathcal{M}_4 \leq &C_{m}  (1+t)^{-1} \nu^{-\f32} \int_{t_0}^t  \bigl(\|\p_Y \Delta_s^{-1}\tilde{\Omega}(s)\|_{L^\oo}\|\p_X \tilde{\Omega}(s)\|_{L^2(m)}\\
&\qquad\qquad\qquad\qquad\qquad+ \|\p_X\Delta_s^{-1}\tilde{\Omega}(s)\|_{L^\oo}\|\p_Y \tilde{\Omega}(s)\|_{L^2(m)}\bigr)(1+s)^{-\f52}\,ds\\
\leq &C_{m,\sigma_3} (1+t)^{-1} \nu^{-\f32} \int_{t_0}^t  \| \tilde{\Omega}(s)\|_{L^2(m)}^{\f12+\sigma_3}\|\p_X \tilde{\Omega}(s)\|_{L^2(m)}^{\f12-\sigma_3} \\
&\qquad\qquad\qquad\times\bigl(\|\p_X \tilde{\Omega}(s)\|_{L^2(m)}+ (1+s)\|\p_Y \tilde{\Omega}(s)\|_{L^2(m)}\bigr)(1+s)^{-\f52+\sigma_3}\,ds\\
\leq &C_{m,\sigma_3}  (1+t)^{-1} \nu^{-\f32} \int_{t_0}^t  \| \tilde{\Omega}(s)\|_{L^2(m)}^{1+\sigma_3}\|\p_X \tilde{\Omega}(s)\|_{L^2(m)}^{\f12-\sigma_3} \\
&\qquad\qquad\qquad\times\bigl(\|\p_X^2 \tilde{\Omega}(s)\|_{L^2(m)}^\f12+ (1+s)\|\p_Y^2 \tilde{\Omega}(s)\|_{L^2(m)}^\f12\bigr)(1+s)^{-\f52+\sigma_3}\,ds.
\end{align*}
Then we  get, by applying \eqref{eq6.6} and \eqref{eq:prop3.3}, that
\begin{equation}\begin{aligned}\label{eq6.9}
\mathcal{M}_4 \leq &C_{m,\sigma,\sigma_3}  (1+t)^{-1} \sup_{s\in[t_0,t]} \bigl((1+s)\|\tilde{\Omega}(s)\|_{L^2(m)}\bigr)\\
&\times\Bigl( \bigl( 1 +\nu^{-\f32}\|\Omega_0\|_{L^2(m)}\bigr)^{m+1+(9-6\sigma_3)\f{m+1}{1+2\sigma}} \int_{t_0}^t(1+s)^{-\f52+\sigma_3} \,ds\\
&\qquad+\bigl( 1 +\nu^{-\f32}\|\Omega_0\|_{L^2(m)}\bigr)^{m+1+(5-6\sigma_3)\f{m+1}{1+2\sigma}} \int_{t_0}^t (1+s)^{-\f32+\sigma_3} \,ds\Bigr)\\
\leq &C_{m,\sigma,\sigma_3}  (1+t)^{-1} \sup_{s\in[t_0,t]} ((1+s)\|\tilde{\Omega}(s)\|_{L^2(m)})\\
&\times\Bigl( \bigl( 1 +\nu^{-\f32}\|\Omega_0\|_{L^2(m)}\bigr)^{m+1+(9-6\sigma_3)\f{m+1}{1+2\sigma}}  (1+t_0)^{-\f32+\sigma_3} \\
&\qquad+\bigl( 1 +\nu^{-\f32}\|\Omega_0\|_{L^2(m)}\bigr)^{m+1+(5-6\sigma_3)\f{m+1}{1+2\sigma}} (1+t_0)^{-\f12+\sigma_3} \Bigr).
\end{aligned}
\end{equation}
To control \eqref{eq6.9}, we need to use some smallness from the negative power of the large time $t_0$. To do so, we take a large time $T_1$ such that,
$$
C_{m,\sigma,\sigma_3} \bigl( 1 +\nu^{-\f32}\|\Omega_0\|_{L^2(m)}\bigr)^{m+1+(9-6\sigma_3)\f{m+1}{1+2\sigma}}  (1+T_1)^{-\f32+\sigma_3}\leq \f1{16},
$$
and
$$
C_{m,\sigma,\sigma_3} \bigl( 1 +\nu^{-\f32}\|\Omega_0\|_{L^2(m)}\bigr)^{m+1+(5-6\sigma_3)\f{m+1}{1+2\sigma}}  (1+T_1)^{-\f12+\sigma_3}\leq \f1{16},
$$
which results in
\begin{align*}
&T_1\geq C_{m,\sigma,\sigma_3} \bigl( 1 +\nu^{-\f32}\|\Omega_0\|_{L^2(m)}\bigr)^{(m+1)\max\Bigl(\f{1+\f{9-6\sigma_3}{1+2\sigma}}{\f32-\sigma_3},\f{1+\f{5-6\sigma_3}{1+2\sigma}}{\f12-\sigma_3}\Bigr)}-1.
\end{align*}
Noticing that as $\sigma\rightarrow \f12$ and $\sigma_3 \rightarrow0$, it comes out
$$
(m+1)\max\Bigl(\f{1+\f{9-6\sigma_3}{1+2\sigma}}{\f32-\sigma_3},\f{1+\f{5-6\sigma_3}{1+2\sigma}}{\f12-\sigma_3}\Bigr) \rightarrow {7(m+1)},
$$
so that we can take $\sigma_3$ small enough and $\sigma$ close enough to $\f12$ and choose
\begin{equation}\label{def:T_1}
T_1\eqdefa C_{m,\epsilon} \bigl( 1 +\nu^{-\f32}\|\Omega_0\|_{L^2(m)}\bigr)^{{7(m+1)}+\epsilon}-1,
\end{equation}
where $\epsilon$ can be taken arbitrarily small. We conclude that for $t_0\geq T_1$, one has
\begin{equation}\label{eq6.11}
\mathcal{M}_4 \leq \f{(1+t)^{-1}}8 \sup_{s\in[t_0,t]} \bigl((1+s)\|\tilde{\Omega}(s)\|_{L^2(m)}\bigr).
\end{equation}

Therefore by summarizing the estimates \eqref{eq6.3}, \eqref{eq6.4}, \eqref{eq6.5}, \eqref{eq6.7}, \eqref{eq6.8} and \eqref{eq6.11},
 we deduce that for $t_0\geq T_1$
\begin{align*}
(1+t)&\|\tilde{\Omega}(t)\|_{L^2(m)}
\leq \f12\sup_{s\in[t_0,t]} \bigl((1+s)\|\tilde{\Omega}(s)\|_{L^2(m)}\bigr) \\
&+C_{m} (1+t_0)^{2}\bigl( 1 +\nu^{-\f32}\|\Omega_0\|_{L^2(m)}\bigr)^{m} \|\Omega_0\|_{L^2(m)}\\
&+C_{m,\sigma} \bigl( 1 +\nu^{-\f32}\|\Omega_0\|_{L^2(m)}\bigr)^{m+\f{18(m+1)}{1+2\sigma}}\|\Omega_0\|_{L^2(m)}\\
&+C_{m,\sigma,\sigma_1}  \Bigl((1+t_0)^{-\f12+\sigma_1}\bigl( 1 +\nu^{-\f32}\|\Omega_0\|_{L^2(m)}\bigr)^{m+1+\f{6(m+1)}{1+2\sigma}}\\
& \quad +(1+t_0)^{-1+2\sigma_1}\bigl( 1 +\nu^{-\f32}\|\Omega_0\|_{L^2(m)}\bigr)^{m+2+\f{4(m+1)}{1+2\sigma}}\Bigr)\|\Omega_0\|_{L^2(m)}\\
&+C_{m,\sigma,\sigma_2} (1+t_0)^{-\f2{1-2\sigma_2}+1} \bigl( 1 +\nu^{-\f32}\|\Omega_0\|_{L^2(m)}\bigr)^{\f{2}{1-2\sigma_2}+m+\f{6(m+1)}{1+2\sigma}}\|\Omega_0\|_{L^2(m)}.
\end{align*}
Then we get, by choosing $t_0=T_1$ with $T_1$ being defined in \eqref{def:T_1} and taking supremum over $t\in [T_1, T],$ that
\begin{align*}
&\sup_{t\in[T_1,T]} \bigl((1+t)\|\tilde{\Omega}(t)\|_{L^2(m)}\bigr) \leq C_{m,\sigma,\sigma_1,\sigma_2} \Bigl( \bigl( 1 +\nu^{-\f32}\|\Omega_0\|_{L^2(m)}\bigr)^{m+2\bigl({7(m+1)}+\epsilon\bigr)} \\
&\quad+\bigl( 1 +\nu^{-\f32}\|\Omega_0\|_{L^2(m)}\bigr)^{m+\f{18(m+1)}{1+2\sigma}} \\
&\quad+\bigl( 1 +\nu^{-\f32}\|\Omega_0\|_{L^2(m)}\bigr)^{m+1+\f{6(m+1)}{1+2\sigma}-\bigl(\f12-\sigma_1\bigr)\bigl({7(m+1)}+\epsilon\bigr)}\\
&\quad+\bigl( 1 +\nu^{-\f32}\|\Omega_0\|_{L^2(m)}\bigr)^{m+2+\f{4(m+1)}{1+2\sigma}-\bigl(1-2\sigma_1\bigr)\bigl({7(m+1)}+\epsilon\bigr)}\\
&\quad+\bigl( 1 +\nu^{-\f32}\|\Omega_0\|_{L^2(m)}\bigr)^{\f{2}{1-2\sigma_2}+m+\f{6(m+1)}{1+2\sigma}-\bigl(\f2{1-2\sigma_2}-1\bigr)\bigl({7(m+1)}+\epsilon\bigr)}\Bigr)\|\Omega_0\|_{L^2(m)}.
\end{align*}
To simplify these complicated powers, one can focus on the limit case that when $\sigma\rightarrow\f12$, $\sigma_1\rightarrow0$, $\sigma_2\rightarrow0$, and $\epsilon\rightarrow0$, the largest power is
$15m+14$,
which comes from the first term, reflecting the evolution of initial data from $T_0$.  Therefore, for any small $\epsilon>0$, we can take $( \sigma_1,\sigma_2)$ to be small enough and $\sigma$ to be close enough to $\f12$ so that
\begin{equation}\label{eq4.21}
\sup_{t\in[T_1,T]} \bigl((1+t)\|\tilde{\Omega}(t)\|_{L^2(m)}) \leq C_{m,\epsilon} \bigl(  1 +\nu^{-\f32}\|\Omega_0\|_{L^2(m)}\bigr)^{15m+14+\e}\|\Omega_0\|_{L^2(m)}.
\end{equation}
 This completes the proof of Theorem \ref{Thm2}.
\end{proof}

\appendix
\section{The proof of Lemma \ref{lem2.3}}\label{appendix A}
In this appendix, we present the proof of Lemma \ref{lem2.3}.

Before the proof, let us recall the asymptotic behaviors of Airy function $Ai(z)$ as $|z|$ goes to infinity. For $|\arg z|<(1-\varepsilon)\pi$ with any $\varepsilon\in (0,1)$, we have 
\begin{equation}\label{eqA.1}
    Ai(z)=\f{z^{-\f14}}{2\pi^\f12 }e^{-\f23z^\f32}\bigl(1+O(|z|^{-\f32})\bigr)
    \andf Ai'(z)=-\f{z^\f14}{2\pi^\f12 }e^{-\f23z^\f32}\bigl(1+O(|z|^{-\f32})\bigr).
\end{equation}
For $|\arg z|<\f23\pi$, we also have 
\begin{equation}\label{eqA.2}
    Ai(-z)=\f{z^{-\f14}}{\pi^\f12 }\sin (\f23z^\f32+\f\pi4)\bigl(1+O(|z|^{-\f32})\bigr)
    \andf Ai'(-z)=-\f{z^\f14}{\pi^\f12 }\cos (\f23z^\f32+\f\pi4)\bigl(1+O(|z|^{-\f32})\bigr).
\end{equation}

\begin{proof}[Proof of \eqref{eq2.22}]
    Let us first consider $\| \f{Ai'\bigl(e^{i(\f{\pi}6+\delta)}\lambda+e^{i\f\pi6}y\bigr)}{Ai\bigl(e^{i(\f{\pi}6+\delta)}\lambda\bigr)} \|_{L^\oo_y(\R_+)} $. 

    Using \eqref{eqA.1}, we introduce a universal large constant $M$ such that for $|z|>M$ and $|\arg z|\leq\f89\pi$, 
    $$
        \f{|z|^{-\f14}}{4\pi^\f12 }|e^{-\f23z^\f32}|\leq |Ai(z)|\leq \f{|z|^{-\f14}}{\pi^\f12 }|e^{-\f23z^\f32}|
        \andf \f{|z|^{\f14}}{4\pi^\f12 }|e^{-\f23z^\f32}|\leq |Ai'(z)|\leq \f{|z|^{\f14}}{\pi^\f12 }|e^{-\f23z^\f32}|.
    $$

    If $1\leq \lambda\leq M$ and $0\leq y\leq M$, one can easily estimate the term as
    \begin{equation}\label{eqA.3}
        \begin{aligned}
            | \f{Ai'\bigl(e^{i(\f{\pi}6+\delta)}\lambda+e^{i\f\pi6}y\bigr)}{Ai\bigl(e^{i(\f{\pi}6+\delta)}\lambda\bigr)}| \leq \|Ai'(z)\|_{L^\oo(|z|<2M)} \|Ai(z)\|^{-1}_{L^\oo(|z|<2M,|\arg z|<\f89\pi)}\leq C_M.
        \end{aligned}
    \end{equation}
    If $1\leq \lambda\leq M$ and $y>M$, we can use the asymptotical behavior to get
    \begin{equation}\label{eqA.4}
        \begin{aligned}
            | \f{Ai'\bigl(e^{i(\f{\pi}6+\delta)}\lambda+e^{i\f\pi6}y\bigr)}{Ai\bigl(e^{i(\f{\pi}6+\delta)}\lambda\bigr)}| \leq C y^\f14 e^{-c y^\f32} \|Ai(z)\|^{-1}_{L^\oo(|z|<2M,|\arg z|<\f89\pi)}\leq C_M.
        \end{aligned}
    \end{equation}

    If $\lambda>M$, the asymptotical behavior implies
    \begin{align*}
        | \f{Ai'\bigl(e^{i(\f{\pi}6+\delta)}\lambda+e^{i\f\pi6}y\bigr)}{Ai\bigl(e^{i(\f{\pi}6+\delta)}\lambda\bigr)}| &\leq C |e^{i\delta}\lambda+y|^\f14 \lambda^{\f14} |\exp \Bigl( -\f23 (e^{i(\f{\pi}6+\delta)}\lambda+e^{i\f\pi6}y)^\f32+\f23 e^{i(\f\pi4+\f32\delta)} \lambda^\f32 \Bigr)|\\
        &\leq C(\lambda+y)^\f14 \lambda^\f14 e^{-c(\lambda+y)^\f12y}
        \leq C\lambda^\f12,
    \end{align*}
    where we use the fact that 
    \begin{align*}
        \Re \bigl(-\f23 (e^{i(\f{\pi}6+\delta)}\lambda+e^{i\f\pi6}y)^\f32+\f23 e^{i(\f\pi4+\f32\delta)} \lambda^\f32 \bigr)&=-\Re\int_0^y  (e^{i(\f{\pi}6+\delta)}\lambda+e^{i\f\pi6}z)^\f12 e^{i\f\pi6}dz 
        \leq -c(\lambda+y)^\f12 y.
    \end{align*}

    Combining the above cases, we conclude that  $\| \f{Ai'\bigl(e^{i(\f{\pi}6+\delta)}\lambda+e^{i\f\pi6}y\bigr)}{Ai\bigl(e^{i(\f{\pi}6+\delta)}\lambda\bigr)} \|_{L^\oo_y(\R_+)}\leq C \lambda^\f12 $.
    
    Now, we turn to estimate $\| \f{Ai'\bigl(e^{-i(\f{5\pi}6+\delta)}\lambda+e^{i\f\pi6}y\bigr)}{Ai\bigl(e^{-i(\f{5\pi}6+\delta)}\lambda\bigr)} \|_{L^\oo_y(\R_+)}.$ The region $\lambda<M$ follows the same estimate as \eqref{eqA.3} and \eqref{eqA.4}. However, when $\lambda>M$, we have to discuss whether $\arg \bigl(e^{-i(\f{5\pi}6+\delta)}\lambda+e^{i\f\pi6}y\bigr)\approx \pi$ to apply different asymptotic behaviors. 
    Recall that $\delta<\f\pi{100}$, we introduce two special values $y_1=\f{\sin (\f\pi{18}-\delta)}{\sin (\f\pi{18})}\lambda$ and $y_2=\f{\sin (\f{5}{18}\pi+\delta)}{\sin (\f{5}{18}\pi+2\delta) }\lambda$, which are chosen to make $\arg \bigl(e^{-i(\f{5\pi}6+\delta)}\lambda+e^{i\f\pi6}y_1\bigr)=-\f89\pi$ and $\arg \bigl(e^{-i(\f{5\pi}6+\delta)}\lambda+e^{i\f\pi6}y_2\bigr)=\f89\pi$.

    When $\lambda>M$ and $0\leq y<y_1$, we observe $-\f56\pi<\arg\bigl(e^{-i(\f{5\pi}6+\delta)}\lambda+e^{i\f\pi6}y\bigr)<-\f89\pi$ and use \eqref{eqA.1} to find
    \begin{align*}
        &|\f{Ai'\bigl(e^{-i(\f{5\pi}6+\delta)}\lambda+e^{i\f\pi6}y\bigr)}{Ai\bigl(e^{-i(\f{5\pi}6+\delta)}\lambda\bigr)}| \\
        \leq &C |e^{-i(\f{5\pi}6+\delta)}\lambda+e^{i\f\pi6}y|^\f14 \lambda^\f14 |\exp \Bigl(-\f23 (e^{-i(\f{5\pi}6+\delta)}\lambda+e^{i\f\pi6}y)^\f32+\f23 e^{-i(\f54\pi+\f32\delta)}\lambda^\f32\Bigr)| \\
        \leq &C(y_1+\lambda)^\f14\lambda^\f14\leq C \lambda^\f12.
    \end{align*}
    where we used that for $y<y_1$,
    \begin{align*}
        &\Re \bigl(-\f23 (e^{-i(\f{5\pi}6+\delta)}\lambda+e^{i\f\pi6}y)^\f32+\f23 e^{-i(\f54\pi+\f32\delta)}\lambda^\f32\bigr) =-\Re \int_0^y (e^{-i(\f{5\pi}6+\delta)}\lambda+e^{i\f\pi6}z)^\f12 e^{i\f\pi6}dz\leq0. 
    \end{align*}

    When $\lambda>M$ and $y_1\leq y\leq y_2$, there hold $\arg\bigl(-e^{-i(\f{5\pi}6+\delta)}\lambda-e^{i\f\pi6}y\bigr)\leq \f\pi9$ and $\f{\sin (\delta)}{\sin (\f5{18}\pi+2\delta)}\lambda\leq |e^{-i(\f{5\pi}6+\delta)}\lambda+e^{i\f\pi6}y|\leq \f{\sin (\delta)}{\sin (\f\pi{18})}\lambda$, which together with \eqref{eqA.2} implies
    $$
        |Ai'\bigl(e^{-i(\f{5\pi}6+\delta)}\lambda+e^{i\f\pi6}y\bigr)| \leq C |e^{-i(\f{5\pi}6+\delta)}\lambda+e^{i\f\pi6}y|^\f14 e^{|e^{-i(\f{5\pi}6+\delta)}\lambda+e^{i\f\pi6}y|^\f32} 
        \leq C \lambda^\f14 e^{(\f{\sin (\delta)}{\sin (\f\pi{18})})^\f32\lambda^\f32}.
    $$
    Therefore, we have 
    \begin{align*}
        &|\f{Ai'\bigl(e^{-i(\f{5\pi}6+\delta)}\lambda+e^{i\f\pi6}y\bigr)}{Ai\bigl(e^{-i(\f{5\pi}6+\delta)}\lambda\bigr)}| 
        \leq C \lambda^\f12 |\exp \Bigl(\f23 e^{-i(\f{5}4\pi+\f32\delta)}\lambda^\f32+(\f{\sin (\delta)}{\sin (\f\pi{18})})^\f32\lambda^\f32\Bigr)| \leq C ,
    \end{align*}
    where we used from the smallness of $\delta$ that 
    $$
        \Re \bigl( \f23 e^{-i(\f{5}4\pi+\f32\delta)}\lambda^\f32+(\f{\sin (\delta)}{\sin (\f\pi{18})})^\f32\lambda^\f32\bigr)\leq -c\lambda^\f32.
    $$

    When $\lambda>M$ and $y> y_2$, we observe $\f\pi6<\arg\bigl(e^{-i(\f{5\pi}6+\delta)}\lambda+e^{i\f\pi6}y\bigr)<\f89\pi$ and use \eqref{eqA.1} to find
    \begin{align*}
        &|\f{Ai'\bigl(e^{-i(\f{5\pi}6+\delta)}\lambda+e^{i\f\pi6}y\bigr)}{Ai\bigl(e^{-i(\f{5\pi}6+\delta)}\lambda\bigr)}| \\
        \leq &C |e^{-i(\f{5\pi}6+\delta)}\lambda+e^{i\f\pi6}y|^\f14 \lambda^\f14 |\exp \Bigl(-\f23 (e^{-i(\f{5\pi}6+\delta)}\lambda+e^{i\f\pi6}y)^\f32+\f23 e^{-i(\f54\pi+\f32\delta)}\lambda^\f32\Bigr)| \\
        \leq &C(y+\lambda)^\f14\lambda^\f14 e^{-c(y^\f32+\lambda^\f32)}\leq C .
    \end{align*}
    where we used that for $y<y_1$, there exists a small enough constant $c>0$ depending on $\delta$ such that 
    \begin{align*}
        &\Re \bigl(-\f23 (e^{-i(\f{5\pi}6+\delta)}\lambda+e^{i\f\pi6}y)^\f32+\f23 e^{-i(\f54\pi+\f32\delta)}\lambda^\f32\bigr) \\
        =&\Re\Bigl(-\f23 (e^{-i(\f{5\pi}6+\delta)}\lambda+e^{i\f\pi6}y_2)^\f32+\f23 e^{-i(\f54\pi+\f32\delta)}\lambda^\f32- \int_{y_2}^y (e^{-i(\f{5\pi}6+\delta)}\lambda+e^{i\f\pi6}z)^\f12 e^{i\f\pi6}dz\Bigr)\\
        \leq & -\f23 \bigl(\cos (\f\pi4-\f32\delta)+(\f{\sin (\delta)}{\sin (\f5{18}\pi+2\delta)})^\f32\bigr)\lambda^\f32 -c (y-y_2)^\f32\\
        \leq & -c(y^\f32+\lambda^\f32).
    \end{align*}

    Combining the above discussions, we arrive at $\| \f{Ai'\bigl(e^{-i(\f{5\pi}6+\delta)}\lambda+e^{i\f\pi6}y\bigr)}{Ai\bigl(e^{-i(\f{5\pi}6+\delta)}\lambda\bigr)} \|_{L^\oo_y(\R_+)}\leq C \lambda^\f12,$ which finishes the proof of \eqref{eq2.22}.
\end{proof}

\begin{proof}[Proof of \eqref{eq2.23}]
    By a change of variable, it remains to prove the following case with $k=1$:
    \begin{equation}\label{eqA.5}
        \begin{aligned}
            &\|Ai\bigl(e^{-i(\f{\pi}6+\delta)}\lambda\bigr) Ai\bigl(e^{-i(\f{5\pi}6+\delta)}\lambda+e^{i\f{\pi}6}y\bigr) \|_{L^2_y(\R_+)}\\
        &\qquad\qquad+\|Ai\bigl(e^{i(\f{5\pi}6+\delta)}\lambda\bigr) Ai\bigl(e^{i(\f{\pi}6+\delta)}\lambda+e^{i\f{\pi}6}y\bigr) \|_{L^2_y(\R_+)} \leq C  \lambda^{-\f34}.
        \end{aligned}
    \end{equation}
    Similarly as in the proof of \eqref{eq2.22}, the second term in \eqref{eqA.5} is easier to handle, since $e^{i(\f{\pi}6+\delta)}\lambda+e^{i\f{\pi}6}y$ never touches the negative real line $\R_{-}$. Therefore, we only do the estimate of $\|Ai\bigl(e^{-i(\f{\pi}6+\delta)}\lambda\bigr) Ai\bigl(e^{-i(\f{5\pi}6+\delta)}\lambda+e^{i\f{\pi}6}y\bigr) \|_{L^2_y(\R_+)}$.

    Again, we denote a large enough constant $M>0$ so that \eqref{eqA.1} and \eqref{eqA.2} make sense for $|z|>cM$.

    For $1\leq \lambda\leq M$, we only need to show the $L^2_y$ norm is bounded. When we also have $y\leq 2M$, one easily have
    $$
    \|Ai\bigl(e^{-i(\f{\pi}6+\delta)}\lambda\bigr) Ai\bigl(e^{-i(\f{5\pi}6+\delta)}\lambda+e^{i\f{\pi}6}y\bigr) \|_{L^2_y([0,2M])}
    \leq C M^\f12 \|Ai(z)\|_{L^\oo(|z|\leq 2M)}^2 \leq C_M.
    $$
    When $1\leq\lambda\leq M$ and $y\geq 2M$, we deduce from \eqref{eqA.1} that 
    \begin{align*}
        &\|Ai\bigl(e^{-i(\f{\pi}6+\delta)}\lambda\bigr) Ai\bigl(e^{-i(\f{5\pi}6+\delta)}\lambda+e^{i\f{\pi}6}y\bigr) \|_{L^2_y([2M,\oo))}\\
        &\qquad\qquad\leq C \|Ai(z)\|_{L^\oo(|z|\leq M)} \||y|^{-\f14} e^{-cy^\f32}\|_{L^2_y([2M,\oo))}
        \leq C_M.
    \end{align*}

    For large $\lambda>M$, we again discuss the region of $y$. We recall the notations $y_1$ and $y_2$ introduced in the proof of \eqref{eq2.22}.

     When $\lambda>M$ and $0\leq y<y_1$, we observe $-\f56\pi<\arg\bigl(e^{-i(\f{5\pi}6+\delta)}\lambda+e^{i\f\pi6}y\bigr)<-\f89\pi$, $\f{\sin (\delta)}{\sin (\f\pi{18})}\lambda\leq |e^{-i(\f{5\pi}6+\delta)}\lambda+e^{i\f\pi6}y|\leq \lambda$ and use \eqref{eqA.1} to find
    \begin{align*}
        &|Ai\bigl(e^{-i(\f{\pi}6+\delta)}\lambda\bigr) Ai\bigl(e^{-i(\f{5\pi}6+\delta)}\lambda+e^{i\f{\pi}6}y\bigr)| \\
        \leq &C \lambda^{-\f12} |\exp \Bigl(-\f23 (e^{-i(\f{5\pi}6+\delta)}\lambda+e^{i\f\pi6}y)^\f32+\f23 e^{-i(\f54\pi+\f32\delta)}\lambda^\f32\Bigr)| \\
        \leq &C\lambda^{-\f12}e^{-c\lambda^\f12 y}.
    \end{align*}
    where we used that for $y<y_1$,
    \begin{align*}
        \Re \bigl(-\f23 (e^{-i(\f{5\pi}6+\delta)}\lambda+e^{i\f\pi6}y)^\f32+\f23 e^{-i(\f54\pi+\f32\delta)}\lambda^\f32\bigr) &=-\Re \int_0^y (e^{-i(\f{5\pi}6+\delta)}\lambda+e^{i\f\pi6}z)^\f12 e^{i\f\pi6}dz\\
        &\leq -c \lambda^\f12y. 
    \end{align*}
    Therefore, the $L^2_y$ norm on this region satisfies
    \begin{equation}\label{eqA.6}
        \begin{aligned}
            \|Ai\bigl(e^{-i(\f{\pi}6+\delta)}\lambda\bigr) Ai\bigl(e^{-i(\f{5\pi}6+\delta)}\lambda+e^{i\f{\pi}6}y\bigr)\|_{L^2_y([0,y_1])}
        \leq C \lambda^{-\f12} \|e^{-c\lambda^\f12 y}\|_{L^2_y(\R^2_+)}\leq C \lambda^{-\f34}.
        \end{aligned}
    \end{equation}

    When $\lambda>M$ and $y_1\leq y\leq y_2$, there hold $\arg\bigl(-e^{-i(\f{5\pi}6+\delta)}\lambda-e^{i\f\pi6}y\bigr)\leq \f\pi9$ and $\f{\sin (\delta)}{\sin (\f5{18}\pi+2\delta)}\lambda\leq |e^{-i(\f{5\pi}6+\delta)}\lambda+e^{i\f\pi6}y|\leq \f{\sin (\delta)}{\sin (\f\pi{18})}\lambda$, which together with \eqref{eqA.2} implies
    $$
        |Ai\bigl(e^{-i(\f{5\pi}6+\delta)}\lambda+e^{i\f\pi6}y\bigr)| \leq C |e^{-i(\f{5\pi}6+\delta)}\lambda+e^{i\f\pi6}y|^{-\f14} e^{|e^{-i(\f{5\pi}6+\delta)}\lambda+e^{i\f\pi6}y|^\f32} 
        \leq C \lambda^{-\f14} e^{(\f{\sin (\delta)}{\sin (\f\pi{18})})^\f32\lambda^\f32}.
    $$
    Therefore, we have 
    \begin{align*}
        |Ai\bigl(e^{-i(\f{\pi}6+\delta)}\lambda\bigr) Ai\bigl(e^{-i(\f{5\pi}6+\delta)}\lambda+e^{i\f{\pi}6}y\bigr)| 
        &\leq C \lambda^{-\f12} |\exp \Bigl(\f23 e^{-i(\f{5}4\pi+\f32\delta)}\lambda^\f32+(\f{\sin (\delta)}{\sin (\f\pi{18})})^\f32\lambda^\f32\Bigr)| \\
        &\leq C \lambda^{-\f12} e^{-c\lambda^\f32} ,
    \end{align*}
    where we used from the smallness of $\delta$ that 
    $$
        \Re \bigl( \f23 e^{-i(\f{5}4\pi+\f32\delta)}\lambda^\f32+(\f{\sin (\delta)}{\sin (\f\pi{18})})^\f32\lambda^\f32\bigr)\leq -c\lambda^\f32.
    $$
    Noticing that $y_2-y_1\leq C\lambda$, we arrive at
    \begin{align*}
        \|Ai\bigl(e^{-i(\f{\pi}6+\delta)}\lambda\bigr) Ai\bigl(e^{-i(\f{5\pi}6+\delta)}\lambda+e^{i\f{\pi}6}y\bigr)\|_{L^2_y([y_1,y_2])}\leq C\lambda^{-\f12} e^{-c\lambda^\f32} |y_2-y_1|^\f12\leq Ce^{-c\lambda^\f32}.
    \end{align*}

    When $\lambda>M$ and $y> y_2$, we observe $\f\pi6<\arg\bigl(e^{-i(\f{5\pi}6+\delta)}\lambda+e^{i\f\pi6}y\bigr)<\f89\pi$, $|e^{-i(\f{5\pi}6+\delta)}\lambda+e^{i\f\pi6}y|>|e^{-i(\f{5\pi}6+\delta)}\lambda+e^{i\f\pi6}y_2|\geq c\lambda$ and use \eqref{eqA.1} to find
    \begin{align*}
        &|Ai\bigl(e^{-i(\f{\pi}6+\delta)}\lambda\bigr) Ai\bigl(e^{-i(\f{5\pi}6+\delta)}\lambda+e^{i\f{\pi}6}y\bigr)| \\
        \leq &C |e^{-i(\f{5\pi}6+\delta)}\lambda+e^{i\f\pi6}y|^{-\f14} \lambda^{-\f14} |\exp \Bigl(-\f23 (e^{-i(\f{5\pi}6+\delta)}\lambda+e^{i\f\pi6}y)^\f32+\f23 e^{-i(\f54\pi+\f32\delta)}\lambda^\f32\Bigr)| \\
        \leq &C\lambda^{-\f12} e^{-c(y^\f32+\lambda^\f32)} .
    \end{align*}
    where we used that for $y<y_1$, there exists a small enough constant $c>0$ depending on $\delta$ such that 
    \begin{align*}
        &\Re \bigl(-\f23 (e^{-i(\f{5\pi}6+\delta)}\lambda+e^{i\f\pi6}y)^\f32+\f23 e^{-i(\f54\pi+\f32\delta)}\lambda^\f32\bigr) \\
        =&\Re\Bigl(-\f23 (e^{-i(\f{5\pi}6+\delta)}\lambda+e^{i\f\pi6}y_2)^\f32+\f23 e^{-i(\f54\pi+\f32\delta)}\lambda^\f32- \int_{y_2}^y (e^{-i(\f{5\pi}6+\delta)}\lambda+e^{i\f\pi6}z)^\f12 e^{i\f\pi6}dz\Bigr)\\
        \leq & -\f23 \bigl(\cos (\f\pi4-\f32\delta)+(\f{\sin (\delta)}{\sin (\f5{18}\pi+2\delta)})^\f32\bigr)\lambda^\f32 -c (y-y_2)^\f32\\
        \leq & -c(y^\f32+\lambda^\f32).
    \end{align*}
    Since $e^{-cy^\f32}$ is integrable, we get
    \begin{align*}
        \|Ai\bigl(e^{-i(\f{\pi}6+\delta)}\lambda\bigr) Ai\bigl(e^{-i(\f{5\pi}6+\delta)}\lambda+e^{i\f{\pi}6}y\bigr)\|_{L^2_y([y_2,\oo))}\leq C\lambda^{-\f12} e^{-c\lambda^\f32}\|e^{-cy^\f32}\|_{L^2_y}\leq Ce^{-c\lambda^\f32}.
    \end{align*}

    Combining all the estimates above, we conclude that $$
    \|Ai\bigl(e^{-i(\f{\pi}6+\delta)}\lambda\bigr) Ai\bigl(e^{-i(\f{5\pi}6+\delta)}\lambda+e^{i\f{\pi}6}y\bigr)\|_{L^2_y(\R_+)}\leq C\lambda^{-\f34},
    $$
    where the right hand side $\lambda^{-\f34}$ comes from the largest part \eqref{eqA.6}.This finishes the proof.
\end{proof}

\begin{proof}[Proof of \eqref{eq2.24} and \eqref{eq2.25}]
    The proof of these two inequalities are similar to the proof of \eqref{eq2.22} and \eqref{eq2.23}, but we only need to discuss whether $y$ is large.

    When $0\leq y\leq M$, there holds
    \begin{align*}
        &|\f{Ai'\bigl(e^{i\f{\pi}6}(\cos \delta\, \lambda+i\sin\delta)+e^{i\f\pi6}y\bigr)}{Ai\bigl(e^{i\f{\pi}6}(\cos \delta\, \lambda+i\sin\delta)\bigr)}|\\
        &\qquad+|Ai\bigl(e^{i\f{5\pi}6}(\cos \delta\, \lambda+i\sin\delta)\bigr) Ai\bigl(e^{i\f{\pi}6}(\cos \delta\, \lambda+i\sin\delta+y)\bigr)| \leq C_M.
    \end{align*}

    When $y>M$, we can use \eqref{eqA.1} to find 
    \begin{align*}
        &|\f{Ai'\bigl(e^{i\f{\pi}6}(\cos \delta\, \lambda+i\sin\delta)+e^{i\f\pi6}y\bigr)}{Ai\bigl(e^{i\f{\pi}6}(\cos \delta\, \lambda+i\sin\delta)\bigr)}|\\
        &\qquad+|Ai\bigl(e^{i\f{5\pi}6}(\cos \delta\, \lambda+i\sin\delta)\bigr) Ai\bigl(e^{i\f{\pi}6}(\cos \delta\, \lambda+i\sin\delta+y)\bigr)| \leq C y^{-\f14} e^{-cy^\f32}.
    \end{align*}

    Combining the above two point-wise estimates, we conclude that 
    \begin{align*}
        &\|\f{Ai'\bigl(e^{i\f{\pi}6}(\cos \delta\, \lambda+i\sin\delta)+e^{i\f\pi6}y\bigr)}{Ai\bigl(e^{i\f{\pi}6}(\cos \delta\, \lambda+i\sin\delta)\bigr)}\|_{L^\oo_y}\\
        &\qquad+\|Ai\bigl(e^{i\f{5\pi}6}(\cos \delta\, \lambda+i\sin\delta)\bigr) Ai\bigl(e^{i\f{\pi}6}(\cos \delta\, \lambda+i\sin\delta+y)\bigr) \|_{L^2_y}\leq C,
    \end{align*}
    which implies \eqref{eq2.24}, and \eqref{eq2.25} follows from a further change of variables. This finishes the proof.
\end{proof}

\section{The Laplace transform of $H_\kappa$}

In this appendix, we provide some basic properties about the Laplace transform of $H_\kappa$ which are used in Subsection \ref{subsection3.3}.

\begin{lem}
    {  Let $H_\kappa(\tau)=\tau^\kappa e^{-\f1{12}\tau^3}$ be defined on $\R_+$ with $\kappa>-1$. Then, there holds that for all $|\arg\lambda|<\f23\pi-\delta$,
    \begin{equation}\label{eq:lem A}
        \Bigl|\ccL[H_{\kappa}](\lambda)- \sum_{m=0}^n \f1{m!}\f{(-1)^m}{12^{m}} \Gamma(\kappa+3m+1) \lambda^{-\kappa-3m-1} \Bigr|\leq C_{\delta,\kappa,n} |\lambda|^{-\kappa-3n-4}.
    \end{equation}
    }
\end{lem}
\begin{proof}
    Due to conjugation symmetry, we only prove the expansion for $\text{Im} \lam>0$.
    First, we extend the definition of $H_\kappa$ to be a complex function $H_\kappa(z)$ on $\arg z\in (-\pi,\pi)$. Noticing that $H_\kappa(z)\eqdefa z^\kappa e^{-\f1{12}z^3}$ is analytic and $|e^{-\f1{12}z^3}|$ remains bounded in the region $\arg z \in [-\f\pi6,\f\pi6]$, we can use Cauchy's theorem to get that for any $\arg\lambda\in [0,\f23\pi)$,
    $$
        \ccL[H_\kappa](\lambda)=\int_0^\oo H_\kappa(\tau)e^{-\lambda \tau}d\tau=\int_{L_{-\f\pi6}} H_\kappa(z) e^{-\lambda z}dz,
    $$
    where $L_{-\f\pi6}=\{\arg z=-\f\pi6\}$. Given the variable $z=e^{-i\f{\pi}6}x$ for the line $L_{-\f\pi6}$, we have
    $$
        \ccL[H_\kappa](\lambda)=\int_0^\oo e^{-i(\kappa+1)\f\pi6} x^{\kappa}e^{\f{i}{12}x^3} e^{-e^{-i\f{\pi}6}\lambda x}dx.
    $$
    Then, by Taylor expansion, we have that for all $x>0$,
    $$
        e^{\f{i}{12}x^3}=\sum_{m=0}^n \f1{m!}(\f{i}{12})^m x^{3m} +R_n(x) \with |R_n(x)|\leq C x^{3n+3}.
    $$
    For the polynomials, by changing $z=e^{-i\f{\pi}6}\lambda x$, we use again the Cauchy's theorem and the notation of $\Gamma$-functions to compute
    \begin{align*}
        &\f1{m!}(\f{i}{12})^{m}\int_0^\oo e^{-i(\kappa+1)\f\pi6} x^{\kappa+3m}e^{-e^{-i\f{\pi}6}\lambda x}dx \\
        =& \f1{m!}(\f{i}{12})^{m} \int_{L_{\arg \lambda-\f\pi 6}}  e^{-i(\kappa+1)\f\pi6} (\f{e^{i\f\pi6}}{\lambda})^{\kappa+3m+1} z^{\kappa+3m}e^{-z}dz\\
        =& \f1{m!}\f{(-1)^m}{12^{m}} \lambda^{-\kappa-3m-1}\int_0^\oo  z^{\kappa+3m}e^{-z}dz\\
        =& \f1{m!}\f{(-1)^m}{12^{m}} \Gamma(1+\kappa+3m) \lambda^{-\kappa-3m-1},
    \end{align*}
    where we used again $\arg \lambda\in[0,\f23\pi)$ so that $L_{\arg \lambda-\f\pi 6}$ belongs to $\{ \text{Re} z>0 \}$. For the remainder term, we need $\arg (e^{-i\f\pi6}\lambda)\in[-\f\pi6,\f\pi2-\delta)$ so that $\text{Re}(e^{-i\f\pi6}\lambda) \geq C_\delta|\lambda| $ to estimate
    \begin{align*}
        &|\int_0^\oo e^{-i(\kappa+1)\f\pi6} x^{\kappa}R_n(x) e^{-e^{-i\f{\pi}6}\lambda x}dx| \\
        \leq & C \int_0^\oo x^{\kappa+3n+3}e^{-\text{Re}(e^{-i\f\pi6}\lambda)x}dx\\
        =& C \text{Re}(e^{-i\f\pi6}\lambda)^{-\kappa-3n-4}\int_0^\oo  x^{\kappa+3n+3}e^{-x}dx\\
        \leq& C_{\delta,\kappa,n}  |\lambda|^{-\kappa-3n-4}.
    \end{align*}

Combining the above estimates, we arrive at \eqref{eq:lem A}, which finishes the proof.
\end{proof}

\noindent{\bf Conflict of interest:}  We confirm that we do not have any conflict of interest. 

\noindent{\bf Data availability:} The manuscript has no associated data. 

\noindent{\bf Author Contributions Statement:} N.L., N.M., and W.Z. have the same contribution to the manuscript.

\bibliographystyle{siam.bst} 
\bibliography{references.bib}

\end{document}